\documentclass[11pt,a4paper,reqno]{amsart}
\usepackage{amsmath,amssymb,amsfonts,epsfig,mathrsfs}
\usepackage[T1]{fontenc}
\usepackage{color}
\usepackage{array}
\usepackage{amsthm}
\usepackage{amstext}
\usepackage{graphicx}
\usepackage{setspace}
\usepackage[margin=2.5cm]{geometry}
\usepackage{color}
\usepackage{enumitem}
\usepackage{undertilde}
\usepackage{amscd,psfrag}
\usepackage{yhmath}
\usepackage[mathscr]{eucal}

\usepackage{slashed}

\makeatletter
\pdfpageheight\paperheight
\pdfpagewidth\paperwidth

\usepackage{epstopdf}
\usepackage{chngcntr}
\counterwithin{figure}{section}
\usepackage{mathrsfs}

\usepackage{indentfirst}	

\usepackage[normalem]{ulem}
\theoremstyle{plain}

\newtheorem{theorem}{Theorem}[section]
\newtheorem{remark}{Remark}[section]
\newtheorem*{theorem*}{Theorem}
\newtheorem{corollary}{Corollary}[section]
\newtheorem*{corollary*}{Corollary}
\newtheorem{lemma}{Lemma}[section]
\newtheorem*{claim*}{Claim}
\newtheorem*{q*}{Question}

\newcommand{\p}{\partial}
\newcommand{\real}{\mathbb{R}}
\newcommand{\R}{\mathbb{R}}
\newcommand{\vf}{\Gamma(TM)}
\newcommand{\na}{\nabla}
\newcommand{\rmm}{\mathbb{R}^m}
\newcommand{\rnk}{\mathbb{R}^{n+k}}
\newcommand{\ptens}{\bigwedge^q T^*M}
\newcommand{\pform}{\Omega^q(M)}
\newcommand{\vfn}{\Gamma(TM^\perp)}
\newcommand{\nap}{\nabla^\perp}
\newcommand{\bep}{B^\epsilon}
\newcommand{\nep}{\nabla^{\perp,\epsilon}}
\newcommand{\harmp}{\text{\rm Har}^q(M)}
\newcommand{\lie}{\mathcal{L}}
\newcommand{\vb}{V^{(B)}}
\newcommand{\vnab}{V^{(\nabla^\perp)}}
\newcommand{\ob}{\Omega^{(B)}}
\newcommand{\onab}{\Omega^{(\nabla^\perp)}}
\newcommand{\vbe}{V^{(B^\epsilon)}}
\newcommand{\vnabe}{V^{(\nabla^{\perp,\epsilon})}}
\newcommand{\obe}{\Omega^{(B^\epsilon)}}
\newcommand{\onabe}{\Omega^{(\nabla^{\perp,\epsilon})}}
\newcommand{\dis}{\mathcal{D}'}

\newcommand{\woneploc}{W^{1,p}_{\rm loc}}
\newcommand{\wtwoploc}{W^{2,p}_{\rm loc}}
\newcommand{\lploc}{L^p_{\rm loc}}
\newcommand{\hil}{\mathcal{H}}
\newcommand{\sdag}{S^\dagger}
\newcommand{\tdag}{T^\dagger}

\newcommand{\st}{S\oplus T}
\newcommand{\stdag}{\sdag\vee\tdag}
\newcommand{\useq}{\{u^\epsilon\}}
\newcommand{\vseq}{\{v^\epsilon\}}
\newcommand{\ub}{\bar{u}}
\newcommand{\vbb}{\bar{v}}
\newcommand{\ran}{\text{\rm ran}}

\newcommand{\antil}{\tilde{a}^\epsilon}

\newcommand{\bntil}{\tilde{b}^\epsilon}

\newcommand{\yz}{Y\bigoplus Z}
\newcommand{\yzstar}{Y^*\bigoplus Z^*}
\newcommand{\sd}{\slashed{\Delta}}

\newcommand{\htt}{\utilde{\hil}}

\newcommand{\e}{\epsilon}
\newcommand{\dd}{{\rm d}}

\numberwithin{equation}{section}
\numberwithin{figure}{section}

\title{Global Weak Rigidity of the Gauss-Codazzi-Ricci\\ Equations and Isometric Immersions
\\of Riemannian Manifolds with Lower Regularity}
\author{Gui-Qiang G. Chen}
\address{Gui-Qiang G. Chen: Mathematical Institute,\
 University of Oxford, Oxford, OX2 6GG, UK}
\email{\texttt{chengq@maths.ox.ac.uk}}
\author{Siran Li}
\address{Siran Li: Mathematical Institute, University of Oxford, Oxford, OX2 6GG, UK}
\email{\texttt{siran.li@rice.edu}}

\date{\today}

\begin{document}

\maketitle

\begin{abstract}
We are concerned with the global weak rigidity of
the Gauss-Codazzi-Ricci (GCR) equations
on Riemannian manifolds and the corresponding isometric
immersions of Riemannian manifolds into
the Euclidean spaces.
We develop a unified intrinsic approach
to establish the global weak rigidity of both the
GCR equations
and isometric
immersions of the Riemannian manifolds,
independent of the local coordinates,
and provide further insights of the previous local
results and arguments.
The critical case has also been analyzed.
To achieve this, we first reformulate the GCR equations
with div-curl structure intrinsically
on Riemannian manifolds and
develop a global, intrinsic version
of the div-curl lemma and other nonlinear techniques to tackle
the global weak rigidity on manifolds.
In particular, a general functional-analytic compensated compactness
theorem on Banach spaces has been established,
which includes the
intrinsic div-curl lemma on Riemannian manifolds
as a special case.
The equivalence of global isometric immersions, the Cartan formalism,
and the GCR equations
on the Riemannian manifolds with lower regularity
is established.
We also prove a new weak rigidity result along the way,
pertaining to the Cartan formalism, for Riemannian manifolds
with lower regularity, and extend the weak rigidity results
for Riemannian manifolds with corresponding different metrics.
\end{abstract}

\section{Introduction}
We are concerned with the global weak rigidity of
the Gauss-Codazzi-Ricci (GCR) equations
on Riemannian manifolds and the corresponding global weak rigidity of isometric
immersions of the Riemannian manifolds into the Euclidean spaces.
The problem of isometric immersions
of Riemannian manifolds into the Euclidean spaces
has been of considerable interest in the development of differential geometry,
which has also led to important developments
of new ideas and methods in nonlinear analysis
and partial differential equations (PDEs)
({\it cf.} \cite{HanHon06,Nas54,Nas56,Nir53,Yau}).
On the other hand,
the GCR equations are a fundamental system of nonlinear PDEs
in differential geometry
({\it cf.} \cite{BGY,BS,Eisenhart,Goenner,Greene,NM,Spi79}).
In particular, the GCR equations serve as the compatibility
conditions to ensure the existence of
isometric immersions.
Therefore, it is important to
understand the global, intrinsic behavior
of this nonlinear system on Riemannian manifolds
for solving the isometric immersion problems
and other important geometric problems,
including the global weak rigidity
of the GCR equations and isometric immersions.
In general, the Gauss-Codazzi-Ricci
system has no type, neither purely hyperbolic nor purely elliptic.
	
The weak rigidity problem for isometric immersions is to decide,
for a given sequence of isometric immersions of an $n$-dimensional manifold with a $W^{1,p}$ metric
whose  second fundamental forms and normal connections are uniformly bounded
in $L^p_{\rm loc}, p>n$,
whether its weak limit is still an isometric immersion with
the same $W^{1,p}_{\rm loc}$ metric.
This rigidity problem has its motivation both from geometric analysis
and nonlinear elasticity:
The existence of isometric immersions of Riemannian manifolds with lower regularity
corresponds naturally to the realization of elastic bodies with lower regularity
in the physical space.
See Ciarlet-Gratie-Mardare \cite{Cia08}, Mardare \cite{Mar03}, Szopos \cite{Szo08},
and the references cited therein.
	
The local weak rigidity of the GCR equations with lower regularity
has been analyzed in Chen-Slemrod-Wang \cite{CSW10-CMP,CSW10},
in which the div-curl structure of the GCR equations in local coordinates
has first been observed so that compensated compactness ideas,
especially the div-curl lemma ({\it cf.} Murat \cite{Mur78} and Tartar \cite{Tar79}),
can be employed in the local coordinates.
One of the advantages of these techniques is their independence
of the type of PDEs -- hence independent of the sign of curvatures
in the setting of isometric immersions of the Riemannian manifolds.
The key results in \cite{CSW10-CMP,CSW10}
are the weak rigidity of solutions to
the GCR equations,
which is known to be equivalent to the existence of local isometric immersions
in the $C^\infty$ category as a classical result ({\it cf.} \cite{HanHon06,Ten71}).
However,
such an equivalence for Riemannian manifolds with
lower regularity ({\it i.e.}, $W^{1,p}_{\rm loc}$ metric) is not a direct consequence
of the aforementioned classical results.
This problem has been treated recently in \cite{Cia08,Mar03,Mar05,Mar07}
from the point of view of nonlinear elasticity.
	
In this paper, we analyze the global weak rigidity of both the GCR equations
and isometric immersions -- via a new approach, independent of local coordinates.
Instead of writing the GCR equations in the local coordinates,
we formulate the GCR equations {\it intrinsically}
on Riemannian manifolds
and develop a global, intrinsic version
of compensated compactness and other nonlinear techniques to tackle
the global weak rigidity on manifolds.
Our aim is to develop a unified intrinsic approach
to establish the global weak rigidity of the GCR equations
and isometric immersions,
and provide further insights of the results and arguments
in
\cite{CSW10,Mar03,Mar05,Mar07,Szo08}
and the references cited therein.
We also establish a new weak rigidity result along the way,
pertaining to the Cartan formalism, for Riemannian
manifolds with lower regularity,
and extend the weak rigidity results for
Riemannian manifolds with corresponding different (unfixed) metrics.

This paper is organized as follows:	
In \S 2, we start with some geometric notations and present
some basic facts about isometric immersions and the GCR equations on Riemannian
manifolds for subsequent developments.
In \S 3, we first formulate and prove a general abstract
compensated compactness theorem on Banach spaces
in the framework of functional analysis.
As a direct corollary, we obtain
a global, intrinsic version of the div-curl lemma,
which has been also further extended to a more general version.
These serve as a basic tool for subsequent sections.
In \S 4, we give a geometric proof for the global weak rigidity of
the GCR equations on Riemannian manifolds.
The formulation and proof in this section are independent of
the local coordinates of the manifolds, and are based on
the geometric div-curl structure of the GCR equations.
In \S 5, the equivalence of global isometric immersions,
the Cartan formalism, and the GCR equations for
simply-connected $n$-dimensional manifolds
with $W^{1,p}_{\rm loc}$ metric, $p>n$, is established.
Then, in \S 6, we analyze the weak rigidity for the
critical case $n=2$ and $p=2$.
In particular, we show the weak rigidity of the GCR equations
when the co-dimension is $1$.
Finally, in \S 7, we first provide a proof of the weak rigidity
of the Cartan formalism, which gives an alternative
intrinsic proof of the main result in \S 4,
and then extend the weak rigidity results to the more general case
such that the underlying metrics of manifolds are allowed
to be a strongly convergent sequence in $W^{1, p}, p>n$.
To keep the paper self-contained, in the appendix, we provide a proof
for a general version of the intrinsic div-curl lemma,
Theorem \ref{thm_general goemetric divcurl},
on Riemannian manifolds.

\section{Geometric Notations, Isometric Immersions and the GCR Equations}\label{Section 2}

In this section, we start with some geometric notations about manifolds and vector bundles
for self-containedness,
and then present
some basic facts about isometric immersions
and the GCR equations on Riemannian manifolds for subsequent developments.

\allowdisplaybreaks			
\subsection{Notations: Manifolds and Vector Bundles}

Throughout this paper, we denote $(M,g)$ as an $n$-dimensional Riemannian manifold.
By definition, $M$ is a second-countable, Hausdorff topological space
with an atlas of local
charts $\mathcal{A}=\{(U_\alpha, \phi_\alpha)\,: \,\alpha\in\mathcal{I}\}$
such that each $U_\alpha \subset M$ is an open subset,
$\phi_\alpha : U_\alpha \mapsto
\phi_\alpha(U_\alpha) \subset \real^n$
is a homeomorphism, and the transition functions $\phi_\alpha \circ \phi_\beta^{-1}$
between the overlapping charts $U_\alpha$ and $U_\beta$ have required regularity.

If each $\phi_\alpha \circ \phi_\beta^{-1}$
can be chosen to have positive definite Jacobian determinant almost everywhere,
then $M$ is said to be orientable.
Moreover, $M$ is simply-connected if its fundamental group is trivial,
{\it i.e.}, each loop on $M$ can be continuously deformed to a point.
The Riemannian metric $g$ on $M$ is given by a field of inner products.
That is, at each point $P\in M$, $g(P): \vf\times\vf \mapsto\real$
is an inner product
denoted by
$$
g(P)(X,Y)=\langle X, Y\rangle \qquad\,\, \mbox{for any}\,\,\ X,Y\in \vf,
$$
where $g$ varies
with respect to $P$,
$\vf$ denotes the space of
vector fields on $M$
with required regularity,
and we have suppressed the dependence on $P$ when using
$\langle \cdot, \cdot \rangle$ to denote the inner product
associated with the metric.
Throughout this paper, we always denote by $\Gamma$ the space of sections of
given vector bundles with required Sobolev regularity,
which needs not be smooth or analytic, in order to be more suitable
for the applications to the realization problem ({\em cf.} \S 5).

Given $(M,g)$, an affine connection on $M$ (more precisely,
on $TM$) is a
bilinear map
$\na: \vf \times \vf \mapsto \vf$ such that,
for any $f: M\mapsto \R$ with required regularity, and any $X,Y,Z\in\vf$,
$$
\na_{fX} Y = f\na_X Y, \qquad \na_X (fY) = f\na_XY + X(f)Y,
$$
where $X(f)$ is the directional derivative of $f$ in the direction of $X$;
that is, vector $X$ is identified as the corresponding first-order differential operator.

We say that $\na$ is compatible with metric $g$ if
$$
Xg(Y,Z)=g(\na_XY, Z) + g(Y,\na_XZ),
$$
and $\na$ is torsion-free if
$$
\na_XY-\na_YX =[X,Y],
$$
where the Lie bracket is defined by $[X,Y]=XY-YX$.
There exists a unique compatible, torsion-free affine connection $\na$ on $M$,
known as the Levi-Civita connection, where the bilinear
map $\na$ is also called the covariant derivative. 	
As a basic example, consider $(\bar{M},\bar{g})=(\rmm, g_0)$ with the Euclidean metric $g_0=\delta_{ij}$,
{\it i.e.}, the dot product, whose Levi-Civita connection $\bar{\na}$ is given by $\bar{\na}_XY := XY$.

Given a manifold $M$, we say that $(E,M,F)$ is a vector bundle of degree $k\in\mathbb{N}$ over $M$
if there is a
surjection $\pi : E \mapsto M$ such that,
for any $P \in M$, there exists a local neighbourhood $U \subset M$ containing $P$
so that there is a
diffeomorphism $\psi_U:\pi^{-1}(U) \mapsto U \times F$
with ${\rm pr}_1\circ\psi_U = \pi$ on $\pi^{-1}(U)$,
where map ${\rm pr}_1$ is the projection map onto the first coordinate,
$E$ and $F$ are also differentiable manifolds, and $F \simeq \real^k$ ((vector space isomorphism).
In this bundle, $E$ is called the total space, $F$ is the fibre, $M$ is the base manifold,
$\pi$ is the projection of the bundle, and $\psi_U$ is termed as a {\it local trivialization}.
For simplicity, we also say that {\it $E$ is a vector bundle over $M$}.
	
If $E_1$ and $E_2$ are both vector bundles over $M$, we can take the direct sum and the quotient
of the bundles, by taking the vector space direct sum and quotient of the fibres.
Also, for a vector bundle $(E,M,F=\real^k)$ with projection $\pi$,
the space of smooth sections is defined
by $\Gamma(E):=\{s\in C^\infty(M;E)\,: \,\pi\circ s={\rm id}_M\}$.
We can define the affine connection $\na^E: \vf \times \Gamma(E) \mapsto \Gamma(E)$,
by linearity and the Leibniz rule.
For our purpose, $\na^E$ is required to restrict to the Levi-Civita connection
on $M$.
	
As a primary example, consider $E=TM$, the tangent bundle over $M$.
Then $\pi$ is the projection onto the base point in $M$, and $\Gamma(TM)$ is the space of smooth
vector fields, agreeing with the previous notation.
Moreover, $\na^{TM}$ is precisely the Levi-Civita connection.
Another example is the cotangent bundle $T^*M$ over $M$,
whose fibres are the dual vector spaces of the fibres of $TM$.
	
Our next example is crucial to this paper.
Consider $E_1=T^*M\otimes T^*M\otimes \cdots \otimes T^*M$, the tensor product
of $q$--copies of $T^*M$, for $q=0,1,2,\ldots$.
This is the (covariant) $q$--tensor algebra over $M$, which
can be viewed as $q$-linear maps on $TM$.
Now let $E_2\subset E_1$ be the subspace of all the alternating $q$--tensors on $TM$, {\it i.e.},
the $q$-linear maps that change sign when switching any pair of its
indices $\{i,j\}\subset\{1,\ldots,q\}$.
By convention, we write $E_2=:\ptens$, known as the alternating $q$--algebra.
Moreover, the sections are $\Omega^q(M):=\Gamma(\ptens)$,
known as the differential $q$--forms.
A general element $\alpha \in \pform$ is written as a linear combination of alternating
forms $\xi_1 \wedge \cdots \wedge \xi_q$,
where $\{\xi_j\}_{j=1}^q$ is a $q$-tuple of linearly independent differential
$1$--form on $M$.
If $\dim(M)=n$,
$\pform = \{0\}$ for $q \geq n+1$.
Thus, we always restrict to $0 \leq q \leq n$.
	
On the space of differential forms, we recall four important operations.
The first is the exterior derivative $d:\pform \mapsto \Omega^{q+1}(M)$,
which is linear and satisfies $d^2=0$.
The second is the Hodge star $\ast: \ptens \mapsto \bigwedge^{n-q}T^*M$,
which can also be regarded as $\ast: \pform \mapsto \Omega^{n-q}(M)$.
It is an isomorphism of vector bundles, which satisfies $\ast\ast=(-1)^{q(n-q)}$
whenever $M$ is orientable.
The third is a natural {\it product} on differential forms:
For $\alpha \in \pform$ and $\beta \in \Omega^r(M)$, we can define the
wedge product $\alpha \wedge \beta \in \Omega^{q+r}(M)$.
The fourth is the covariant derivative $\na: \Omega^q(M) \mapsto \Omega^{q+1}(M)$:
For $\alpha \in\Omega^q(M)$, define $\nabla\alpha \in \Omega^{q+1}(M)$ via the Leibniz rule:
\begin{align}\label{eqn_first equation}
	&(\nabla\alpha)(X, Y_1, \ldots, Y_q) \equiv \na_X\alpha (Y_1, \ldots, Y_q) \nonumber\\
	&\quad :=X \big(\alpha(Y_1, \ldots, Y_q)\big) - \alpha(\nabla_X Y_1, \ldots, Y_q)
     - \cdots - \alpha (Y_1, \ldots, \nabla_X Y_q)
\end{align}
for any $X, Y_1, \ldots, Y_q \in \Gamma(TM)$.

There is a natural isomorphism between $TM$ and $T^*M$, by identifying canonically
each fibre of $TM$ with its dual.
It induces a canonical isomorphism $\sharp: \Omega^1(M) \mapsto \Gamma(TM)$,
for which we write $\sharp^{-1}=:\flat$.
Note that, if a vector field and its corresponding $1$--form in local coordinates
are written by the Einstein summation convention,
$\sharp$ (or $\flat$) amounts to raising (or lowering)
the indices of the coefficients.
Clearly, they extend to the isomorphisms between $T^*M \otimes \cdots \otimes T^*M$
({\it i.e.}, covariant tensors) and $TM \otimes \cdots \otimes TM$ ({\it i.e.}, contravariant tensors).

For instance, consider the covariant derivative $\na: \pform \mapsto \Omega^{q+1}(M)$ defined above.
For $q=1$ and $\alpha \in \Omega^1(M)$, we set $X:=\alpha^\sharp \in \vf$.
Since $\na_YX \in \vf$ for any $Y \in\vf$, we can view $\na X = \na\alpha^\sharp \in \Gamma(TM \otimes TM)$,
{\it i.e.}, $(\na \alpha^\sharp)^\flat \in \Omega^2(M)$.
This example shows that, via the identifications $\sharp$ and $\flat$,
the covariant derivative $\na$ on $\pform$ generalizes
the definition of the Levi-Civita connection.

Now let us briefly review how differential $n$--forms can be integrated on $n$-dimensional
manifolds.
On any orientable $n$-dimensional Riemannian manifold $(M,g)$,
there is a natural $n$--form $dV_g \in \Omega^n(M)$,
called the volume form, which satisfies $\ast 1=dV_g$ and $\ast dV_g = 1$.
Let $\mathcal{A}=\{(U_\alpha,\phi_\alpha): \alpha\in\mathcal{I}\}$ be an atlas for $M$ as before.
By the basic manifold theory, there exists a locally finite $C^\infty$--partition $\{\rho_\alpha: \alpha\in\mathcal{I}\}$
of unity  subordinate to $\mathcal{A}$,
{\it i.e.}, $\sum_{\alpha \in \mathcal{I}} \rho_\alpha =1$,
$0 \leq \rho_\alpha \leq 1$, and $supp (\rho_\alpha) \Subset U_\alpha$.
We define the integration of $\omega\in\Omega^n(M)$ over $M$ by
\begin{equation}
\int_M\omega:=\sum_{\alpha\in\mathcal{I}} \int_{\mathbb{R}^n}
\rho_\alpha((\phi_\alpha^{-1})^\ast\omega) \chi_{\phi_\alpha(U_\alpha)} \sqrt{|g|}\, {\rm d}x_1\cdots {\rm d}x_n,
\end{equation}
where $\phi^\ast$ denotes the pullback of a tensor under $\phi$ with required regularity on $M$,
and $|g| := \det(g_{ij})$.
The integration of function $\phi$ on $M$  is defined
as the integration of its Hodge dual, {\it i.e.}, $\int_M \phi := \int_M\phi\, {\rm d}V_g$.

We now define the Sobolev spaces $W^{k,p}(M; \bigwedge^q T^*M)$ on an $n$-dimensional Riemannian manifold $(M,g)$,
which generalize $W^{k,p}(\R^n)$ for $k\in \mathbb{Z}$ and $1\leq p \leq \infty$.
First, for differential $q$--forms $\alpha,\beta \in \Omega^q(M)$,
$g$ on $M$ defines an inner product $g(\alpha,\beta) =\langle\alpha,\beta\rangle$ by
\begin{equation}
\langle \alpha,\beta\rangle \, dV_g  := \alpha \wedge \ast \beta,
\end{equation}
which is an equality of $n$--forms on $M$.
Then, for alternating contravariant $q$--tensor fields $T$ and $S$,
set $\langle T,S\rangle:= \langle T^\flat, S^\flat \rangle$,
which is consistent with the previous notations.
Next, the $L^p$--norm of $\alpha \in \Omega^q(M)$ is defined as
\begin{equation}
\|\alpha\|_{L^p} := \Big(\int_M \big[\ast (\alpha \wedge \ast \alpha )\big]^{\frac{p}{2}}
\,{\rm d}V_g\Big)^{\frac{1}{p}} = \Big(\int_M \langle\alpha,\alpha\rangle^{\frac{p}{2}}\,{\rm d}V_g\Big)^{\frac{1}{p}}.
\end{equation}
Moreover, for $\alpha \in \Omega^q(M)$, set
\begin{equation}
\|\alpha\|_{W^{k,p}}:=\Big(\sum_{j=0}^k \|\na^j \alpha\|^p_{L^p} \Big)^{\frac{1}{p}}
= \Big(\sum_{j=0}^k \|\underbrace{\na\circ\cdots\circ\na}_{j \text{ times}} \alpha\|^p_{L^p} \Big)^{\frac{1}{p}}.
\end{equation}
We denote by $W^{k,p}(M; \bigwedge^qT^*M)$ the completion of the space of compactly supported $q$--forms
with respect to $\|\cdot\|_{W^{k,p}}$. Notice that, for any contravariant tensor field $X$, the following holds:
\begin{equation*}
\|X\|_{W^{k,p}}:=\|X^\flat\|_{W^{k,p}}.
\end{equation*}
In fact, $W^{k,p}(M;E)$ can be defined for an arbitrary bundle $E$.
Furthermore, we can define the $W^{k,p}$ connection $\na^E$ on bundle $E$.
This can be done since the moduli space of connections is an affine space modelled over
the tensor algebra $\bigotimes^\bullet TM \otimes \bigotimes^\bullet T^*M \otimes E$.
We refer the reader to Jost \cite{Jos08} for the detailed construction.
A key feature for this definition lies in its intrinsic nature,
since the local coordinates on $M$ are not needed to define $W^{k,p}(M;\ptens)$.
In particular, when $p=2$, we denote  $H^{k}(M;\ptens):=W^{k,2}(M;\ptens)$ which are
Hilbert spaces.

\allowdisplaybreaks	
\subsection{Isometric Immersions}
We are concerned with the isometric immersions
of an $n$-dimensional manifold $(M,g)$ into the Euclidean spaces.
A map $f: (M,g)\hookrightarrow (\rnk, g_0)$ is an isometric immersion
if the differential $df: TM \rightarrow T\rnk$ is everywhere injective, and
\begin{equation} \label{eqn_isom immersion}
g_0(f(P))(df_P(X),df_P(Y)) = g(P) (X,Y)
\end{equation}
for every $P \in M$ and $X,Y\in\vf$.
Since $g_0$ is the Euclidean dot product,
Eq. \eqref{eqn_isom immersion} reads
\begin{equation}\label{eqn_reduced isom immersion}
\langle df_P(X), df_P(Y)\rangle = df_P(X) \cdot df_P(Y) = \langle X, Y \rangle.
\end{equation}
	
Then it is natural to ask
whether, for any given $(M,g)$, there is an isometric immersion $f$ (or embedding,
{\it i.e.}, in addition, $f$ is injective everywhere) into the Euclidean space $\rnk$.
The existence of smooth isometric embeddings was established in Nash \cite{Nas56}
in the large when the dimension of the target Euclidean space is high enough:
For any $(M,g)$, there exist a large enough $k$ and a corresponding
smooth isometric embedding $f: (M,g) \hookrightarrow (\rnk, g_0)$
(also see  \cite{Gunther89}).
To achieve this, the problem was approached directly from
Eq. \eqref{eqn_reduced isom immersion},
which is a first-order system of fully nonlinear PDEs,
generically under-determined when $k$ is large.

On the other hand, some progress has been made on the existence and
regularity of immersions/embeddings of two-dimensional Riemannian manifolds
$M^2$ into $\real^3$, with the minimal target dimension $3$, the Janet dimension.
In this setting, the problem can be reduced to a
fully nonlinear Monge-Amp\`{e}re equation,
whose type is determined by the Gauss curvature $K$ of $M$.
For $K>0, K=0$, and $K<0$, the corresponding equation is elliptic, parabolic,
and hyperbolic, respectively.
The case of $K>0$ has a solution in the large due to Nirenberg \cite{Nir53},
while the other two cases are more delicate, and are still widely open in the general setting;
see Han-Hong \cite{HanHon06}.

\subsection{The Gauss-Codazzi-Ricci Equations}

The isometric immersion problem can also be approached via the GCR equations
as the compatibility conditions,
instead of directly tackling Eq. \eqref{eqn_reduced isom immersion}
({\it cf.} do Carmo \cite{DoC92}).
	
The GCR equations are derived from the orthogonal splitting of the tangent
spaces along the isometric immersion $f: (M,g) \hookrightarrow (\rnk, g_0)$.
Indeed, the tangent spaces satisfy $T_P\rnk \cong \rnk \cong T_PM \bigoplus T_PM^\perp$
for each $P \in M$,
where $T_PM^\perp$ is the fibre of the {\em normal bundle} $TM^\perp$ at point $P$,
and $TM^\perp$ is defined as the quotient vector bundle $T\rnk/TM$.
Here and hereafter, we obey the widely adopted convention
to identify $TM$ with $T(fM)$, that is, we view $f$ as the inclusion map.

From now on, we will always use Latin letters $X,Y,Z,\ldots$ to denote tangential vector fields,
{\it i.e.}, elements in $\vf$,
and Greek letters $\xi,\eta,\zeta,\ldots$ to denote normal vector fields,
{\it i.e.}, elements in $\vfn$.

The Levi-Civita connection $\bar{\na}$ on $\rnk$ is the trivial flat connection given by
\begin{equation}
\bar{\na}_V W:= VW = V^i(\p_i W^j)\p_j \qquad\,\, \text{ for } V = V^i\p_i, W=W^i\p_i \in \Gamma(T\rnk),
\end{equation}
where we have used the Einstein summation convention.
In other words, the covariant derivative corresponding to $\bar{\na}$ is just the usual derivative in Euclidean spaces.
It is crucial for the study of isometric immersions
that the projection of $\bar{\na}$ onto $TM$ coincides with the Levi-Civita connection on $M$
which is denoted by $\na$ in the sequel,
and its projection onto $TM^\perp$ defines an affine connection on the normal bundle which is
written as $\na^\perp$; see \S 6 in \cite{DoC92}.
Throughout this paper, we use notation $\bar{\na}_V W$ instead of $VW$
to emphasize that $\{\na, \na^\perp\}$ both come from the orthogonal splitting of $\bar{\na}$.
We also adopt the convention that the geometric quantities with an overhead bar are associated
with the {\em total space} $\rnk$ (as introduced in \S 1),
while the quantities with superscript $\perp$ are associated with the normal bundle $TM^\perp$.
Also, for any $X,Y\in\vf$, the vector field $\bar{\na}_XY \in T\rnk$ is well-defined.
Thus, we can define a symmetric bilinear form $B: \vf \times \vf \mapsto \Gamma(TM^\perp) \nonumber$,
known as the second fundamental form, by
\begin{equation}
B(X,Y) := \bar{\na}_XY - \na_XY.
\end{equation}
By a slightly abusive notation,
we may view $B: \vf \times \vf \times \vfn \mapsto \real$
as
$$
B(X,Y,\xi):= \langle B(X,Y),\xi\rangle \qquad\,\,\,\, \mbox{for $X,Y\in\vf$ and $\xi \in \vfn$}.
$$

Furthermore,
for $\xi \in \vfn$, define $S_\xi: \vf \mapsto \vf$, the shape operator (sometimes also called the second fundamental form), by
\begin{equation} \label{eqn_shape operator}
S_\xi X:= -\bar{\na}_X\xi + \nap_X \xi.
\end{equation}
Equivalently, it can be defined by contracting $B$:
\begin{equation}\label{eqn_shape operator and B}
B(X,Y,\xi) =: \langle S_\xi X, Y\rangle.
\end{equation}

In addition, the Riemann curvature tensor on $M$
is a rank--$4$ tensor $R: \vf \times \vf \times \vf \times \vf \mapsto \real$, defined by
\begin{equation}
R(X,Y,Z,W):=\langle\na_X\na_Y Z, W\rangle-\langle \na_Y\na_X Z, W\rangle + \langle \na_{[X,Y]}Z, W\rangle.
\end{equation}
Notice that the last two coordinates $(Z,W)$ do not enter the definition of $R$ in an essential way.
In fact, for any vector bundle $E$ over $M$ with affine connection $\na^E$,
we can define the Riemann curvature on $E$ as $R^E:\vf \times \vf \times \Gamma(E) \times \Gamma(E) \mapsto \real$ by
\begin{equation}
R^E(X,Y,s_1,s_2):=\langle[\na^E_X, \na^E_Y] s_1, s_2 \rangle - \langle \na^E_{[X,Y]}s_1, s_2 \rangle
\end{equation}
for $X,Y\in\vf$ and $s_1,s_2 \in \Gamma(E)$.
For our purpose,
we consider three bundles over $M$: $(TM, \na)$, $(TM^\perp, \nap)$, and $(T\rnk,\bar{\na})$.
We denote their Riemann curvatures by $R, R^\perp$, and $\bar{R}$, respectively.

With the geometric notations above, we are now at the stage of introducing the GCR equations.
The GCR equations express the flatness of the Euclidean space $(\rnk, g_0)$, {\it i.e.}, $\bar{R}=0$.
We take any $X,Y \in \vf$ and sections $s_1,s_2\in \Gamma(E)$ in
\begin{equation}
\bar{R}(X,Y, s_1, s_2) = 0.
\end{equation}
Owing to the split: $T\rnk \simeq TM \bigoplus TM^\perp$ (at least locally),
we can take $(s_1, s_2)$ to be one of the three combinations: (tangential, tangential),
(tangential, normal), and (normal, normal).
The resulting equations are named after Gauss, Codazzi, and  Ricci, respectively.

\begin{theorem}[GCR Equations] \label{thn_GCR}
Assume that $f: (M,g) \hookrightarrow (\rnk, g_0)$ is an isometric immersion.
Then the following Gauss-Codazzi-Ricci equations are satisfied{\rm :}
\begin{eqnarray}
&&\langle B(Y,W), B(X,Z) \rangle - \langle B(X,W), B(Y,Z)\rangle=R(X,Y,Z,W),\label{2.14a}\\
&&\bar{\na}_Y B(X,Z) = \bar{\na}_X B(Y,Z),\\
&&\langle [S_\eta, S_\xi]X, Y\rangle = R^\perp (X,Y,\eta,\xi). \label{2.16a}
\end{eqnarray}
\end{theorem}

In Theorem \ref{thn_GCR}, we have expressed the GCR equations
in the most compact form.
Nevertheless, to analyze the weak rigidity,
it is helpful to rewrite the Codazzi and Ricci equations
in a less concise manner.

\begin{theorem} \label{thm_reduced GCR}
The following equations are equivalent to the GCR system{\rm :}
\begin{equation}\label{eqn_gauss}
\langle B(X,W), B(Y,Z)\rangle - \langle B(Y,W), B(X,Z) \rangle =-R(X,Y,Z,W),
\end{equation}
\begin{align}\label{eqn_reduced codazzi}
XB(Y,Z,\eta) - YB(X,Z,\eta)
&= B([X,Y],Z,\eta)	- B(X,\na_YZ,\eta) - B(X,Z,\nap_Y\eta)\nonumber\\
&\,\,\,\,\,\,\, + B(Y,\na_XZ,\eta) + B(Y,Z,\nap_X\eta),
\end{align}
\begin{align}\label{eqn_reduced ricci}
X\langle \nap_Y \xi,\eta\rangle - Y \langle \nap_X\xi,\eta \rangle
&= \langle \nap_{[X,Y]}\xi,\eta \rangle - \langle \nap_X \xi, \nap_Y \eta\rangle
+ \langle \nap_Y\xi, \nap_X \eta\rangle\nonumber\\
&\,\,\,\,\,\,\, + B (\bar{\na}_X \xi-\nap_X\xi, Y, \eta) - B (\bar{\na}_X\eta-\nap_X\eta, Y, \xi)
\end{align}
for any tangential vector fields $X,Y,Z,W\in \Gamma(TM)$
and any normal vector fields $\xi,\eta \in \Gamma(TM^\perp)$.
\end{theorem}

\begin{proof}
The Gauss equation (2.18) takes the same form as in Theorem \ref{thn_GCR}.
For the Codazzi equation (2.19), using the Leibniz rule, we have
\begin{align*}
\bar{\na}_XB(Y,Z,\eta) &= X\langle B(Y,Z),\eta\rangle - \langle B(\na_XY, Z), \eta\rangle - \langle B(Y,\na_XZ), \eta\rangle - \langle B(Y,Z), \na^\perp_X\eta\rangle\\
&= \langle \na_X^\perp B(Y,Z), \eta\rangle - \langle B(\na_XY, Z), \eta\rangle - \langle B(Y,\na_XZ), \eta\rangle,
\end{align*}
as well as the analogous expression for $\bar{\na}_YB(X,Z,\eta)$:
\begin{align*}
\bar{\na}_YB(X,Z,\eta) = \langle \na_Y^\perp B(X,Z), \eta\rangle - \langle B(\na_YX, Z), \eta\rangle - \langle B(X,\na_YZ), \eta\rangle.
	\end{align*}
Equating these two expressions by Eq.\,(2.16) and noticing that $\na_XY-\na_YX = [X,Y]$, we obtain Eq.\,(2.19).

For the Ricci equation (2.20), the right-hand side of Eq.\,(2.17) can be expanded as
\begin{align*}
R^\perp(X,Y,\eta,\xi) &= \langle\na^\perp_X\na^\perp_Y\eta,\xi\rangle - \langle \na^\perp_Y \na^\perp_X \eta,\xi\rangle + \langle\na^\perp_{[X,Y]}\eta, \xi\rangle \\
&= X\langle \na^\perp_Y\eta, \xi\rangle - \langle \na^\perp_Y\eta, \na^\perp_X\xi\rangle - Y\langle\na^\perp_X\eta,\xi\rangle
+ \langle\na^\perp_X\eta,\na^\perp_Y\xi\rangle + \langle\na^\perp_{[X,Y]}\eta, \xi\rangle,
\end{align*}
by the definition of $R^\perp$ and the Leibniz rule.
Moreover, the left-hand side of Eq.\,(2.17) equals $\langle S_\eta S_\xi X, Y\rangle - \langle S_\xi S_\eta X, Y \rangle$, where
\begin{align*}
\langle S_\eta S_\xi X, Y\rangle &= B(S_\xi X, Y, \eta) = -B(\bar{\na}_X\xi-\na^\perp_X\xi, Y, \eta),
\end{align*}
thanks to the definition of $S$;
Similarly,
\begin{align*}
\langle S_\xi S_\eta X, Y\rangle= -B(\bar{\na}_X\eta-\na^\perp_X\eta, Y, \xi).
\end{align*}
Thus, we obtain Eq.\,(2.20). This completes the proof.
\end{proof}

In the GCR equations in the form of either Theorem 2.1 or Theorem 2.2,
we view $(B, \nabla^\perp)$ as unknowns and $g$ (hence $R$) as given.
The GCR equations constitute a necessary condition for the existence of isometric immersions.
Tenenblat \cite{Ten71} proved that, if everything is smooth,
this is also sufficient for the local existence of isometric immersions.

From the point of view of PDEs, the three equations in Theorem \ref{thm_reduced GCR}
form a system of first-order nonlinear equations.
The left-hand sides of Eqs. \eqref{eqn_reduced codazzi}--\eqref{eqn_reduced ricci}
can be regarded as the principal parts, while the nonlinear terms on the right-hand sides
are of zero-th order.
The nonlinear terms are quadratic in the form of $B\otimes B$, $\nap \otimes \nap$,
and $B\otimes \nap$.

\section{Intrinsic Compensated Compactness Theorems on Riemannian Manifolds}

In this section, we first formulate and prove a general abstract compensated compactness
theorem in the framework of functional analysis (in \S 3.1).
As a special case, it implies
a global intrinsic version of the div-curl lemma on Riemannian manifolds,
which generalizes the well-known classical versions in $\real^n$,
first by  Murat \cite{Mur78} and Tartar \cite{Tar79}.
Such a geometric div-curl lemma,
which is presented in \S 3.2, will serve as a basic tool
in the subsequent development.

\subsection{General Compensated Compactness Theorem on Banach Spaces}

Throughout this section, for a normed vector space $X$ with dual space $X^*$,
we use $\langle\cdot,\cdot\rangle_X$
to denote the duality pairing of $(X,X^*)$.
Let $\hil$ be a Hilbert space over field $\mathbb{K}=\real \text{ or } \mathbb{C}$ such that $\hil=\hil^*$.
Let $Y$ and $Z$ be two Banach spaces over $\mathbb{K}$ with their dual spaces
$Y^*$ and $Z^*$, respectively.
In what follows, we consider two bounded linear operators $S:\hil\mapsto Y$,
$T:\hil\mapsto Z$,
and their adjoint operators $\sdag: Y^*\mapsto \hil$ and $\tdag: Z^* \mapsto \hil$,
respectively.

Furthermore, the following conventional notations are adopted:
For any normed vector spaces $X$, $X_1$, and $X_2$,
we write $\{s^\epsilon\}\subset X$ for a sequence $\{s^\epsilon\}$ in $X$ as a subset,
and $X_1 \Subset X_2$ for a compact embedding between the normed vector spaces.
We use $\|\cdot\|_X$ to denote the norm in $X$.
Also, we use $\rightarrow$ to denote the strong convergence of sequences and $\rightharpoonup$
for the weak convergence.
Furthermore, we denote the closed unit ball of space $X$ by $\bar{B}_X:=\{x\in X: \|x\|\leq 1\}$,
the open unit ball by $B_X:=\{x\in X: \|x\|<1\}$.
Moreover, for a linear operator $L: X_1\mapsto X_2$,
its kernel is written as $\ker(L)\subset X_1$,
and its range is denoted by $\ran(L)\subset X_2$.
Finally, for $X_1\subset X$ as a vector subspace,
its {\em annihilator} is defined
as $X_1^\perp:=\{f\in X^*:f(x)=0 \text{ for all } x\in X_1\}$.

To formulate the compensated compactness theorem
in the general functional-analytic framework, we
first introduce the following two bounded linear operators:
Define
\begin{equation*}
\begin{cases}
\st : \hil \mapsto Y\bigoplus Z, \quad\,\, (\st) h:=(Sh,Th) \,\,\,\mbox{for $h\in\hil$};\\
\stdag: (Y\bigoplus Z)^*\cong Y^*\bigoplus Z^* \mapsto \hil,
\quad\,\, (\stdag)(a,b) := \sdag a + \tdag b \,\,\,\mbox{for $a\in Y^*$ and $b\in Z^*$}.
\end{cases}
\end{equation*}
Here and throughout, the Banach space $Y\bigoplus Z$ is endowed with
norm $\|(y,z)\|_{Y\bigoplus Z}:=\|y\|_Y+\|z\|_Z$.
Notice that, for any $a\in Y^*, b\in Z^*$, and $h\in \hil$,
\begin{equation*}
\langle h, (\stdag)(a,b)\rangle = \langle h, \sdag a + \tdag b\rangle
=\langle Sh, a \rangle_Y + \langle Th, b\rangle_Z = \langle (\st)h, (a,b)\rangle_{Y\bigoplus Z}.
\end{equation*}
Thus, $\stdag$ is in fact the adjoint operator of $\st$, namely $(\st)^\dagger = \stdag$.

In the setting above, we are concerned with the following question:
\begin{q*}
Let $\useq$, $\vseq \Subset \hil$ be two sequences so that
there exist $\ub, \vbb \in \hil$ such that $u^\e \rightharpoonup \ub$ and $v^\e \rightharpoonup \vbb$ in $\hil$ as $\e\to 0$.
Under which conditions does the following hold{\rm :}
$$
\langle u^\e, v^\e \rangle_\hil \rightarrow \langle \ub,\vbb\rangle_\hil \qquad \text{ as } \epsilon \rightarrow 0?
$$
\end{q*}

The goal of this subsection is to provide a sufficient condition for
the convergence $\langle u^\e, v^\e \rangle_\hil \rightarrow \langle \ub,\vbb\rangle_\hil$ as $\e\to 0$.
Roughly speaking, it requires the existence of a ``nice'' pair of bounded
linear operators $S:\hil\mapsto Y$ and $T:\hil\mapsto Z$
such that $S$ and $T$ are orthogonal to each other, and $\st$ gains certain compactness/regularity.
More precisely, we prove

\begin{theorem}[General Compensated Compactness Theorem on Banach Spaces]\label{thm_abstract compensated compactness}
Let $\hil$ be a Hilbert space over $\mathbb{K}$, $Y$ and $Z$ be reflexive Banach spaces over $\mathbb{K}$,
and let $S:\hil\mapsto Y$ and $T:\hil\mapsto Z$ be bounded linear operators satisfying
\begin{enumerate}
\item[\rm (Op 1)]
Orthogonality{\rm :}
\begin{equation}\label{eqn_almost ortho two}
	S\circ \tdag = 0, \qquad T\circ \sdag =0;
\end{equation}
\item[\rm (Op 2)]
For some Hilbert space $(\htt; \|\cdot\|_{\htt})$ so that $\hil$ embeds compactly into $\htt$,
there exists a constant $C>0$ such that, for any $h\in\hil$,
\begin{equation}\label{estimate new}
\|h\|_{\hil} \leq C\Big(\|(Sh,Th)\|_{Y\bigoplus Z}+\|h\|_{\htt}\Big) = C\Big(\|Sh\|_Y+\|Th\|_Z + \|h\|_{\htt}\Big).
\end{equation}
\end{enumerate}

\noindent
Assume that two sequences $\useq, \vseq \subset \hil$ satisfy the following conditions{\rm :}
\begin{enumerate}
\item[\rm (Seq 1)]
$u^\e \rightharpoonup \ub$ and  $v^\e\rightharpoonup \vbb$ in $\hil$ as $\e \to 0${\rm ;}
\item[\rm (Seq 2)]
$\{Su^\e\}$ is pre-compact in $Y$, and $\{Tv^\e\}$ is pre-compact in $Z$.
\end{enumerate}
Then
$$
\langle u^\e, v^\e\rangle_\hil \rightarrow \langle\ub, \vbb\rangle_\hil \qquad \text{ as } \e\to 0.
$$
\end{theorem}

\begin{proof}
We divide the proof into eight steps.

\smallskip
{\bf 1.} In order to show that $\st: \hil \mapsto \yz$ has a finite-dimensional kernel,
we consider the following subset of $\hil$:
\begin{equation*}
E=j^{-1}\big(j[\ker(\st)]\cap\bar{B}_{\utilde{\hil}}\big),
\end{equation*}
where $j:\hil\mapsto \htt$ is a compact embedding between the Hilbert spaces.

Suppose that $j(E)\subset\htt$ is compact. First notice that $j(E)$ is the closed unit ball of $j[\ker(\st)]$,
which is a Banach space, due to the closedness of the kernel.
It then follows from the classical
Riesz lemma that $j[\ker(\st)]$ is finite-dimensional.
Since $j$ is an embedding, we can conclude that $\dim\ker(\st)=\dim(j[\ker(\st)])<\infty$.

To prove the compactness of $j(E)$,
take any $h\in E$ and consider the following estimate deduced from condition (Op 2):
\begin{equation}
\|h\|_{\hil} \leq C(\|Sh\|_Y+\|Th\|_Z+\|j(h)\|_{\htt}) \leq C.
\end{equation}
Hence, the boundedness of $E$ in $\hil$ and the compactness of $j: \hil\mapsto \htt$ imply that $j(E)$ is compact.
Thus, the first step is complete.

\smallskip
{\bf 2.} We now show that $\st$ is a closed-ranged operator,
{\em i.e.}, $\ran(\st)\subset \yz$ is a closed subspace.
To this end, we first prove the existence of a constant $\epsilon_0>0$ such that,
for all $h\in\hil$,
\begin{equation}\label{lower bound for st}
\|(\st)h\|_{\yz} \geq \epsilon_0 \|j(h)\|_{\htt}.
\end{equation}
In fact, if the inequality were false, then, for any $\mu \in (0,1)$,
we could find $\{h^\mu\}\subset [\ker(\st)]^\perp$ and $\|j(h^\mu)\|_{\htt}=1$
such that
\begin{equation*}
\|(\st)h^\mu\|_{\yz}\leq \mu
\end{equation*}
and
\begin{equation*}
{\rm dist}\big(h^\mu, \ker(\st)\big)\geq\hat{\e}_0
\end{equation*}
for some $\hat{\e}_0>0$.
Such a choice of $\{h^\mu\}$ is possible, owing to the finite-dimensionality of $\st$.
Now, plugging $\{h^\mu\}$ into condition (Op 2), we have
$$
\|h^\mu\|_{\hil}\leq C\big(\|(\st)h^\mu\|_{\yz}+\|j(h^\mu)\|_{\htt}\big)\leq C(1+\mu)\leq 2C.
$$
It follows from the compactness of $j$ that $\{j(h^\mu)\}$ is pre-compact in $\htt$.
Let $j(h)$ be a limit point of $\{j(h^\mu)\}$ in $j(\hil) \subset \htt$.
Then one must have $h\in\ker(\st)$, in view of $\|(\st)h^\mu\|_{\yz}\leq \mu$.
However, this contradicts the fact that $\{h^\mu\}$ were chosen to have a distance
at least $\hat{\e}_0$ from $\ker(\st)$.
Therefore, we have established estimate \eqref{lower bound for st}.

To proceed, consider a sequence $\{h^\mu\}\subset\hil$ such that
$(\st)h^\mu \rightarrow w$ for some $w\in\yz$.
Our goal is to show that $w\in\ran(\st)$.
Indeed, by the projection theorem for Hilbert spaces, we can decompose $h^\mu=k^\mu+r^\mu$ with
$k^\mu\in\ker(\st)$ and $r^\mu=[\ker(\st)]^\perp$
(which is a closed subspace of $\hil$).
Then estimate \eqref{lower bound for st} gives
\begin{equation}
\|(\st)(h^{\mu_1}-h^{\mu_2})\|_{\yz}=\|(\st)(r^{\mu_1}-r^{\mu_2})\|_{\yz}\geq\epsilon_0 \|j(r^{\mu_1}-r^{\mu_2})\|_{\htt}.
\end{equation}
As a consequence, $\{j(r^\mu)\}$ is a Cauchy sequence in $j(\hil)$, which implies that
there exists $j(r)\in j(\hil)$ such that $j(r^\mu)\to j(r)$ in $j(\hil)$ as $\mu\to 0$.
Since $j$ is an embedding, it follows that $r^\mu\rightarrow r$ in $\hil$.
Therefore, by the closed graph theorem of Banach spaces,
$(\st)h^\mu=(\st)r^\mu\rightarrow (\st)r$. It now follows that $w=(\st)r$ so that the range of
$\st$ is closed.

{\bf 3.}
As an immediate corollary, we can obtain the following decomposition of $\hil$:
\begin{equation}\label{orthogonal decomp of H}
\hil = \ker(\st) \bigoplus \ran(\stdag).
\end{equation}
Here, $\bigoplus$ denotes the topological direct sum of Banach spaces, and the direct summands are orthogonal as Hilbert spaces.

Indeed, by the projection theorem, $\hil=\ker(\st)\bigoplus[\ker(\st)]^\perp$.
On the other hand,
$$
[\ker(\st)]^\perp=\overline{\ran[(\st)^\dagger]}=\overline{\ran(\stdag)}.
$$
We recall the closed range theorem in Banach spaces, which states that
a bounded linear operator is closed-ranged if and only if its adjoint operator is closed-ranged
({\it cf.} \cite{f}).
Hence, we may deduce from the preceding equalities that
\begin{equation*}
[\ker(\st)]^\perp=\ran(\stdag).
\end{equation*}
The decomposition in Eq. \eqref{orthogonal decomp of H} now follows immediately.

With this decomposition,
it now suffices to prove Theorem \ref{thm_abstract compensated compactness}
for \emph{surjective} operators $S$ and $T$.
Indeed, all the conditions in (Op 1)--(Op 2) and (Seq 1)--(Seq 2)
continue to hold when $\yz$ is replaced by $\ran(\st)$,
which has been proved to be a closed subspace of $\yz$.
Here we have used the fact that closed subspaces of reflexive Banach spaces
are still reflexive.
Therefore, in the sequel, we always assume $\ran(\st)=\yz$ without loss of generality.

\smallskip
{\bf 4.} Now, we introduce the following operator $\sd: \yzstar \mapsto \yz$:
\begin{equation}
\sd:=(\st)\circ(\stdag)=S\sdag \oplus T\tdag.
\end{equation}
For this {\it generalized Laplacian},
we also prove that it has a finite-dimensional kernel, and its range is closed
(in fact, surjective, in view of the reduction at the end of Step 3).

Indeed, we can find the range of $\sd$ as follows:
\begin{align}\label{aaa}
\ran(\sd) &= \ran\big((\st)|_{\ran(\stdag)}\big)\nonumber\\
&=  \ran\big((\st)|_{[\ker(\st)]^\perp}\big)\nonumber\\
&=\ran\big((\st)|_{[\ker(\st)]^\perp\bigoplus\ker(\st)}\big)\nonumber \\
&=\ran(\st)=\yz,
\end{align}
where the decomposition in Eq. \eqref{orthogonal decomp of H} has been used
in the second equality.

On the other hand, notice that
\begin{equation}
\ker(\sd) = \ker\big(\st|_{\ran(\stdag)}\big)\bigoplus \ker(\stdag).
\end{equation}
In this expression, the first direct summand equals $\ker(\st)$,
by a similar argument as in Eq. \eqref{aaa}.
For the second summand, a standard result in functional analysis gives $\ker(\stdag) = [\ran(\st)]^\perp$,
which is assumed to be $\{0\}$ at the end of Step 3. It follows that
$\ker(\sd)=\ker(\st)$, which is finite-dimensional, by Step $1$.

To summarize,
$\sd$ is a closed-ranged operator with a finite dimensional kernel;
without loss of generality, we may assume $\sd$ to be surjective.

\smallskip
{\bf 5.} We now show a crucial result concerning $\sd$:
\begin{equation*}
(\clubsuit) \qquad \text{\, For each } w\in \ran(\sd), \text{ there exists } \xi \in \sd^{-1}\{w\}
\text{ such that } \|\xi\|_{\yzstar}\leq M\|w\|_{\yz}.
\end{equation*}

First, applying the open mapping theorem in Banach spaces to
operator $\sd:\yzstar \to \ran(\sd)=\yz$,
we can find a constant $\delta >0$ and an element $w_0\in\ran(\sd)$
such that $w_0+\delta B_{\yz} \subset \sd[B_{\yzstar}]$.

From this inclusion, we can prove
\begin{equation}\label{OMT}
\delta\bar{B}_{\yz}\subset \sd(\bar{B}_{\yzstar}).
\end{equation}

In fact, for any $v_0 \in \delta\bar{B}_{\yz}$,
we can write $v_0$ as a convex combination of elements in $w_0+\delta B_{\yz}$,
{\em e.g.}, $v_0=\frac{1}{2}\big((v_0+w_0)+(v_0-w_0)\big)$.
Observe that $\sd[B_{\yzstar}]$ is a convex set, as $\sd$ is a bounded linear operator and $B_{\yzstar}$ is convex.
Since $w_0+\delta B_{\yz}$ lies in $\sd[B_{\yzstar}]$, we conclude that $v_0\in\sd[B_{\yzstar}]$.
Thus,  \eqref{OMT} now follows.

As a consequence, given any $w\in\ran(\sd)=\yz$,
there exists $\eta\in\bar{B}_{\yzstar}$ such that $\delta \frac{w}{\|w\|}= \sd \eta$.
Define $\xi:=\frac{\|w\|}{\delta}\eta$, which yields
$$
\|\xi\|_{\yzstar}\leq \delta^{-1}\|w\|_{\yz}, \qquad
\sd\xi=w.
$$
Now, the proof for $(\clubsuit)$ is completed by choosing $M=\delta^{-1}$.

\smallskip
{\bf 6.}  Now, we employ
\eqref{orthogonal decomp of H}
to decompose sequences $\{u^\e\}$ and $\{v^\e\}$, and the weak limits $\ub$ and $\vbb$.
In the sequel, we denote the canonical projection of $\hil$ onto the first factor
by $\pi_1:\hil=\ker(\st)\bigoplus \ran(\stdag) \mapsto \ker(\st)$.
By Step 1, $\pi_1$ is a finite-rank (hence compact) operator.

Employing such a decomposition, we can write
\begin{equation}\label{decompositions of vector fields u,v}
\begin{cases}
u^\e = \pi_1 u^\e + \sdag a^\e+\tdag b^\e,\\
v^\e=\pi_1v^\e + \sdag \antil + \tdag \bntil,\\
\ub = \pi_1 \ub + \sdag a + \tdag b,\\
\vbb = \pi_1\vbb + \sdag \tilde{a} +\tdag \tilde{b},
\end{cases}
\end{equation}
for some $a,\tilde{a}, a^\e,\antil \in Y^*$ and $b,\tilde{b}, b^\e, \bntil\in Z^*$.
Moreover, applying the orthogonality condition (Op 1), we have
\begin{equation*}
\begin{cases}
\langle u^\e, v^\e\rangle_\hil = \langle \pi_1u^\e, \pi_1v^\e\rangle_\hil + \langle \sdag a^\e, \sdag \antil\rangle_\hil
  + \langle \tdag b^\e, \tdag\bntil\rangle_\hil,\\
\langle \ub, \vbb\rangle_\hil = \langle \pi_1\ub, \pi_1\vbb\rangle_\hil
 + \langle \sdag a, \sdag \tilde{a}\rangle_\hil + \langle \tdag b, \tdag\tilde{b}\rangle_\hil.
\end{cases}
\end{equation*}

Since $\{u^\e\}$ and $\{v^\e\}$ are weakly convergent by assumption (Seq 1),
and $\pi_1$ is a compact operator,
we obtain that $\langle \pi_1u^\e, \pi_1v^\e\rangle \rightarrow \langle \pi_1\ub, \pi_1\vbb \rangle$.
It thus remains to establish
\begin{equation}\label{key convergence}
\langle \sdag a^\e, \sdag \antil\rangle_\hil + \langle \tdag b^\e, \tdag\bntil\rangle_\hil
\rightarrow \langle \sdag a, \sdag \tilde{a}\rangle_\hil + \langle \tdag b, \tdag\tilde{b}\rangle_\hil.
\end{equation}
In the next two steps, we prove \eqref{key convergence}.

\smallskip
{\bf 7.}
The starting point is to observe that the left-hand side of \eqref{key convergence} can be rewritten as follows:
\begin{align}\label{aa}
\langle \sdag a^\e, \sdag \antil\rangle_\hil + \langle \tdag b^\e, \tdag\bntil\rangle_\hil
=\langle S\sdag a^\e,\antil\rangle_{Y}
+ \langle T\tdag \bntil,  b^\e \rangle_{Z}
=\big\langle\sd (a^\e,\bntil), (\antil,b^\e)\big\rangle_{\yz}.
\end{align}
On the other hand, let us apply $S$ to $u^\e$ and $T$ to $v^\e$ in \eqref{decompositions of vector fields u,v}.
Using the definition of $\pi_1$ and the orthogonality condition (Op 1), we immediately find that
\begin{equation}\label{aaa-1}
Su^\e = S\sdag a^\e, \qquad Tv^\e = T\tdag \bntil,
\end{equation}
{\it i.e.}, $\sd  (a^\e,\bntil) = (Su^\e, Tv^\e)$.
As $\{Su^\e\}\subset Y$ and $\{Tv^\e\}\subset Z$ are pre-compact by (Seq 2),
it suffices to show the boundedness of $\{(\antil, b^\e)\}\subset\yzstar$
to conclude \eqref{key convergence}. To see this point, assuming the boundedness, one can deduce that
\begin{align*}
&\left|\big\langle\sd(a^\e,\bntil),(\antil, b^\e)\big\rangle_{\yz}
-\big\langle\sd (a,\tilde{b}),(\tilde{a},b) \big\rangle_{\yz}\right| \\[1mm]
& \leq \left|\big\langle \sd(a^\e-a,\bntil-\tilde{b}),(\antil,b^\e)\big\rangle_{\yz}\right|
 +\left|\big\langle\sd (a,\tilde{b}),(\antil-\tilde{a},b^\e-b)\big\rangle_{\yz}\right| \\[1mm]
&\leq \left|\big\langle(Su^\e-S\ub, Tv^\e-T\vbb),(\antil, b^\e)\big\rangle_{\yz}\right|
 +\left|\big\langle\sd(a,\tilde{b}),(\antil-\tilde{a},b^\e-b)\big\rangle_{\yz}\right|\\[1mm]
&=: \text{I}^\e + \text{II}^\e.
\end{align*}

\medskip
For $\text{I}^\epsilon$, since $\|(\antil,b^\e)\|_{\yzstar}\leq M$,
we can use the pre-compactness of $\{Su^\e\}\subset Y$
and $\{Tv^\e\}\subset Z$ in assumption (Seq 1) to conclude
\begin{equation*}
\text{I}^\e \leq M \big\|\big(S(u^\e-\ub), T(v^\e - \vbb)\big) \big\|_{\yz} \rightarrow 0.
\end{equation*}

For $\text{II}^\e$, we need to invoke the reflexivity of $Y$ and $Z$.
As a Banach space is reflexive if and only if its dual space is reflexive,
we know that $\yzstar$ is reflexive.
Then the bounded sequence $\{(\antil, b^\e)\}$ is weakly pre-compact, by Theorem 3.31 in \cite{f}.
Moreover,  Theorem 4.47 (Eberlein-\u{S}mulian theorem) in \cite{f} implies the weak
convergence of $\{(\antil,b^\e)\}$. Thus, viewing $\sd(a,\tilde{b})$ as an element in $(\yz)^{**}$,
we conclude that $\text{II}^\e \rightarrow 0$.

{\bf 8.} From the above arguments, it remains to prove the boundedness of
$\{(\antil, b^\e)\}\subset\yzstar$.
We further remark that $(\antil, b^\e)$ can be chosen {\it modulo $\ker(\sd)$}.
More precisely,
it is enough to exhibit one particular representative $(\antil, b^\e)$ in $\sd^{-1}\{(Sv^\e, Tu^\e)\}$
such that $\|(\antil, b^\e)\|_{\yzstar}\leq C$, where $C<\infty$ is independent of $\e$.
To see this, notice that
\begin{equation*}
\langle \sdag a^\e, \sdag \antil\rangle_\hil
+ \langle \tdag b^\e, \tdag\bntil\rangle_\hil
= \big\langle\sd (\antil,b^\e), (a^\e,\bntil)\big\rangle_{\yz}
=\big\langle (Sv^\e, Tu^\e), (a^\e,\bntil)\big\rangle_{\yz},
\end{equation*}
analogous to Eqs. \eqref{aa}--\eqref{aaa-1}.
Therefore, we have the freedom to choose a representative $(\antil,b^\e)$
in the co-set $(Sv^\e, Tu^\e)+\ker(\sd)$, without changing the expression on
the left-hand side of \eqref{key convergence}.

For this purpose, let us invoke the key estimate $(\clubsuit)$ proved in Step $5$
to find a constant $M\in (0, \infty)$ satisfying
\begin{equation*}
\|(\antil, b^\e)\|_{\yzstar} \leq M\|(Sv^\e, Tu^\e)\|_{\yz}.
\end{equation*}
Then it suffices to prove the uniform boundedness of $\{(Sv^\e, Tu^\e)\}$ in $\yz$.

Indeed, as $u^\e \rightharpoonup\ub$ and $v^\e\rightharpoonup \vbb$ according to assumption (Seq 1),
we know that $\{u^\e\}$ and $\{v^\e\}$ are bounded in $\hil$,
by using the weak lower semi-continuity of $\|\cdot\|_\hil$.
By the reflexivity of Hilbert spaces (due to the Riesz representation theorem), $\{u^\e\}$ and $\{v^\e\}$ are weakly pre-compact
in $\hil$. Next, by standard results in functional analysis, the continuous linear operator $\st:\hil\mapsto \yz$
with respect to the strong topologies is also continuous
when both $\hil$ and $\yz$ are endowed with the weak topologies
(see
\cite{f} for details).
As continuous mappings take pre-compact sets to pre-compact sets, $\{(Sv^\e, Tu^\e)\}$ is weakly pre-compact in $\yz$;
equivalently, we have established the uniform boundedness of $\{(Sv^\e, Tu^\e)\}$
in the strong topology of $\yz$, because the Banach spaces $Y$ and $Z$ are reflexive.

Therefore, we have proved that
\begin{equation}
\|(\antil, b^\e)\|_{\yzstar} \leq M \sup_{\e>0}\|(Sv^\e, Tu^\e)\|_{\yz}\leq M'<\infty,
\end{equation}
from which the convergence in \eqref{key convergence} follows. This completes the proof.
\end{proof}

Let us close this subsection
with two remarks on Theorem \ref{thm_abstract compensated compactness}.
First, in terms of the applications,
$Y$ and $Z$ are usually Hilbert spaces $H^s=W^{s,2}$ for $s\in\mathbb{Z}$,
or $S$ and $T$ are known to be Fredholm operators.
In such cases, our arguments can be essentially simplified.
Second, the reflexivity of $Y$ and $Z$ is necessary for
Theorem \ref{thm_abstract compensated compactness} to hold.

\begin{remark}
If $Y$ and $Z$ are Hilbert spaces,
then $\sd:\yzstar\cong\yz\mapsto\yz$ is self-adjoint{\rm :}
$\sd^\dagger=(S\sdag)^\dagger \oplus (T\tdag)^\dagger = S\sdag \oplus T\tdag =\sd$.
Thus, we can decompose
\begin{equation}
\yz = \ker(\sd) \bigoplus {\rm \ran} (\sd)
\end{equation}
{\rm (}cf. Proposition {\rm 7.32} in \cite{f}{\rm )} so that
the closed-rangedness of $\sd$ and the crucial estimate $(\clubsuit)$ are automatically
verified.
As a consequence, after establishing the finite-dimensionality of $\ker(\sd)$ as in the first half of Step $1$,
we can proceed to Step $5$ in the proof of Theorem {\rm \ref{thm_abstract compensated compactness}}.
On the other hand, if $S$ and $T$ are given {\em a priori} to have finite-dimensional kernels
and co-kernels, then $\st$ is a Fredholm operator {\rm (}whose range is automatically closed{\rm )}.
Again, we can directly proceed to Step $5$.
\end{remark}

\begin{remark}
The assumption of reflexivity of $Y$ and $Z$ is indispensable.
One counterexample for non-reflexive $Y$ and $Z$ is
the ``Fakir's carpet'' in Conti-Dolzmann-M\"{u}ller \cite{CDM} {\rm (}a more extensive variant
also appeared in DiPerna-Majda \cite{DM2}{\rm )}.
Consider the following vector fields on $\Omega:=(0,1)^3${\rm :}
\begin{equation}
u^\e(x) = v^\e(x) = (\sqrt{m}\sum_{j=1}^m\chi_{[\frac{j}{m}, \frac{j}{m}+\frac{1}{m^2}]}, 0, 0)^\top\qquad \mbox{with  $m:= \lfloor \frac{1}{\e}\rfloor$},
\end{equation}
operators $S={\rm div}$, $T={\rm curl}$, and spaces $\mathcal{H} = L^2(\Omega; \R^3)$, $Y=W^{-1,1}(\Omega)$,
and $Z=W^{-1,1}(\Omega; \R^3)$.
Then $u^\e \rightharpoonup 0$ and $v^\e \rightharpoonup 0$ in $\mathcal{H}$,
but $\langle u^\e, v^\e\rangle \to 1$, as $\e\to 0$, where $\langle\cdot,\cdot\rangle$ denotes the $L^2$ inner product.
On the other hand, for any test function $\psi \in W^{1,\infty}_0(\Omega)$ and any test vector
field $\phi=(\phi^1, \phi^2, \phi^3)^\top \in W^{1,\infty}_0(\Omega; \R^3)$, we have
\begin{align*}
\Big|\int_\Omega Su^\e(x)\psi(x)\,\dd x\Big| &= \Big|\sqrt{m}\sum_{j=1}^m \int_{\frac{j}{m}}^{\frac{j}{m} + \frac{1}{m^2}} \frac{\p\psi}{\p x^1}(x)\,\dd x\Big|\\
&\leq \|\psi\|_{W^{1,\infty}(\Omega)} \frac{1}{\sqrt{m}} \longrightarrow 0 \qquad \text{ as } \e \rightarrow 0,
\end{align*}
and similarly
\begin{align*}
\Big|\int_\Omega Tv^\e(x)\cdot\phi(x)\,\dd x\Big| &= \Big|\int_\Omega v^\e(x) \cdot {\rm curl}\,\phi(x)\,\dd x \Big|\\
&=  \Big|\sqrt{m}\sum_{j=1}^m \int_{\frac{j}{m}}^{\frac{j}{m} + \frac{1}{m^2}} \Big(\frac{\p \phi^3}{\p x^2} - \frac{\p \phi^2}{\p x^3}\Big) \,\dd x\Big|\\
&\leq \|\phi\|_{W^{1,\infty}(\Omega; \R^3)}\frac{1}{\sqrt{m}} \longrightarrow 0 \qquad \text{ as } \e \rightarrow 0.
\end{align*}
Therefore, $Su^\e \rightharpoonup 0$ in $Y$ and $Tv^\e \rightharpoonup 0$ in $Z$.

The failure of the weak continuity is related to the phenomenon of ``concentration'' in fluid mechanics, nonlinear elasticity,
and calculus of variations; see \S {\rm 6} for further discussions.
\end{remark}

\subsection{Intrinsic Div-Curl Lemma on Riemannian Manifolds}

Now we adapt the general functional-analytic theorem, Theorem \ref{thm_abstract compensated compactness},
to the geometric settings.
Our aim is to obtain a global intrinsic div-curl lemma on
Riemannian manifolds (Theorem \ref{thm_main thm goemetric divcurl}), independent of local coordinates,
which will be applied to analyze the global weak rigidity of isometric immersions of
Riemannian manifolds in the subsequent development.
Let us begin with the notions of several geometric quantities.

First of all, it is well-known that the {\it divergence} of a vector field is globally defined
for an arbitrary dimensional orientable manifold $M$:
\begin{equation}\label{eqn_def of divergence by lie derivative}
{\rm div} X:= \ast d \ast (X^\flat) \equiv \ast (\lie_X dV_g) \qquad\,\,\,\mbox{for any $X\in\vf$},
\end{equation}
where $\ast$ is the Hodge star, $\mathcal{L}$ is the Lie derivative,
and $dV_g$ is the volume form with respect to metric $g$.
Geometrically, ${\rm div} X$ measures the change of volume in the direction of $X$.
The {\it gradient} can also be defined in higher dimensions:
\begin{equation}
{\rm grad} f:= (df)^\sharp.
\end{equation}

The notion of {\it curl} is more subtle.
Our ordinary definition for {\it curl}, if we require to be intrinsic, is
only well-defined on three-dimensional manifolds.
This is because one needs to identify canonically $\Omega^2(M)$ with $\vf$,
which is only valid in three dimensions via the Hodge duality.
Physically, ${\rm curl}(X)$ measures the rotation of the flow generated
by $X$ pointing to the rotation axis, where the direction of the rotation axis is unambiguous
in three dimensions.
For $X\in\vf$ with $\dim M=3$,
\begin{equation}
{\rm curl} X:= [\ast d (X^\flat)]^\sharp,
\end{equation}
where $\sharp \circ \ast$ identifies the $2$--form $d X^\flat$ with the vector field.
On an arbitrarily dimensional manifold $M$,
we can define the {\it generalized curl} just as the $2$--form,
without pulling back to vector fields:
\begin{equation}
{\rm curl} X:= d(X^\flat).
\end{equation}

\begin{remark}
The underlying reason for introducing the generalized curl is
the algebra isomorphism $\bigwedge^2(T^*M) \cong \mathfrak{so}(n)$,
which is the Lie algebra of the special orthogonal group.
Our ${\rm curl}(X)$ is just a field of anti-symmetric matrices
which can be naturally interpreted as rotations,
thanks to the structure of $\mathfrak{so}(n)$.
\end{remark}

Next, we define $\delta: \Omega^q(M) \mapsto \Omega^{q-1}(M)$ to be
the (formal) adjoint operator of $d$ in the following sense:
\begin{equation}
\int_M \delta \alpha \wedge \ast \beta = \int_M \alpha \wedge \ast d\beta \qquad\,\,
\text{ for } \alpha \in \Omega^q(M) \text{ and } \beta \in \Omega^{q-1}(M).
\end{equation}
To emphasize the dependence on metric $g$, it can also be expressed as
\begin{equation}
\int_M \langle \delta \alpha, \beta \rangle\,{\rm d} V_g = \int_M \langle \alpha,d\beta\rangle\,{\rm d}V_g
\qquad \text{ for } \alpha \in \Omega^q(M) \text{ and } \beta \in \Omega^{q-1}(M).
\end{equation}
Equivalently, we can define $\delta:= (-1)^{n(q+1)+1} \ast d\ast$,
where $\ast$ is the Hodge star.
Hence, in the definition of divergence \eqref{eqn_def of divergence by lie derivative},
${\rm div} X$ is simply $\delta X^\flat$ modulo a sign.
Therefore, we always regard $\delta$ as the {\it divergence}
and $d$ as the {\it curl}.

Moreover, the Laplace-Beltrami operator
$\Delta: \Omega^q(M)\mapsto \Omega^q(M)$ on manifold $M$
is defined for each $0\leq q \leq n$ as
\begin{equation}
\Delta:=d\circ\delta+\delta\circ d.
\end{equation}
Denote ${\rm Har}^q(M):=\ker (\Delta)$,
the space of harmonic $q$--forms.
Then $\alpha \in {\rm Har}^q(M)$
if and only if $d\alpha =0$ and $\delta \alpha =0$.

One fundamental result concerning the Laplace-Beltrami operator is the Hodge decomposition theorem
({\it cf.} Warner \cite{War71}):

\begin{theorem}[Hodge Decomposition]\label{thm_hodge}
Let $M$ be a closed, orientable Riemannian manifold.
For each integer $q$ with $0\leq q \leq n$, the following orthogonal decomposition holds{\rm :}
\begin{equation}
\pform = \harmp \bigoplus {\rm Im}(\Delta)= \harmp \bigoplus d\Omega^{q-1}(M) \bigoplus \delta\Omega^{q+1}(M),
\end{equation}
where the orthogonality is taken with respect to the $L^2$ inner product on $\pform$.
Moreover, $\dim_\real (\harmp) =\dim_\real (H_{\rm dR}^q(M;\real))< \infty$ for closed manifolds,
where $H_{\rm dR}^q(M;\real)$ is the $q^{\text{th}}$ {\em de Rham} cohomology group of $M$.
\end{theorem}
	
In view of Theorem \ref{thm_hodge}, we can define the solution
operator ({\em i.e.}, the Green operator) to the Laplace-Beltrami on closed manifolds.
Let $\pi_H: \pform \mapsto \harmp$ stand for the canonical projection.
Then, for $\alpha \in \pform$, we set
\begin{equation}
G(\alpha)= \Delta^{-1}(\alpha-\pi_H\alpha).
\end{equation}
It can be shown that, if $T$ is a linear operator commuting with $\Delta$,
then it also commutes with $G$.
In particular, $d, \delta, \Delta$, and $\ast$ commute with $G$.
Moreover, $\Delta$ and $G$ are bounded linear operators on the Sobolev spaces.
To be explicit, $\Delta: {W^{k+2,p}}(M;\ptens) \mapsto W^{k,p}(M;\ptens)$
and $G: W^{k,p}(M;\ptens)\mapsto W^{k+2, p}(M;\ptens)$ are continuous
for each $k \in \mathbb{Z}$, $p\in(1,\infty)$, and $0 \leq q\leq n$.

\begin{theorem}\label{thm_main thm goemetric divcurl}
Let $(M,g)$ be an $n$-dimensional Riemannian manifold.
Let
$$
\{\omega^\epsilon\}, \{\tau^\epsilon\} \subset L^2_{\rm loc}(M;\ptens)
$$
be two families of  differential $q$--forms, for $0\leq q\leq n$,
such that
\begin{enumerate}
\item[\rm (i)]
$\omega^\epsilon \rightharpoonup \overline{\omega}$ weakly in $L^2$,
and $\tau^\epsilon \rightharpoonup \overline{\tau}$ weakly in $L^2$;
\item[\rm (ii)] there are compact subsets of the corresponding Sobolev spaces, $K_d$ and $K_\delta$, such that
\begin{equation*}
\begin{cases}
\{d\omega^\epsilon\}\subset K_d \Subset H^{-1}_{\rm loc}(M;\bigwedge^{q+1}T^*M),\\
\{\delta\tau^\epsilon\}\subset K_\delta \Subset H^{-1}_{\rm loc}(M;\bigwedge^{q-1}T^*M).
\end{cases}
\end{equation*}
\end{enumerate}
Then $\langle\omega^\epsilon, \tau^\epsilon\rangle$ converges to $\langle\overline{\omega}, \overline{\tau}\rangle$
in the sense of distributions:
\begin{equation*}
\int_M \langle\omega^\epsilon, \tau^\epsilon\rangle\psi\, {\rm d}V_g
\longrightarrow \int_M \langle\overline{\omega}, \overline{\tau}\rangle \psi\, {\rm d}V_g
\qquad \mbox{for any $\psi \in C^\infty_{c}(M)$}.
\end{equation*}
\end{theorem}

\allowdisplaybreaks
\begin{proof}
First of all, we can reduce the statement of the theorem only for compact manifolds.
Indeed, for any test function $\psi \in C^\infty_c(M)$, it suffices to assume that $\psi \geq 0$.
Otherwise, we may decompose it as
$\psi =\psi^+ - \psi^-$ and approximate $\psi^\pm$ by $C^\infty$ nonnegative functions respectively.
Then, without loss of generality, we can always replace $(\omega^\epsilon,\tau^\epsilon)$
by $(\sqrt{\psi}\omega^\epsilon,\sqrt{\psi}\tau^\epsilon)$,
and replace $M$ by any compact submanifold of $M$ containing the support of $\psi$.
Thus we may drop the test function $\psi$ and the subscripts ``loc'' in the function spaces:
More precisely, it suffices to take $M$ as a closed Riemannian manifold and establish the convergence:
$$
\int_M \langle\omega^\epsilon, \tau^\epsilon\rangle\, {\rm d}V_g
\longrightarrow \int_M \langle\overline{\omega}, \overline{\tau}\rangle \, {\rm d}V_g
\qquad\,\,\mbox{as $\e\to 0$},
$$
under the assumptions that $\{\omega^\epsilon\}, \{\tau^\epsilon\} \subset L^2 (M;\ptens)$,
$\{d\omega^\epsilon\}\subset K_d \Subset H^{-1}(M;\bigwedge^{q+1}T^*M)$, and
$\{\delta\tau^\epsilon\}\subset K_\delta \Subset H^{-1}(M;\bigwedge^{q-1}T^*M)$.

Now we are in the position of applying Theorem \ref{thm_abstract compensated compactness}.
For this purpose, we take $\hil= L^2 (M;\ptens)$, $Y=H^{-1}(M;\bigwedge^{q+1}T^*M)$, $Z=H^{-1}(M;\bigwedge^{q-1}T^*M)$, $\htt=H^{-1}(M;\ptens)$,
$S=d$, and $T=\delta$.
In this setting, $\{\omega^\epsilon\}$ and $\{\tau^\epsilon\}$  play the role of $\{u^\e\}$ and $\{v^\e\}$, respectively.
Conditions (Seq 1)--(Seq 2) of Theorem \ref{thm_abstract compensated compactness}
correspond precisely to conditions (i)--(ii), and condition (Op 1) of Theorem \ref{thm_abstract compensated compactness}
is verified by the cohomology chain condition $d\circ d=0$ and $\delta\circ\delta=0$.

Thus, we are left with checking the estimate in (Op 2), {\em i.e.},
for any $\alpha \in L^2(M;\ptens)$, there exits a constant $C>0$ such that
\begin{align}\label{check Op 2}
&\|\alpha\|_{L^2(M;\ptens)}\nonumber\\
&\leq\; C\Big(\|d\alpha\|_{H^{-1}(M;\bigwedge^{q+1}T^*M)}
+\|\delta\alpha\|_{H^{-1}(M;\bigwedge^{q-1}T^*M)}+\|\alpha\|_{H^{-1}(M;\bigwedge^{q}T^*M)}\Big).
\end{align}
To this end, we rely crucially on the Hodge decomposition theorem, Theorem \ref{thm_hodge},
as well as the standard elliptic estimate for the Laplace-Beltrami operator in the following form:
\begin{equation}\label{elliptic}
\|\omega\|_{L^2(M;\ptens)} \leq C\Big(\|\Delta\omega\|_{H^{-2}(M;\ptens)} + \|\omega\|_{H^{-2}(M;\ptens)}\Big)
\end{equation}
for arbitrary $\omega\in L^2(M;\ptens)$.

Now let us prove \eqref{check Op 2}. Given $\alpha \in L^2(M;\ptens)$,  we define $\beta:=G\alpha$,
which is equivalent to the decomposition: $\alpha=\pi_H \alpha + \Delta \beta$.
Applying the elliptic estimate \eqref{elliptic} to $\omega:=\pi_H \alpha$, we have
\begin{align*}
\|\pi_H \alpha\|_{L^2(M;\ptens)} \leq C\|\pi_H\alpha\|_{H^{-2}(M;\ptens)},
\end{align*}
where $C>0$ is a constant, independent of $\alpha$. As $H^{-2}(M;\ptens)$ is a Hilbert space and $\pi_H$ is a projection,
the Pythagorean law gives that $\|\pi_H \alpha\|_{H^{-2}(M;\ptens)} \leq \|\alpha\|_{H^{-2}(M;\ptens)}$.
Thus, in view of the compact embedding: $H^{-1}(M;\ptens)\hookrightarrow{H^{-2}(M;\ptens)}$ due to the Rellich lemma,
we conclude
\begin{equation}\label{1}
\|\pi_H \alpha\|_{L^2(M;\ptens)} \leq C\|\alpha\|_{H^{-1}(M;\ptens)}.
\end{equation}

Finally, we can bound $\Delta\beta$ in $L^2$. The bound follows from the following estimates:
\begin{align}\label{2}
&\|\Delta\beta\|_{L^2(M;\ptens)}\nonumber\\
&=\:  \|\Delta(G\alpha)\|_{L^2(M;\ptens)}\nonumber \\
&=\: \|G(\Delta\alpha)\|_{L^2(M;\ptens)} \nonumber\\
&\leq\:  C\Big(\|\Delta\alpha\|_{H^{-2}(M;\ptens)} + \|G(\Delta\alpha)\|_{H^{-2}(M;\ptens)} \Big)\nonumber\\
&\leq\: C\Big(\|(d\delta+\delta d)\alpha\|_{H^{-2}(M;\ptens)} + \|\alpha-\pi_H(\alpha)\|_{H^{-2}(M;\ptens)} \Big)\nonumber\\
&\leq\: C\Big(\|\delta\alpha\|_{H^{-1}(M;\bigwedge^{q-1}T^*M)}+\|d\alpha\|_{H^{-1}(M;\bigwedge^{q+1}T^*M)}+\|\alpha\|_{H^{-2}(M;\ptens)} \Big)\nonumber\\
&\leq\: C\Big(\|\delta\alpha\|_{H^{-1}(M;\bigwedge^{q-1}T^*M)}+\|d\alpha\|_{H^{-1}(M;\bigwedge^{q+1}T^*M)}
+\|\alpha\|_{H^{-1}(M;\bigwedge^{q}T^*M)}\Big),
\end{align}
where we have used the commutativity of $G$ and $\Delta$ in the third line, the elliptic estimate
in the fourth line with $\omega=G\alpha$ in \eqref{elliptic},
the definition of $\Delta$ and $G$ in the fifth line,
the Pythagorean theorem in the Hilbert space $H^{-2}(M;\ptens)$
in the sixth line, and the Rellich lemma in the final line.
Therefore, the estimate in \eqref{check Op 2}
has been obtained, and the proof is now complete.
\end{proof}

To conclude this section, we point out that some connections between the Hodge decomposition theorem
and compensated compactness have been observed by Robbin-Rogers-Temple in \cite{RRT87}
(also {\it cf.} Tartar \cite{Tar79}), and some generalized versions of the
div-curl lemma have been obtained by Kozono-Yanagisawa \cite{KY09}--\cite{KY13}, among others.
Our general functional-analytic compensated compactness theorem,
Theorem \ref{thm_abstract compensated compactness},
further provides such a connection in the abstract form.
In particular, our proof is essentially based upon the analysis
of operator $\sd:\yzstar\mapsto \yz$,
which is an analogue of the Laplace-Beltrami operator $\Delta$.

In fact,
a global intrinsic div-curl lemma,
more general than Theorem \ref{thm_main thm goemetric divcurl}, on Riemannian manifolds
can also be established,
which applies for the two sequences $\{\omega^\epsilon\}$ and $\{\tau^\epsilon\}$
lying in $L^r_{\text{\rm loc}}$ and $L^s_{\text{\rm loc}}$, where $1<r,s<\infty$ and $\frac{1}{r}+\frac{1}{s}=1$.

\begin{theorem}\label{thm_general goemetric divcurl}
Let $(M,g)$ be an $n$-dimensional Riemannian manifold. Let $\{\omega^\epsilon\} \subset L^r_{\rm loc}(M;\ptens)$
and $\{\tau^\epsilon\} \subset L^s_{\rm loc}(M;\ptens)$ be two families
of differential $q$--forms, for $0\leq q\leq n$, $1<r,s<\infty$, and $\frac{1}{r}+\frac{1}{s}=1$.
Suppose that
\begin{enumerate}
\item[\rm (i)]
$\omega^\epsilon \rightharpoonup \overline{\omega}$ weakly in $L^r$,
and $\tau^\epsilon \rightharpoonup \overline{\tau}$ weakly in $L^s$ as $\e\to 0${\rm ;}

\item[\rm (ii)] There are compact subsets of the corresponding Sobolev spaces, $K_d$ and $K_\delta$, such that
\begin{equation*}
\begin{cases}
\{d\omega^\epsilon\}\subset K_d \Subset W^{-1,r}_{\rm loc}(M;\bigwedge^{p+1}T^*M),\\[1mm]
\{\delta\tau^\epsilon\}\subset K_\delta \Subset W^{-1,s}_{\rm loc}(M;\bigwedge^{q-1}T^*M).
\end{cases}
\end{equation*}
\end{enumerate}
Then $\langle\omega^\epsilon, \tau^\epsilon\rangle$ converges to $\langle\overline{\omega}, \overline{\tau}\rangle$
in the sense of distributions{\rm :} For any $\psi \in C^\infty_{c}(M)$,
\begin{equation*}
\int_M \langle\omega^\epsilon, \tau^\epsilon\rangle\psi\, {\rm d}V_g
\longrightarrow \int_M \langle\overline{\omega}, \overline{\tau}\rangle \psi\, {\rm d}V_g
\qquad \mbox{as $\e\to 0$}.
\end{equation*}
\end{theorem}

To make this paper self-contained, we will present its proof
in the appendix.

\section{Global Weak Rigidity of the Gauss-Codazzi-Ricci Equations on Riemannian Manifolds}\label{section_GCR}

In this section, we establish the global weak rigidity of the GCR equations on Riemannian manifolds,
independent of the local coordinates.

\subsection{Global Weak Rigidity Theorem}

Our main result in this section is the following:

\begin{theorem}[Global Weak Rigidity of the GCR Equations]\label{thm_main theorem weak continuity}
Let $(M,g)$	be an $n$-dimensional manifold with $W^{1,p}$ metric for $p > 2$.
Let a sequence of operators $\{(\bep,\nep)\}$ satisfy
\begin{enumerate}
\item[\rm (i)] The tensor fields $\bep: \vf \times \vf \mapsto \vfn$
and the affine connections $\nep: \vf \times \vfn \mapsto \vfn$ are uniformly
bounded in $L^p_{\rm loc}$ with
	\begin{equation}
	\sup_{\epsilon>0} \big\{\|\bep\|_{L^p(K)} + \|\nep\|_{L^p(K)} \big\} \leq C_K
	\end{equation}
for a constant $C_K$ on any  $K\Subset M$ compact subsets, independent of $\epsilon${\rm ;}
\item[\rm (ii)]
$(\bep,\nep)$ are solutions of the GCR equations in the distributional sense{\rm :}
For any $X,Y,Z,W \in \vf$ and $\eta,\xi\in\vfn$,
\begin{eqnarray}
&& \langle B^\epsilon(X,W), B^\epsilon(Y,Z)\rangle -\langle B^\epsilon(Y,W), B^\epsilon(X,Z) \rangle =-R(X,Y,Z,W),
\label{eqn_O1 epsilon}\\
&& \langle [S^\epsilon_\eta, S^\epsilon_\xi]X, Y\rangle = R^\perp (X,Y,\eta,\xi),\label{eqn_O2 epsilon}\\
&&\bar{\na}_Y B^\epsilon(X,Z)-\bar{\na}_X B^\epsilon(Y,Z) = 0 \label{eqn_O3 epsilon}
\end{eqnarray}
in the distributional sense,
where
$S^\epsilon$ is the shape operator corresponding to $B^\epsilon$
{\rm (}Note that $\nep$ is implicit in the above equations{\rm )}.
\end{enumerate}
Then, after passing to a subsequence, $\{(\bep,\nep)\}$ converges weakly in $L^p$ to a pair $(B,\nap)$ that is still
a weak solution
of the GCR equations \eqref{2.14a}--\eqref{2.16a}.
\end{theorem}

This result can be regarded as a global version on Riemannian manifolds of Theorem 3.3 in Chen-Slemrod-Wang \cite{CSW10}.
In this paper, both the statement and the proof for this weak rigidity theorem are global, intrinsic,
independent of the local coordinates of
the Riemannian manifolds, which offers further geometric insights into the GCR equations.

\begin{remark}
Theorem {\rm \ref{thm_main theorem weak continuity}} for the exact solutions
can be extended to the weak rigidity of approximate solutions $(\bep,\nep)$ of the GCR equations.
More precisely, instead of \eqref{eqn_O1 epsilon}--\eqref{eqn_O3 epsilon} in {\rm (ii)}, let $(\bep,\nep)$
solve the following approximate GCR equations in the distributional sense{\rm :}
For any $X,Y,Z,W \in \vf$ and $\eta,\xi\in\vfn$,
\begin{eqnarray}
&& \langle B^\epsilon(X,W), B^\epsilon(Y,Z)\rangle -\langle B^\epsilon(Y,W), B^\epsilon(X,Z) \rangle  +R(X,Y,Z,W)
  =O_1^\epsilon, \label{eqn_O1 epsilon-a}\\
&& \langle [S^\epsilon_\eta, S^\epsilon_\xi]X, Y\rangle - R^\perp (X,Y,\eta,\xi)=O_2^\epsilon,\label{eqn_O2 epsilon-a}\\
&&\bar{\na}_Y B^\epsilon(X,Z) - \bar{\na}_X B^\epsilon(Y,Z) = O_3^\epsilon, \label{eqn_O3 epsilon-a}
\end{eqnarray}
such that $\lim_{\epsilon \rightarrow 0} O_i^\epsilon=0$ in $W^{-1, r}_{\rm loc}$ for some $r>1$.
Then, after passing to a subsequence, $\{(\bep,\nep)\}$ converges weakly in $L^p$ to a pair $(B,\nap)$ that is
a weak solution of the GCR equations \eqref{2.14a}--\eqref{2.16a}.
\end{remark}

\subsection{First Formulation: Identification
of the Tensor Fields with Special Div-Curl Structure on Riemannian Manifolds}
	
Now we begin the proof of Theorem \ref{thm_main theorem weak continuity}.
First we seek tensor fields with special {\it div-curl structure} for the GCR equations on manifolds.
	
Recall our previous convention:
$X,Y,Z,\dots\in\vf$ and $\xi,\eta,\beta,\dots\in\vfn$.
For each fixed $(Z,\eta,\xi)$, define the tensor
fields $\vb_{Z,\eta}, \vnab_{\xi,\eta}: \vf \times \vf \mapsto \vf$
and $1$--forms $\ob_{Z,\eta}, \onab_{\xi,\eta} \in \Omega^1(M)=\Gamma(T^*M)$ by
\begin{eqnarray}
&&\vb_{Z,\eta}(X,Y):= B(X,Z,\eta)Y-B(Y,Z,\eta)X, \label{def/eqn_v b}\\
&&\vnab_{\xi,\eta}(X,Y):=\langle \nap_Y \xi, \eta\rangle X - \langle \nap_X \xi, \eta\rangle Y,\label{def/eqn_v nabla}\\
&&\ob_{Z,\eta}:= -B(\bullet, Z, \eta), \label{def/eqn_omega b}\\
&&\onab_{\xi,\eta}:= \langle \nap_{\bullet} \xi, \eta\rangle.\label{def/eqn_omega nabla}
\end{eqnarray}
To avoid further notations, we denote the vector fields canonically isomorphic
to $\ob_{Z,\eta}$ and $\onab_{\xi,\eta}$ (via $\sharp$ and $\flat$)
by the same symbols.
	
Our geometric picture is as follows:
The $V$--tensors take two tangential vector fields $(X,Y)$ to a vector field spanned
by $X$ and $Y$, which are anti-symmetric in $(X,Y)$.
Thus, $\vb_{Z,\eta}(X,Y)$ and $\vnab_{\xi,\eta}(X,Y)$ are precisely the {\it rate of rotations}
of $\ob_{Z,\eta}$ and $\onab_{\xi,\eta}$ in the $2$--planes generated by $(X,Y)$.

The $1$--forms $(\ob,\onab)$ are simply contractions of $(B,\nap)$, and the tensors $(V^{(B)}, V^{(\na^\perp)})$
can be obtained by applying the $\Omega$--tensors to the $2$--Grassmannian of  $TM$.

Our first formulation concerns the {\it divergence} of $V$
in Eqs. \eqref{def/eqn_v b}--\eqref{def/eqn_v nabla}
and the {\it curl} (as $2$-forms) of $\Omega$ in
Eqs. \eqref{def/eqn_omega b}--\eqref{def/eqn_omega nabla}:

\begin{lemma}[First Formulation]\label{lemma:4.1} The divergence of $V$ and the curl of $\Omega$
can be reformulated as
\begin{eqnarray}
&&{\rm div}\big(\vb_{Z,\eta}(X,Y)\big)
  = Y B(X,Z,\eta) - X B(Y,Z,\eta)+B(X,Z,\eta){\rm div}Y-B(Y,Z,\eta){\rm div X},\qquad\quad
  \label{eqn_div VB}\\
&&{\rm div}\big(\vnab_{\xi,\eta}(X,Y)\big)
 =-Y\langle \nap_X\xi,\eta\rangle +X\langle \nap_Y\xi,\eta\rangle
   +\langle\nap_Y\xi,\eta\rangle{\rm div}X-\langle\nap_X\xi,\eta\rangle{\rm div}Y,
    \label{eqn_div V NABLA}\\
&& d\big(\ob_{Z,\eta}\big)(X,Y) = YB(X,Z,\eta) - XB(Y,Z,\eta) + B([X,Y],Z,\eta),\label{eqn_curl Omega B}\\
&& d \big( \onab_{\xi,\eta}\big)(X,Y) =  -Y\langle \nap_X\xi,\eta\rangle
+X\langle \nap_Y\xi,\eta\rangle - \langle \nap_{[X,Y]}\xi,\eta\rangle.\label{eqn_curl Omega NABLA}
\end{eqnarray}
\end{lemma}

\begin{proof}
To show the first two identities,
we use Eq.\,\eqref{eqn_def of divergence by lie derivative}
to express
the divergence in terms of the Lie derivative, which is further computed
from Cartan's formula:
\begin{equation}\label{cartan's formula, Lx=di+id}
\lie_X=d\circ \iota_X + \iota_X \circ d,
\end{equation}
where $\iota$ is the interior multiplication.
Indeed, in local coordinates, we write $X=X^i\p_i$, $Y=Y^j\p_j$,
and $dV_g = dx^1 \wedge\ldots\wedge dx^n$. Then
\begin{equation*}
\begin{cases}
\iota_X \, dV_g = (-1)^i X^i dx^1\wedge\ldots\wedge \widehat{dx^i}\wedge \ldots\wedge dx^n,\\
\iota_Y \, dV_g = (-1)^j Y^j dx^1\wedge\ldots\wedge \widehat{dx^j}\wedge \ldots\wedge dx^n,
\end{cases}
\end{equation*}
where $\widehat{\cdots}$ denotes the omission of the corresponding term. As a result,
\begin{equation}\label{add: div 1}
\ast \big(d ( \iota_X \, dV_g)\big) = {\rm div}\, X,\qquad\,\, \ast \big(d ( \iota_Y \, dV_g)\big)
= {\rm div}\,Y,
\end{equation}
as $\p_i X^i = {\rm div}\, X$ and $\ast dV_g = 1$.
On the other hand, for any $f:M\rightarrow \R$, we have
\begin{equation}\label{add: div 2}
\begin{cases}
\ast \big(df \wedge (\iota_Y \, dV_g)\big) = \ast (\p_j f Y^j dV_g) = Yf,\\
\ast \big(df \wedge (\iota_X \, dV_g)\big) = \ast (\p_i f X^i dV_g) = Xf,
\end{cases}
\end{equation}
where the vector fields $X,Y$ are identified with the directional derivatives.
Therefore, we conclude
\begin{align*}
{\rm div}\, V^{(B)}_{Z,\eta}(X,Y) &= \ast \big(\mathcal{L}_{V^{(B)}_{Z,\eta}(X,Y)} dV_g\big)\\
&= \ast \big(d \iota_{V^{(B)}_{Z,\eta}(X,Y)} \, dV_g \big) \\
&= \ast \big\{d\big(B(X,Z,\eta)\iota_Y\, dV_g - B(Y,Z,\eta)\iota_X \,dV_g\big) \big\}\\
&= \ast \big\{d\big(B(X,Z,\eta)\wedge (\iota_Y dV_g)\big)\big\} + B(X,Z,\eta) \ast \big(d (\iota_Y \, dV_g)\big) \\
&\qquad- \ast \big\{d\big(B(Y,Z,\eta)\wedge (\iota_X dV_g)\big)\big\} - B(Y,Z,\eta) \ast \big(d (\iota_X \, dV_g)\big)\\
&= Y B(X,Z,\eta) - X B(Y,Z,\eta)+B(X,Z,\eta){\rm div}\,Y-B(Y,Z,\eta){\rm div}\, X.
\end{align*}
We use Cartan's formula \eqref{cartan's formula, Lx=di+id} in the first line,
$d(dV_g)=0$ in the second,  the definition of $V^{(B)}$ in the third,
the definition of $d$ in the fourth, and  we apply Eqs.\,\eqref{add: div 1}--\eqref{add: div 2}
in the last line, so that Eq.\,\eqref {eqn_div VB} is established. Moreover,
the proof of Eq.\,\eqref{eqn_div V NABLA} is analogous.

\smallskip
For the {\it generalised curl} ({\it i.e.}, $d$),
recall  the identity from \S 2.2.1:
\begin{equation*}
d\alpha(X,Y):= X\alpha(Y)-Y\alpha(X) - \alpha([X,Y])
\qquad \mbox{for $\alpha \in \Omega^1(M)$}.
\end{equation*}
Applying the above to $\alpha=\ob_{Z,\eta}$, we have
\begin{align*}
d\ob_{Z,\eta}(X,Y) &= X \ob_{Z,\eta}(Y) - Y\ob_{Z,\eta}(X) - \ob_{Z,\eta}([X,Y])\\
&= -X B(Y,Z,\eta) + YB(X,Z,\eta) + B([X,Y],Z,\eta),
\end{align*}
which is Eq.\, \eqref{eqn_curl Omega B}.
The proof of Eq.\,\eqref{eqn_curl Omega NABLA} is analogous.
This completes the proof.
\end{proof}

As a remark, it is crucial to recognize
that $\delta\big(\vb_{Z,\eta}(X,Y)\big)$ and $d\big(\ob_{Z,\eta}\big)(X,Y)$,
as well as $\delta\big(\vnab_{\xi,\eta}(X,Y)\big)$ and $d\big( \onab_{\xi,\eta}\big)(X,Y)$,
are essentially the same.
They only differ by a zero-th order term involving the Lie bracket $[X,Y]$.
This observation turns out to be crucial in the proof of Theorem \ref{thm_main theorem weak continuity}.

\subsection{Second Formulation: The Div-Curl Structure of the GCR Equations}

We now express the GCR equations in another form,
which is suitable for applying the intrinsic div-curl lemma,
{\em i.e.}, Theorem \ref{thm_main thm goemetric divcurl}.
To achieve this, we employ the geometric
quantities $\vb, \ob,\vnab$, and $\onab$,
in the reduced GCR equations
\eqref{eqn_gauss}--\eqref{eqn_reduced ricci}
in Theorem \ref{thm_reduced GCR} to obtain

\begin{lemma}[Second Formulation]\label{lemma:4.2}
The Gauss, Codazzi, and Ricci equations are equivalent
to the following equations, respectively{\rm :}
\begin{eqnarray}
&&\sum_\eta \langle\vb_{Z,\eta}(X,Y), \ob_{W,\eta} \rangle=-R(X,Y,Z,W),
\label{eqn_Gauss Eqn in second computation}\\
&&d(\ob_{Z,\eta})(X,Y) + \sum_\beta \langle \vnab_{\eta,\beta}(X,Y),\ob_{Z,\beta}\rangle + E(B)=0,
\label{eqn_Codazzi Eqn in second computation}\\
&& d(\onab_{\xi,\eta})(X,Y) + \sum_{\beta} \langle \vnab_{\eta,\beta}(X,Y), \onab_{\xi,\beta} \rangle
= \sum_Z\langle\vb_{Z,\xi}(X,Y),\ob_{Z,\eta}\rangle,
\label{eqn_Ricci Eqn in second computation}
\end{eqnarray}
where $E(B):=B(Y,\na_XZ,\eta)-B(X,\na_YZ,\eta)$
is linear in $B$,
and all the summations are at most countable and locally finite.
\end{lemma}

\begin{proof}  We divide the proof into four steps.

{\bf 1.} We start with Eq. \eqref{eqn_Gauss Eqn in second computation}.
Now consider the following spanning set of normal vector fields:
\begin{equation}\label{eqn_set S1}
\mathcal{S}_1:=\big\{\{\eta_j\}_{j=1}^k\subset \vfn: |\eta_j|=1, \text{span}\{\eta_j\}_{j=1}^k = \vfn\big\},
\end{equation}
where $n+k$ is the dimension of the target space of the isometric immersion.

Since metric $g$ is in $W^{1,p}$ on $M$, and
the second fundamental form and the normal connection are well-defined in $L^p$,
then $\mathcal{S}_1$ exists {\it a.e.} on $M$.
In view of the remark above, we can make the following computation in the distributional sense:
\begin{align}
&\langle B(X,W), B(Y,Z)\rangle-\langle B(Y,W), B(X,Z) \rangle \nonumber \\
&= \sum_{\eta\in\mathcal{S}_1}\big(B(X,W,\eta) B(Y,Z,\eta) - B(Y,W,\eta) B(X,Z,\eta)\big) \nonumber\\
&=\sum_{\eta\in\mathcal{S}_1} B\big(B(X,W,\eta)Y - B(Y,W,\eta)X,Z,\eta\big)\nonumber\\
&= \sum_{\eta\in\mathcal{S}_1} \langle \vb_{W,\eta}(X,Y), \ob_{Z,\eta} \rangle,
\end{align}
where the last equality follows from the definition of $\vb$ and $\ob$.
Combining the above computation with \eqref{eqn_gauss}, we conclude
Eq. \eqref{eqn_Gauss Eqn in second computation}.

\smallskip
{\bf 2.} Next, we establish Eq. \eqref{eqn_Codazzi Eqn in second computation}.
Let us start from the Codazzi equation \eqref{eqn_reduced codazzi} in the form of Theorem \ref{thm_reduced GCR}:
\begin{align*}
& XB (Y,Z,\eta) - YB(X,Z,\eta) \\
&= B([X,Y],Z,\eta)	- B(X,\na_YZ,\eta)- B(X,Z,\nap_Y\eta)  + B(Y,\na_XZ,\eta) + B(Y,Z,\nap_X\eta).
\end{align*}
By Eq. \eqref{def/eqn_omega b} in the first formulation,
we can transform the above equation to the following:
\begin{equation}\label{interm eqn 5}
d\big(\ob_{Z,\eta}\big)(X,Y) - B(X,\na_YZ,\eta) - B(X,Z,\nap_Y\eta) + B(Y,\na_XZ,\eta) + B(Y,Z,\nap_X\eta)=0.
\end{equation}
Moreover, observe that
\begin{align}\label{interm eqn_6}
B(Y,Z,\nap_X\eta) - B(X,Z,\nap_Y\eta)
&= \sum_{\beta\in\mathcal{S}_1} \big(B(Y,Z,\beta)\langle\nap_X\eta,\beta\rangle - B(X,Z,\beta)\langle\nap_Y\eta,\beta\rangle\big) \nonumber\\
&=\sum_{\beta\in\mathcal{S}_1} B(\langle\nap_X\eta,\beta\rangle Y-\langle \nap_Y \eta,\beta\rangle X, Z,\beta)\nonumber\\
&=: \sum_{\beta\in\mathcal{S}_1} \langle\vnab_{\eta,\beta}(X,Y), \ob_{Z,\beta}\rangle.
\end{align}
Thus, putting  Eqs. \eqref{interm eqn 5}--\eqref{interm eqn_6} together,
we arrive at Eq. \eqref{eqn_Codazzi Eqn in second computation}. It expresses the Codazzi equation in terms of $V$ and $\Omega$.

{\bf 3}. Finally, we prove Eq. \eqref{eqn_Ricci Eqn in second computation}.
The Ricci equation \eqref{eqn_reduced ricci} in Theorem \ref{thm_reduced GCR} reads
\begin{align*}
&X\langle \nap_Y \xi,\eta\rangle - Y \langle \nap_X\xi,\eta \rangle\\
&= \langle \nap_{[X,Y]}\xi,\eta \rangle - \langle \nap_X \xi, \nap_Y \eta\rangle + \langle \nap_Y\xi, \nap_X \eta\rangle \\
&\quad + B (\bar{\na}_X \xi-\nap_X\xi, Y, \eta) - B (\bar{\na}_X\eta-\nap_X\eta, Y, \xi).
	\end{align*}
In light of Eq. \eqref{eqn_curl Omega NABLA} in the first formulation, we have
\begin{align}\label{interm eqn 1}
&d(\ob_{\xi,\eta})(X,Y) + \langle \nap_X \xi, \nap_Y \eta\rangle - \langle \nap_Y\xi, \nap_X \eta\rangle \nonumber \\
&\, = B (\bar{\na}_X \xi-\nap_X\xi, Y, \eta) - B (\bar{\na}_Y\eta-\nap_Y\eta, X, \xi).
\end{align}
The last two terms on the left-hand side can be expressed as
\begin{align}\label{interm eqn 2}
\langle \nap_X \xi, \nap_Y \eta\rangle - \langle \nap_Y\xi, \nap_X \eta\rangle
&= \sum_{\beta \in \mathcal{S}_1} \big(\langle \nap_X \xi,\beta \rangle \langle \beta, \nap_Y\eta\rangle
  - \langle\nap_Y\xi, \beta\rangle \langle\beta, \nap_X\eta\rangle\big) \nonumber\\
&= \sum_{\beta \in \mathcal{S}_1} \langle \nap_{\langle\nap_Y\eta,\beta\rangle X
- \langle\nap_X\eta,\beta\rangle Y} \xi ,\beta \rangle \nonumber\\
&=: \sum_{\beta \in \mathcal{S}_1}  \langle\vnab_{\eta,\beta} (X,Y), \onab_{\xi,\beta} \rangle,
\end{align}
by using the definition of $\vnab$ and $\onab$ in Eqs. \eqref{def/eqn_v nabla}
and \eqref{def/eqn_omega nabla}.
	
To deal with the right-hand side of Eq. \eqref{interm eqn 1},
we temporarily assume the existence of the tangential spanning set of unit vector fields:
\begin{equation}\label{eqn_def of S2}
\mathcal{S}_2:= \big\{\{Z_j\}_{j=1}^n\subset\vf\,:\, |Z_j|\equiv 1, \text{ span}\{Z_j\}_{j=1}^n=\vf\big\}.
\end{equation}
In this case, one can compute the right-hand side of equation
\eqref{interm eqn 1}:
\begin{align}\label{interm eqn 3}
&B (\bar{\na}_X \xi- \nap_X\xi, Y, \eta) - B (\bar{\na}_X\eta -\nap_X\eta, Y, \xi)\nonumber\\
&= \sum_{Z\in\mathcal{S}_2} \big(B(\langle\bar{\na}_X\xi-\nap_X\xi, Z \rangle Z, Y,\eta)
- B(\langle\bar{\na}_X\eta-\nap_X\eta, Z \rangle Z, Y,\xi)\big)\nonumber\\
&= \sum_{Z\in\mathcal{S}_2} \big(B(Z,X,\eta)B(Z,Y,\xi) - B(Z,X,\xi)B(Z,Y,\eta)\big).
\end{align}
Indeed, to obtain the last equality, recall that $B$ is symmetric in the first two arguments.
Then, using Eqs. \eqref{eqn_shape operator}--\eqref{eqn_shape operator and B}, we have
\begin{equation*}
\langle\bar{\na}_X\xi-\nap_X\xi, Z \rangle = -\langle S_\xi X,Z\rangle = -B(X,Z,\xi),
\end{equation*}
where $S$ is the shape operator. The other term follows similarly.

On the other hand, by Eqs. \eqref{def/eqn_v b} and \eqref{def/eqn_omega b}, we have
\begin{align}\label{interm eqn 4}
B(Z,X,\eta)B(Z,Y,\xi) - B(Z,X,\xi)B(Z,Y,\eta)
&=-B(B(X,Z,\xi)Y - B(Y,Z,\xi)X, Z, \xi)\nonumber\\
&=: \langle \vb_{Z,\xi}(X,Y), \ob_{Z,\eta} \rangle.
\end{align}
Thus, putting \eqref{interm eqn 1}--\eqref{interm eqn 2} and \eqref{interm eqn 3}--\eqref{interm eqn 4} together,
we obtain Eq. \eqref{eqn_Ricci Eqn in second computation},
provided that the spanning set $\mathcal{S}_2 \subset \vf$ is well-defined.

\smallskip
{\bf 4.} Unfortunately, we cannot take the existence of $\mathcal{S}_2$ for granted. For example,
on $M=\mathbb{S}^{2m}$, the Hairy Ball theorem shows that any $Z\in\vf$ must vanish at some point.
To overcome this difficulty, we may resort to a partition of unity argument.
Let $\mathcal{A}=\{U_\alpha:\alpha\in\mathcal{I}\}$
be an atlas for $M$,
and let $\{\rho_\alpha:\alpha\in\mathcal{I}\}$ be a smooth partition of unity subordinate to $\mathcal{A}$
({\it cf.} \S \ref{Section 2}).
On each $U_\alpha$, the spanning set in the form of $\mathcal{S}_2$ exists,
since $U_\alpha$ is diffeomorphic to $\real^n$. We write
\begin{equation*}
\mathcal{S}_2^\alpha =\big\{\{Z^\alpha_j\}_{j=1}^n\subset\Gamma(TU_\alpha)\; :\;
|Z^\alpha_j|\equiv 1, \text{span}\{Z^\alpha_j\}_{j=1}^n=\Gamma(TU_\alpha) \big\}.
\end{equation*}
Define
\begin{equation}
\widetilde{\mathcal{S}_2}= \big\{Z:=\sum_{\alpha\in\mathcal{I}} \rho_\alpha Z^\alpha\chi_{\text{supp}(\rho_\alpha)}\; :\;
Z^\alpha \in \mathcal{S}_2^\alpha, \alpha\in\mathcal{I}\big\}.
\end{equation}
Then we can take $\widetilde{\mathcal{S}_2}$ in place of $\mathcal{S}_2$
so that all the preceding arguments for the Ricci equation
pass through.
This completes the proof.
\end{proof}

\subsection{Proof of Theorem \ref{thm_main theorem weak continuity}}

We can now prove the global weak rigidity theorem, Theorem \ref{thm_main theorem weak continuity},
for the GCR equations on Riemannian manifolds.
In fact, the main ingredients of the proof have been provided in the previous
two formulations, which express the GCR equations in the form
suitable for employing
Theorem \ref{thm_main thm goemetric divcurl}.

\begin{proof}[Proof of Theorem {\rm \ref{thm_main theorem weak continuity}}]
The proof consists of three steps.

\smallskip
{\bf 1.} By taking the orientable double cover when necessary,
we may assume that $M$ is orientable.
Then we can employ Theorem
\ref{thm_main thm goemetric divcurl} on $M$.
Moreover, by Theorems \ref{thn_GCR}--\ref{thm_reduced GCR},
$(\bep,\nep)$ are solutions of Eqs. \eqref{eqn_gauss}--\eqref{eqn_reduced ricci}
in the distributional sense.

\smallskip
{\bf 2.} Consider our second formulation.
The zero-th order terms of
Eqs. \eqref{eqn_Gauss Eqn in second computation}--\eqref{eqn_Ricci Eqn in second computation}
contain linear combinations of the following quadratic nonlinear forms:
\begin{equation}\label{four terms}
\begin{aligned}
&\langle \vbe_{W,\eta}(X,Y), \obe_{Z,\eta}\rangle, \quad \langle\vnabe_{\eta,\beta}(X,Y),\onabe_{\xi,\beta}\rangle,\\[2mm]
&\langle \vbe_{Z,\xi}(X,Y),\obe_{Z,\eta} \rangle, \quad \langle\vnabe_{\eta,\beta}(X,Y),\obe_{Z,\beta}\rangle,
\end{aligned}
\end{equation}
and the linear term $E(B^\epsilon)$ in $B^\e$.
Clearly, $E(B^\epsilon)\rightarrow E(B)$ in $\dis$, owing to the linearity of $B$
and the uniform $L^p$ boundedness of $\{\bep\}$.

For the four terms in \eqref{four terms}, by the hypotheses of Theorem \ref{thm_main theorem weak continuity},
we find that
\begin{equation*}
\{\vbe(X,Y), \obe,  \vnabe(X,Y), \onabe\}
\end{equation*}
are uniformly bounded in $L^p_{\rm loc}$.
As a consequence, the four quadratic terms in \eqref{four terms} are uniformly bounded
in $L^{p/2}_{\rm loc}$ for $p>2$.

Then, by Eqs. \eqref{eqn_Codazzi Eqn in second computation}--\eqref{eqn_Ricci Eqn in second computation},
we know that $d(\onabe_{\xi,\eta})(X,Y)$ and $d(\obe_{Z,\eta})(X,Y)$ are uniformly bounded in $L^{p/2}_{\rm loc}$,
so that they are compact at least in $W^{-1, p'}_{\rm loc}$ for some $p'\in (1, 2)$ by the Sobolev embedding.
On the other hand, since
$\onabe_{\xi,\eta}$ and $\obe_{Z,\eta}$ are uniformly bounded in $L^p_{\rm loc}$,
$d(\onabe_{\xi,\eta})(X,Y)$ and $d(\obe_{Z,\eta})(X,Y)$ are uniformly bounded in $W^{-1,p}_{\rm loc}, p>2$.
Then, by interpolation,
we conclude that
$$
\mbox{$\big\{d(\onabe_{\xi,\eta})(X,Y)$, $\,\, d(\obe_{Z,\eta})(X,Y)\big\}$ $\qquad\,\,\,$
are pre-compact subsets of $H^{-1}_{\rm loc}(M)$.}
$$

Furthermore, Eqs. \eqref{eqn_div VB}--\eqref{eqn_curl Omega NABLA}
of the first formulation lead to the following identities:
\begin{eqnarray*}
&&{\rm div}\big(\vbe_{Z,\eta}(X,Y)\big)
= d(\onabe_{Z,\eta})(X,Y)-\bep([X,Y],Z,\eta)+\bep(X,Z,\eta){\rm div}Y-\bep(Y,Z,\eta){\rm div}X,\qquad\quad\\
&&{\rm div}\big(\vnabe_{\xi,\eta}(X,Y)\big) = d(\onabe_{\xi,\eta})(X,Y) + \langle \nep_{[X,Y]}\xi,\eta \rangle
   +\langle\nep_Y\xi,\eta\rangle{\rm div}X-\langle\nep_X\xi, \eta\rangle{\rm div}Y.
\end{eqnarray*}
Since $(\bep,\nep)$ are uniformly bounded in $L^p_{\rm loc}, p>2$,
again by the Sobolev embeddings, they are compact in $H^{-1}_{\rm loc}$.
Since the divergence operator equals $\delta$ (the adjoint of $d$) modulo a sign,
we conclude that
$$
\mbox{$\big\{\delta\big(\vbe_{Z,\eta}(X,Y)\big)$, $\,\,\delta\big(\vnabe_{\xi,\eta}(X,Y)\big)\big\}$ $\quad\,\,$
are pre-compact sets of $H^{-1}_{\rm loc}(M)$.}
$$

{\bf 3.}
Since $(\bep,\nep)$ are uniformly bounded in $L^p_{\rm loc}, p>2$, after passing to subsequences,
they converge to some $(B,\na^\perp)$ weakly in $L^p_{\rm loc}$.
Combining the arguments in Step 2 with Theorem \ref{thm_main thm goemetric divcurl},
we obtain the following convergence in the distributional sense:
\begin{eqnarray*}
&&\langle \vbe_{W,\eta}(X,Y), \obe_{Z,\eta}\rangle \longrightarrow  \langle \vb_{W,\eta}(X,Y), \ob_{Z,\eta}\rangle,\\
&&\langle\vnabe_{\eta,\beta}(X,Y),\onabe_{\xi,\beta}\rangle \longrightarrow\langle\vnab_{\eta,\beta}(X,Y),\onab_{\xi,\beta}\rangle,\\
&&\langle \vbe_{Z,\xi}(X,Y),\obe_{Z,\eta} \rangle \longrightarrow \langle \vb_{Z,\xi}(X,Y),\ob_{Z,\eta} \rangle,\\
&&\langle\vnabe_{\eta,\beta}(X,Y),\obe_{Z,\beta}\rangle \longrightarrow\langle\vnab_{\eta,\beta}(X,Y),\ob_{Z,\beta}\rangle.
\end{eqnarray*}

Therefore, in Eqs. \eqref{eqn_Gauss Eqn in second computation}--\eqref{eqn_Ricci Eqn in second computation}
in the second formulation, we can pass the limits as $\epsilon \rightarrow 0$.
This guarantees that the weak limit $(B,\na^\perp)$ is still a weak solution of the GCR system.
This completes the proof.
\end{proof}

To conclude this section, we remark that Eqs. (2.1)--(2.3) in \cite{CSW10},
which are the Gauss, Codazzi, and Ricci equations in local coordinates,
can be seen directly from our global formulations
in Theorems \ref{thn_GCR}--\ref{thm_reduced GCR}.

Indeed,  write $\{\partial_i:1\leq i \leq n\}$ as a local coordinate system on the tangent bundle $TM$,
and $\{\partial_\alpha: {n+1} \leq \alpha \leq {n+k}\}$ for a local coordinate system on the normal bundle $TM^\perp$.
To write the GCR equations in the local coordinates, we define
\begin{equation}\label{local h, kappa}
h_{ij}^\alpha := \langle B(\partial_i, \partial_j), \partial_\alpha \rangle,
\qquad \kappa_{i\beta}^\alpha := \langle \na^\perp_{\partial_\alpha} \partial_i, \partial_\beta \rangle,
\end{equation}
and substitute the unit vector fields $\partial_i$ (or $\partial_\alpha$) in place of $X$ (or $\xi$,  respectively),
in the GCR equations in Theorems \ref{thn_GCR}--\ref{thm_reduced GCR}.
We solve for $\{h_{ij}^\alpha,\kappa_{i\beta}^\alpha\}$,
namely the second fundamental form and the normal affine connection written coordinate-wise.
In this way, we recover Eqs. (2.1)--(2.3) in \cite{CSW10}:
\begin{eqnarray}
&& h^\alpha_{ik} h^\alpha_{jl} - h^\alpha_{il} h^\alpha_{jk}
= R_{ijkl}, \label{local gauss}\\
&&\frac{\partial h^\alpha_{lj}}{\partial x^k} - \frac{\partial h^\alpha_{kj}}{\partial x^l} + \Gamma^m_{lj}h^\alpha_{km}
 - \Gamma^m_{kj}h^\alpha_{lm} + \kappa^\alpha_{k\beta}h^\beta_{lj} - \kappa^\alpha_{l\beta}h^\beta_{kj} = 0,\label{local codazzi}\\
&&\frac{\partial\kappa^\alpha_{l\beta}}{\partial x^k} - \frac{\partial \kappa^\alpha_{k\beta}}{\partial x^l}
- g^{mn}\big(h^\alpha_{ml}h^\beta_{kn}-h^\alpha_{mk}h^\beta_{ln}\big)
+ \kappa^\alpha_{k\gamma}\kappa^\gamma_{l\beta}
-\kappa^\alpha_{l\gamma}\kappa^\gamma_{k\beta}=0,\label{local ricci}
\end{eqnarray}
where $R_{ijkl}:=R(\partial_i,\partial_j,\partial_k,\partial_l)$, and the Christoffel symbols are defined as usual by
\begin{equation*}
\Gamma^k_{ij}:=\frac{1}{2}g^{kl} \big(\partial_i g_{jl} + \partial_j g_{il} - \partial_l g_{ij}\big).
\end{equation*}
Therefore, Theorem \ref{thm_main theorem weak continuity} for the weak rigidity of the GCR equations
is a global, intrinsic version of the local weak rigidity result of Theorem 3.3 in \cite{CSW10}.

\section{Isometric Immersions, the Cartan Formalism, and the GCR Equations on Manifolds with Lower Regularities}\label{section_equivalence}

In \S 4,  we have proved the global weak rigidity of the GCR equations on Riemannian manifolds
with lower regularity. Then the next natural question is about its geometric implications.
We address this question with two interrelated goals:
\begin{enumerate}
\item[(i)] Global isometric immersions
for simply-connected manifolds with lower regularity are constructed
from the GCR equations. As remarked in the introduction,
this is related to the {\it realization problem} in elasticity.
\item[(ii)]
The global weak rigidity of isometric immersions in turn
provides crucial insights to the global weak rigidity of the GCR equations.
This observation will lead to an alternative proof of Theorem \ref{thm_main theorem weak continuity}; {\it cf.} \S 7.1.
\end{enumerate}

\subsection{From PDEs to Geometry: An Equivalence Theorem}

We now address the central question raised above.
First, it should be noticed that the results in \S \ref{section_GCR} are
essentially {\it PDE-theoretic}.
Despite the geometric -- global and intrinsic -- nature of the formulation
and proof of Theorem \ref{thm_main theorem weak continuity},
we have only analyzed the GCR equations \emph{per se},
but have not referred to their geometric origin, {\it i.e.},
the isometric immersion problem.
One would expect that, providing that our formulation is natural,
the weak rigidity of isometric immersions and the weak rigidity of the GCR equations should
be essentially the same problem. We now formalize this observation and prove it in mathematical rigor.

We point out that the {\it realization problem}, {\it i.e.}, the construction of isometric immersions
from the GCR equations, has been investigated in the recent years; see
Ciarlet-Gratie-Mardare \cite{Cia08}, Mardare \cite{Mar03}--\cite{Mar07}, Szopos \cite{Szo08},
and the references cited therein.
These previous results, which solve the {\it realization problem} locally, can be summarized
in the following theorem.

\begin{theorem}
\label{proposition_mardare, szopos}
Let $U\subset \real^n$ be a
simply-connected open set.
Suppose that
the symmetric matrix fields $\{g_{ij}\}\in W^{1,p}_{\rm loc}(U;O(n))$ and
$\{h_{ij}=(h^{\alpha}_{ij})\}_{n+1\leq \alpha \leq n+k}\subset L^p_{\rm loc}(U;O(n))$,
and the anti-symmetric matrix field
$\{\kappa_{ij}=(\kappa^\alpha_{ij})\}_{n+1\leq \alpha \leq n+k} \subset  L^p_{\rm loc}(U;\mathfrak{so}(n))$
prescribed on $U$
satisfy the GCR equations \eqref{local gauss}--\eqref{local ricci}
in local coordinates in the distributional sense.
Then there exists a $W^{2,p}_{\rm loc}$ isometric immersion $f: U\hookrightarrow\rnk$
such that $g_{ij}, h_{ij}$, and $\kappa_{ij}$ are its metric, second fundamental form, and normal connection
in local coordinates, respectively.
\end{theorem}

Here and in the sequel,
we write $\mathfrak{gl}(q;\real)$ for the space of $q\times q$ matrices
with real entries,
$O(q) \subset \mathfrak{gl}(q;\real)$ for the space of symmetric $q \times q$ matrices,
and $\mathfrak{so}(q)\subset \mathfrak{gl}(q;\real)$ the space of $q\times q$ anti-symmetric matrices.

\begin{remark}
The codimension $k$ of the isometric immersion in Theorem {\rm \ref{proposition_mardare, szopos}}
above is required to be larger than or equal to
the minimal {\em Janet dimension} $J(n):=\frac{n(n+1)}{2}$.
For more details, we refer to Janet \cite{Janet} and the exposition by Han-Hong \cite{HanHon06}.
\end{remark}

The strategy for the proof in \cite{Mar03}--\cite{Mar07} and \cite{Szo08} can be briefly sketched as follows:
First, the GCR equations can be transformed into two types of first-order matrix-valued PDE systems
with $W^{2,p}$
coefficients, known as the Pfaff and Poincar\'e systems;
then, applying various analytic results established in \cite{Mar03}--\cite{Mar07},
one can construct explicitly the local isometric immersions by solving the Pfaff and Poincar\'e systems
with the rough coefficients.
Nevertheless, despite the successful solution to the {\it realization problem} (at least locally),
the transformations from the GCR equations to the Pfaff and Poincar\'e systems in
\cite{Mar03}--\cite{Mar07} and \cite{Szo08}
appear quite delicate,
which involve many different types of geometric quantities in local coordinates
({\it e.g.}, metrics, connections, and curvatures) as the entries of the same matrix
of enormous size.

In this section, we give an alternative global geometric proof
of Theorem \ref{proposition_mardare, szopos}.
In addition, we solve the problems listed in goals (i) and (ii) at one stroke.
Our perspective is essentially different from those in \cite{Mar03}--\cite{Mar07} and \cite{Szo08}:
We aim at establishing the equivalence of the GCR equations and the existence of isometric immersions
on Riemannian manifolds with $W^{1,p}_{\rm loc}$ metric, $p>n$, via the Cartan formalism
for exterior differential calculus
(see \cite{Spi79}).

We first state the main theorem
of this section,
which concerns the equivalence of three formulations of the GCR equations {\it in disguise}.
Roughly speaking, we view the Cartan formalism
as the bridge between the GCR equations and isometric immersions.

\begin{theorem}\label{thm_equivalence of 3 formulations}
Let $(M,g)$ be an $n$-dimensional, simply-connected Riemannian manifold with metric $g\in\woneploc$, $p>n$,
and let $(E,M,\real^k)$ be a vector bundle over $M$.
Assume that $E$ has a $\woneploc$ metric $g^E$ and an $\lploc$ connection $\na^E$
such that $\na^E$ is compatible with metric $g^E$.
Moreover, suppose that there exists
an $\lploc$ tensor field $S:\Gamma(E) \times \vf \mapsto \vf$ satisfying
\begin{equation*}
\langle X, S_\eta (Y)\rangle - \langle S_\eta (X), Y\rangle =0
\end{equation*}
with the corresponding $\lploc$ tensor field $B: \vf \times \vf \mapsto \Gamma (E)$ defined by
\begin{equation*}
\langle B(X,Y),\eta \rangle = -\langle S_\eta (X), Y \rangle.
\end{equation*}
Then the following are equivalent{\rm :}
\begin{enumerate}
\item[\rm (i)]
The GCR equations as in Theorem {\rm \ref{thn_GCR}} with $R^\perp$ replaced by $R^E$,
the Riemann curvature operator on the bundle{\rm ;}
\item[\rm (ii)]
The Cartan formalism
\eqref{eqn_first structure eqn}--\eqref{eqn_def of connection form 3}{\rm ;}
\item[\rm (iii)]
The existence of a global isometric immersion $f\in W^{2,p}_{\rm loc}(M; \real^{n+k})$ such that the induced normal
bundle $T(fM)^\perp$, normal connection $\na^\perp$, and second fundamental form can be identified
with $E, \na^E$, and $B$, respectively.
\end{enumerate}
In {\rm (i)}--{\rm (ii)}, the equalities are taken in the distributional sense and, in {\rm (iii)},
the isometric immersion $f\in W^{2,p}_{\rm loc}$ is unique
{\it a.e.}, modulo the group of  Euclidean motions $G=\real^{n+k} \rtimes O(n+k)$.
\end{theorem}

In Theorem \ref{thm_equivalence of 3 formulations},
we require $p>n$ (instead of $p>2$ for the weak rigidity of the GCR equations, as in \S \ref{section_GCR})
to guarantee that the immersion is $C^1$, which
agrees with the classical notions from differential geometry.
In \S 5.2 below, the Cartan formalism is introduced.
This clarifies the precise meaning for the second item in the above theorem.
Then, in \S 5.3--\S 5.4, we give a proof of Theorem \ref{thm_equivalence of 3 formulations}.

Finally, the implication: ${\rm (i)} \Rightarrow {\rm (iii)}$
leads to the following global realization theorem on Riemannian manifolds,
independent of the local coordinates:

\begin{corollary}\label{cor_ realisation of mnfd}
Let $(M,g)$ be an $n$-dimensional, simply-connected Riemannian manifold with metric $g\in\woneploc, p>n,$
and let $(E,M,\real^k)$ be a vector bundle over $M$.
Assume that $E$ has a $\woneploc$ metric $g^E$ and an $\lploc$ connection $\na^E$
such that $\na^E$ is compatible with metric $g^E$.
Moreover, suppose that there exists an $\lploc$ tensor field $S:\Gamma(E) \times \vf \mapsto \vf$ satisfying
\begin{equation*}
\langle X, S_\eta (Y)\rangle - \langle S_\eta (X), Y\rangle =0
\end{equation*}
with the corresponding $\lploc$ tensor field $B: \vf \times \vf \mapsto \Gamma (E)$ defined by
\begin{equation*}
\langle B(X,Y),\eta \rangle = -\langle S_\eta (X), Y \rangle.
\end{equation*}
Assume that the GCR equations {\rm (}as in Theorem {\rm \ref{thn_GCR}}, with $R^\perp$ replaced by $R^E${\rm )}
are satisfied in the distributional sense.
Then there exists a global isometric immersion $f \in \wtwoploc (M; \rnk)$
such that $TM^\perp$ is identified with $E$ {\rm (}together with the metric and connection{\rm )},
and $B$ and $S$ coincide with the second fundamental form and the shape operator associated
with $f$.
\end{corollary}

\subsection{Cartan Formalism}
We now introduce a useful tool, the Cartan formalism,
for isometric immersions.
We sketch some results directly pertaining to isometric immersions. For an local isometric
immersion $f: U\subset M^n \hookrightarrow \real^{n+k}$,
we adopt the index convention as in \cite{Ten71}:
\begin{equation*}
1\leq i,j \leq n;\qquad 1\leq a,b,c,d,e \leq n+k; \qquad n+1 \leq \alpha,\beta,\gamma \leq n+k.
\end{equation*}
In the sequel, our setting is always the same as in Theorem \ref{thm_equivalence of 3 formulations}.

On a local chart $U \subset M$ where the vector bundle $E$ is trivialized
({\it i.e.}, $E|_U$ is diffeomorphic to $U\times \real^k$),
let $\{\omega^i\} \subset \Omega^1(U)$ be an orthonormal co-frame dual to the orthonormal
frame $\{\partial_i\} \subset \Gamma(TU)$. The latter is called a moving frame adapted to $U$.
The Cartan  formalism is thus also known as the {\it method of moving frames}.

It is well-known that the GCR equations are equivalent to the following two systems:
\begin{equation}\label{eqn_first structure eqn}
d\omega^i = \sum_j \omega^j \wedge \omega^i_j,
\end{equation}
\begin{equation}\label{eqn_second structure eqn}
d\omega^a_b = - \sum_c \omega_b^c \wedge \omega^a_c,
\end{equation}
where \eqref{eqn_first structure eqn}--\eqref{eqn_second structure eqn} are known
as the first and second structural equations ({\it cf.} \cite{Clelland17,Spi79}).
The $1$--forms $\{\omega^a_b\}$ are defined as follows:
Let $\{\eta_{n+1}, \ldots, \eta_{n+k}\} \subset \Gamma(E)$
be an orthonormal basis for fibre $\real^k$ of bundle $E$ over the trivialized chart $U$. Then set
\begin{eqnarray}
&&\omega^i_j (\partial_k):= \langle \na_{\partial_k}\partial_j, \partial_i \rangle,
\label{eqn_def of connection form 1}\\
&&\omega^i_\alpha(\partial_j)=-\omega^\alpha_i(\partial_j):=\langle B(\partial_i, \partial_j),\eta_\alpha\rangle,
\label{eqn_def of connection form 2}\\
&&\omega^\alpha_\beta (\partial_j):= \langle\na^E_{\partial_j}\eta_\alpha, \eta_\beta \rangle.
\label{eqn_def of connection form 3}
\end{eqnarray}
The structural equations \eqref{eqn_first structure eqn}--\eqref{eqn_second structure eqn},
together with the definitions for the connection form in local coordinates,
{\it i.e.}, Eqs. \eqref{eqn_def of connection form 1}--\eqref{eqn_def of connection form 3},
are referred to as the Cartan formalism.

It is both convenient and conceptually important to introduce the following short-hand notations: Write
\begin{equation}
 W= \{\omega^a_b\}\in \Omega^1(U; \mathfrak{so}(n+k)),
\end{equation}
which is a field of $1$--form-valued anti-symmetric
matrices (or equivalently, the field of anti-symmetric matrix-valued $1$--forms).
Heuristically, $W$ packages together the second fundamental form and normal connection.
We also write
\begin{equation}
w=(\omega^1, \ldots, \omega^n, 0, \ldots, 0)\in \Omega^1(U; \rnk).
\end{equation}

As a result, the structural equations can be written as
\begin{equation}\label{short-handed structure equations}
dw = w\wedge W, \qquad dW+W\wedge W=0,
\end{equation}
where operator $\wedge$ between two matrix-valued differential forms denotes the wedge product
of differential forms together with matrix multiplication.
In the sequel, we always adhere to the practice of writing (matrix)
Lie group-valued PDEs in short-handed forms
as in \eqref{short-handed structure equations}.

We remark that the structure equations, as well as various other Lie group-valued equations we will introduce below,
are \emph{intrinsic}. That is, these equations are independent of the {\it moving frames},
hence are natural in coordinate-free notations.
This ensures the global and intrinsic nature of the proof of Theorem \ref{thm_equivalence of 3 formulations},
which is the content in \S 5.3 below. For more details on the Cartan formalism, see \cite{Spi79}.

\subsection{Proof of Theorem \ref{thm_equivalence of 3 formulations}}
With the Cartan formalism introduced,
we are now in the position to prove the main result of this section,
{\it i.e.}, Theorem \ref{thm_equivalence of 3 formulations}.
Our proof is intrinsic and global in nature, {\it i.e.},
covariant with respect to the change of local frames.
We divide the proof in seven steps.

\smallskip
{\bf 1.} We first notice the equivalence between the GCR equations and the Cartan formalism ({\it cf}. \cite{Spi79}).
Also, it is well-known that the existence of a local isometric immersion implies the GCR equations ({\it cf}. \cite{DoC92}).
Although the above classical results are established in the $C^\infty$ category in \cite{DoC92,Spi79},
it is easy to check that the proofs remain unaltered for the lower regularity case,
since $g\in W^{1,p}_{\rm loc}$ ensures that all the calculations involved make sense in the distributional sense.
Thus, it remains to prove ${\rm (ii)}\Rightarrow {\rm (iii)}$, {\it i.e.},
the Cartan formalism implies the existence of an isometric immersion.
		
We follow closely the arguments in Tenenblat \cite{Ten71} for the $C^\infty$ case.
We first prove the local version of the theorem
and then extend to the global version on simply-connected manifolds via topological arguments.

\smallskip
{\bf 2.} Recall that $W=\{\omega^a_b\}\in \Omega^1(U; \mathfrak{so}(n+k))$
and $w=(\omega^1, \ldots, \omega^n, 0, \ldots, 0)\in \Omega^1(U; \rnk)$.
In order to find an isometric immersion based upon the structural
equations \eqref{eqn_first structure eqn}--\eqref{eqn_second structure eqn},
we first solve for an affine map $A$ (taking $\partial_j$ to $\eta_\alpha$),
which essentially consists of the components of the normal affine connection,
and then the isometric immersion $f$ is constructed from $A$.

We first carry out such constructions locally.
More precisely, given $P_0\in U$, we find an open set $V$ with $P_0 \in V \subset U$
and a field of symmetric matrices $A=\{A^a_b\}: V \mapsto O(n+k)$, satisfying the first-order PDE system:
\begin{equation}\label{system 1}
\begin{cases}
W = dA \cdot A^\top \qquad \text{ in } V,\\
A(P_0)=A_0,
\end{cases}
\end{equation}
where $A^\top$ is the transpose of matrix $A$.
This equation means that
$$
W(X|_P)=dA(X|_P)\cdot A(P) \qquad \mbox{for $P\in V$ and $X\in \Gamma(TV)$},
$$
where the dot is the matrix multiplication.

Now, under the assumption that Eq. \eqref{system 1} is satisfied,
we next solve for function $f: V \mapsto \rnk$ such that
\begin{equation}\label{system 2}
\begin{cases}
df = w \cdot A \qquad \text{ in } V,\\
f(P_0)=f_0,
\end{cases}
\end{equation}
{\it i.e.}, $df_P(X|_P)=w(X|_P) \cdot A(P)$ for $P\in V$ and $X\in \Gamma(TV)$.
Following \cite{Mar05,Mar07}, we call \eqref{system 1}
the Pfaff system, and call \eqref{system 2} the Poincar\'e system.
The names come from the theory of integrable systems and the Poincar\'e lemma in cohomology theory,
respectively.
	
\smallskip
{\bf 3.}
In the $C^\infty$ category, to establish the solvability of PDEs in form \eqref{system 1}--\eqref{system 2},
which can be viewed as complete integrability conditions,
we only need to check the {\it involutiveness}, in view of the Frobenius theorem.
This constitutes the arguments in \cite{Ten71}.
In the lower regularity case, we seek for the correct analogue to the Frobenius theorem.
The following lemma
can serve for this purpose:	
	
\begin{lemma}[Mardare  \cite{Mar07}]\label{Mardare's Lemmas}
The following solubility criteria hold for two types of first-order matrix
Lie group-valued PDE systems{\rm :}
\begin{enumerate}
\item[\rm (i)]
{\rm The Pfaff system, Theorem {\rm 7} of {\rm \cite{Mar05}}}{\rm :}
Let $U\subset \real^n$ be a
simply-connected open set,
$x_0 \in U$, and $M_0\in \mathfrak{gl}(n; \real)$, the space of $n\times n$
real matrices.
Then the following system{\rm :}
\begin{equation}
\frac{\partial M}{\partial x^i}=Q_i\cdot M, \,\,\, i=1,2,\ldots,n, \,\,\qquad M(x_0)=M_0,
\end{equation}
with the matrix fields $Q_i\in\lploc(U; \mathfrak{gl}(n;\real))$ for $i=1,2,\ldots, n$, and $p>n$,
has a unique solution $M\in \woneploc (U; \mathfrak{gl}(n; \real))$ if and only if the following compatibility condition holds{\rm :}
\begin{equation}\label{compatibility pfaff}
\frac{\partial Q_i}{\partial x^j} - \frac{\partial Q_j}{\partial x^i} =[Q_i, Q_j] \qquad\text{ in } \dis \,\,\,\text{ for each } i,j=1,2, \dots, n.
\end{equation}

\item[\rm (ii)]
{\rm The Poincar\'e system, Theorem {\rm 6.5} of {\rm \cite{Mar07}}}{\rm :}
Let $U\subset \real^n$ be a
simply-connected open set, $x_0 \in U$, and $\psi_0 \in \real$.
Then the following system{\rm :}
\begin{equation}
\frac{\partial \psi}{\partial x^i} =  \phi_i, \,\,\, i=1,2,\ldots, n, \,\qquad
\psi(x_0) = \psi_0,
\end{equation}
with $\phi_i \in \lploc(U)$ for $i=1,2,\ldots, n$, and $1 \leq p \leq \infty$,
has a unique solution $\psi \in \woneploc (U)$ if and only if the following compatibility condition holds:
\begin{equation}\label{compatibility poincare}
\frac{\partial \phi_i}{\partial x^j} - \frac{\partial \phi_j}{\partial x^i} = 0 \qquad \text{ in } \dis\,\,\, \text{ for each } i,j=1,2, \dots, n.
\end{equation}
\end{enumerate}
\end{lemma}

We also remark that the uniqueness in (ii) is proved by Schwartz \cite{schwartz}.
Thanks to Lemma \ref{Mardare's Lemmas},
solving for systems \eqref{system 1} and \eqref{system 2} can be reduced to
checking the compatibility conditions in the form of
Eqs. \eqref{compatibility pfaff} and \eqref{compatibility poincare},
provided that the regularity assumptions are satisfied.

\smallskip
{\bf 4.} Let us first solve for  the Pfaff system \eqref{system 1}.
Performing contraction with the moving frame $\{\partial_a\}$, one has
\begin{equation}
dA(\partial_a) = \frac{\partial A}{\partial{x^a}}=  W(\partial_a)\cdot A,
\end{equation}
which is a Pfaff system on $\mathfrak{gl}(n+k;\real)$.
By assumption, $g, g^E \in \woneploc$, and $B, \na^E \in \lploc$,
so that $W\in L^p_{\rm loc}$ with $p>n$.
This verifies the regularity hypotheses in Lemma \ref{Mardare's Lemmas}(i).

The compatibility criterion \eqref{compatibility pfaff} can be written in local coordinates as
\begin{equation}\label{x}
\partial_b \big(\omega^c_d(\partial_a)\big) - \partial_a \big(\omega^c_d(\partial_b)\big)
= \sum_e [\omega^c_e (\partial_a), \omega^e_d(\partial_b)].
\end{equation}
Recall formula \eqref{eqn_first equation} in \S \ref{Section 2}: As a special case,
for any $1$--form $\alpha$ and vector fields $(X, Y)$,
\begin{equation}\label{5.17a}
d\alpha(X,Y)= X\alpha(Y)-Y\alpha(X) - \alpha([X,Y]).
\end{equation}
Thus, taking $\alpha=\omega^c_d$, $X=\partial_b$, and $Y=\partial_a$, we have
\begin{equation}
d\omega^c_d (\partial_a, \partial_b)
= \partial_b \big(\omega^c_d(\partial_a)\big)
- \partial_a \big(\omega^c_d(\partial_b)\big).
\end{equation}

Moreover, the Lie bracket on the right-hand side of Eq. \eqref{x} can be written
as the wedge-product of $1$--forms:
\begin{equation*}
\sum_e [\omega^c_e (\partial_a), \omega^e_d(\partial_b)]
= \sum_e \big(\omega^c_e\wedge \omega^e_d\big)(\partial_a, \partial_b).
\end{equation*}

Combining these two observations, the compatibility criterion for the Pfaff system \eqref{system 1}
is equivalent to
\begin{equation*}
d\omega^c_d (\partial_a, \partial_b) = \sum_e \big(\omega^e_d \wedge \omega^c_e\big) (\partial_a, \partial_b).
\end{equation*}
This is precisely the second structural equation \eqref{eqn_second structure eqn}.
Thus, by Lemma \ref{Mardare's Lemmas}(i),
we can find $A\in W^{1,p}_{\rm loc}(V; \mathfrak{gl}(n+k;\real))$ satisfying Eq. \eqref{system 1}.
Using equation $W=dA\cdot A^\top$ and the fact that $W\in \mathfrak{so}(n+k)$,
we deduce that $A$ maps into $O(n+k)$.

\smallskip
{\bf 5.} Next, we solve for immersion $f$ from the Poincar\'e system \eqref{system 2}.
Since $w\cdot A \in W^{1,p}_{\rm loc}$ with $p>n$, the regularity assumption in Lemma \ref{Mardare's Lemmas}(ii)
is satisfied, both for $w\cdot A$ and its first derivatives.
Thus, if we verify the compatibility condition in the form of Eq. \eqref{compatibility poincare},
we can establish the existence of $f\in W^{2,p}_{\rm loc}$, thanks to Lemma \ref{Mardare's Lemmas}.

Now, observe that the compatibility condition for the Poincar\'e system \eqref{system 2} is equivalent to
\begin{equation*}
\partial_b \big(\omega(\partial_a)\cdot A\big) - \partial_a \big(\omega(\partial_b) \cdot A \big) = 0
\end{equation*}
for each index $1\leq a,b \leq n+k$.
Invoking identity \eqref{5.17a} once more,
we can express
\begin{equation*}
d(w\cdot A)(\partial_a, \partial_b) - \partial_b \big(\omega(\partial_a)\cdot A\big) + \partial_a \big(\omega(\partial_b) \cdot A \big)=0.
\end{equation*}
Therefore, it suffices to show that
\begin{equation}\label{dwA=0}
d(w\cdot A)=0.
\end{equation}

Indeed, using the Pfaff system \eqref{system 1} solved above, we have
\begin{align}\label{id1}
d(w\cdot A)= dw \cdot A - w \wedge dA = dw \cdot A - w \wedge (W\cdot A).
\end{align}
However, the first structural equation \eqref{eqn_first structure eqn} can be expressed as
\begin{equation*}
dw^i = \sum_a \omega^a \wedge \omega^i_a = \sum_j \omega^j \wedge \omega_j^i,
\end{equation*}
since $w_\beta=0$ for any $n+1 \leq \beta \leq n+k$ by construction.
Then
\begin{equation}\label{id2}
dw\cdot A = w\wedge (W\cdot A).
\end{equation}
Putting Eqs. \eqref{id1}--\eqref{id2} together, we deduce \eqref{dwA=0},
which is equivalent to the compatibility condition for the Poincar\'e system \eqref{system 2}.
As a consequence, we find that
$f\in W^{2,p}_{\rm loc}(V;\real^{n+k})$ solves Eq. \eqref{system 2}.

\smallskip
{\bf 6.}
With the solution to the Pfaff system \eqref{system 1} and the Poincar\'e system \eqref{system 2}
associated with our isometric immersions problem,
it remains to check that $f$ is indeed the immersion for which we are seeking.
To this end, we can define a new local moving frame $\{e_1, \ldots, e_{n+k}\}$ by
\begin{equation}\label{eqn_define the new frame in rnk}
e_i := df(\partial_i), \qquad  e_\alpha:= \sum_b A^\alpha_b \partial_b.
\end{equation}
Then, following precisely the same arguments as those in pages 32--35 in \cite{Ten71},
we can show that such a construction gives the correct normal bundle metric,
normal affine connection, and second fundamental form.

Moreover, observe that $df= w\,\cdot A$ from Eq. \eqref{system 2},
$A \in O(n+k)$ is a field of nonsingular matrices, and $\{\omega^1, \ldots, \omega^n\}$
are linearly independent, so that $df \neq 0$ in $\woneploc$.
Since $p > n$, the Sobolev embedding yields that $df \neq 0$ almost everywhere,
which verifies that $f$ is indeed an immersion.
The almost everywhere uniqueness of $f$ follows from Lemma \ref{Mardare's Lemmas}
and the $\real^{n+k}\rtimes O(n+k)$--symmetry of the Euclidean space.

\smallskip
{\bf 7.} It remains to globalize our arguments
for simply-connected manifolds.
This follows from a standard argument in topology.

Given any two points $x \neq y \in M$, we connect them by a continuous
curve (again since $g \in \woneploc \hookrightarrow C^0_{\rm loc}$ for $p>n$),
denoted by $\gamma: [0,1]\mapsto M$ with $\gamma(0)=x$ and $\gamma(1)=y$.
Let $f$ be the $\wtwoploc$ isometric immersion in a neighbourhood of $x$,
whose existence is guaranteed by the preceding steps of the same proof.
Cover curve $\gamma([0,1])$ by finitely many charts $\{V_1,\ldots,V_N\}$.
By the uniqueness statements in Lemma \ref{Mardare's Lemmas},
we can extend the isometric immersion $f$ to $\bigcup_{i=1}^N V_i$,
especially including a neighbourhood of $y$.

Thus, it suffices to verify that the extension of $f$ is independent of the choice of $\gamma$.
Indeed, if $\eta:[0,1]\mapsto M$ is another continuous curve connecting $x$ and $y$,
by concatenating $\gamma$ with $\eta$, we can obtain a loop $L\subset M$.
As $M$ is simply-connected, the restriction $f|_L$ is homotopic to a constant map
so that $(f\circ\gamma)(1)=(f\circ\eta)(1)$.
In this way, we have verified that $f$ can be extended to a global isometric
immersion of $M$ into $\rnk$, provided that $M$ is simply-connected.

This completes the proof of Theorem \ref{thm_equivalence of 3 formulations}.

\medskip
As a corollary of Theorems \ref{thm_main theorem weak continuity} and \ref{thm_equivalence of 3 formulations},
we can deduce the weak rigidity of isometric immersions.

\begin{corollary}[Weak rigidity of isometric immersions]\label{cor_Weak continuity of isometric immersions}
Let $M$ be an $n$-dimensional simply-connected Riemannian manifold with metric $g\in W^{1,p}_{\rm loc}$ for $p>n$.
Suppose that $\{f^\epsilon\}$ is a family of isometric immersions of $M$ into $\rnk$, uniformly bounded
in $W^{2,p}_{\rm loc}(M; \rnk)$, whose second fundamental forms and normal connections are $\{B^\epsilon\}$
and $\{\nep\}$ respectively.
Then, after passing to subsequences, $\{f^\epsilon\}$ converges to $\bar{f}$
weakly in $W_{{\rm loc}}^{2,p}$
which is still an isometric
immersion $\bar{f}: (M,g) \hookrightarrow \rnk$.
Moreover, the corresponding second fundamental form $\bar{B}$ is a limit point of $\{B^\epsilon\}$,
and the corresponding normal connection $\overline{\na^\perp}$
is a limit point of $\{\nep\}$, both taken in the $L^p_{\rm loc}$ topology.
\end{corollary}

\begin{proof}
For a sequence $\{f^\epsilon\}$ of isometric immersions, uniformly bounded
in $W^{2,p}_{\rm loc}(M; \rnk)$, we can define the corresponding second fundamental
forms and normal affine connections, denoted by $(\bep,\nep)$.

By Theorem \ref{thm_equivalence of 3 formulations},
$(\bep,\nep)$ satisfy the GCR equations in the distributional sense
and are uniformly bounded in the $L^p$ norm.
Let $(\overline{B},\overline{\na^\perp})$ be a weak limit of this family.
Now, in view of Theorem \ref{thm_main theorem weak continuity} on the weak rigidity of
the GCR equations,
$(\overline{B},\overline{\na^\perp})$ is still a solution to the GCR equations in the distributional sense.
Thus, invoking Theorem \ref{thm_equivalence of 3 formulations} again,
we can find a $W^{2,p}_{\rm loc}$ isometric immersion $\overline{f}$,
for which  $\overline{B}$ and $\overline{\na^\perp}$ are its second fundamental form
and normal connection respectively.
By the uniqueness of distributional limits, $\overline{f}$ must coincide with
the weak limit of $\{f^\epsilon\}$.

Therefore, we have verified that the weak limit of a sequence of isometric immersions of manifold $M$
is still an isometric immersion of $M$, with corresponding second fundamental form and normal connection.
This completes the proof.
\end{proof}

One natural question arises at this stage is about the criteria for the existence
of global isometric immersions for a non-simply-connected manifold $M$ with $W^{1,p}_{\rm loc}$ metrics,
especially for the case of the minimal target dimension ({\it i.e.}, the Janet dimension).
However, as far as we have known,
it is still open, primarily owing to some topological obstructions
to the GCR equations.
Suppose that there is a loop $\gamma$ generating non-trivial homotopy on $M$,
it is far from being clear if one can find a well-defined immersion along the entire loop.
The problems on global isometric immersions/embeddings for general manifolds
remain vastly open in the large.
See Schoen-Yau \cite{SY} and Bryant-Griffith-Yang \cite{BGY} for further discussions.

Therefore, the best we may say for non-simply-connected manifolds are
the \emph{local} versions of Theorem \ref{thm_equivalence of 3 formulations}
and Corollaries \ref{cor_ realisation of mnfd} and \ref{cor_Weak continuity of isometric immersions}.

\section{The Critical Case:  $n=p=2$}

Our main geometric result established in \S 5, {\it i.e.}, Theorem \ref{thm_equivalence of 3 formulations},
deals with the Riemannian manifolds with $W^{1,p}_{\rm loc}$ metrics for $p>n$.
Even for the weak rigidity of the GCR equations, we still
require $p>2$, as in \S \ref{section_GCR}.
Therefore, $n=p=2$ becomes a critical case, from both the geometric perspectives
and the PDE point of view.
This is what we investigate in this section.
We focus on a $2$-dimensional manifold $M$ with some Riemannian metric $g\in H^1_{\rm loc}$.

The difficulty lies in the insufficiency of regularity for certain Sobolev embeddings.
Notice that, for $n=p=2$, the second fundamental forms and normal connections $(\bep,\nep)$
are only uniformly bounded in $L^2_{\rm loc}$, so the quadratic nonlinearities
in the GCR equations are at most bounded in $L^1_{\rm loc}$.
However, $L^1_{\rm loc}$ cannot be embedded into $H^{-1}_{\rm loc}$.
This can be seen via the dual spaces,
since $H^1(\real^2)$ is only continuously,
but not compactly, embedded into $BMO (\real^2)$,
which is the space of functions of bounded mean oscillations; moreover, $BMO (\real^2)$ is strictly contained in $L^\infty(\real^2)$.
A standard example for $f \in H^{1}(\real^2) \setminus L^\infty(\real^2) $
is $f(x)=\log\log(|x|^{-1})\chi_{B_1(0)}$.

However, if the codimension of the immersion is $1$, {\it i.e.},
the surface $M$ is immersed into $\real^3$, we can still obtain the weak rigidity.
For this purpose, we need the corresponding critical case
of Theorem \ref{thm_main thm goemetric divcurl},
in which $\{d\omega^\epsilon\}$ and $\{\delta\tau^\epsilon\}$ are contained
in compact subsets of $W^{-1,1}_{\rm loc}$ spaces.
This has been treated in Conti-Dolzmann-M\"{u}ller in \cite{CDM}.

We now state and prove a slight variant of the critical case result in \cite{CDM},
formulated for global differential forms on the Riemannian manifolds.
This can be achieved by combining the global arguments on the manifolds developed above with the
arguments in \cite{CDM}; thus, its proof will be only sketched briefly.

\begin{theorem}\label{thm_critical case, div curl lemma}
Let $(M,g)$ be an $n$-dimensional manifold.
Let $\{\omega^\epsilon\}, \{\tau^\epsilon\}\subset L^2_{\rm loc}(M;\ptens)$
be two families of  differential $q$--forms, for $0\leq q\leq n$. Suppose that
\begin{enumerate}
\item[\rm (i)]
$\omega^\epsilon \rightharpoonup \overline{\omega}$ weakly in $L^2$,
and $\tau^\epsilon \rightharpoonup \overline{\tau}$ weakly in $L^2$;
\item[\rm (ii)] There are compact subsets of the corresponding Sobolev spaces, $K_d$ and $K_\delta$, such that
\begin{equation*}
\begin{cases}
\{d\omega^\epsilon\}\subset K_d \Subset W^{-1,1}_{\rm loc}(M;\bigwedge^{q+1}T^*M),\\
\{\delta\tau^\epsilon\}\subset K_\delta \Subset W^{-1,1}_{\rm loc}(M;\bigwedge^{q-1}T^*M){\rm ;}
\end{cases}
\end{equation*}
\item[\rm (iii)] $\{\langle\omega^\epsilon,\tau^\epsilon\rangle\}$ is  equi-integrable.
\end{enumerate}
Then $\langle\omega^\epsilon, \tau^\epsilon\rangle$ converges to $\langle\overline{\omega}, \overline{\tau}\rangle$
in the sense of distributions:
\begin{equation*}
\int_M \langle\omega^\epsilon, \tau^\epsilon\rangle\psi\, {\rm d}V_g
\longrightarrow \int_M \langle\overline{\omega}, \overline{\tau}\rangle \psi\, {\rm d}V_g
\qquad \mbox{for any $\psi \in C^\infty_{c}(M)$}.
\end{equation*}
\end{theorem}

\begin{proof}
First of all, let us make two reductions. First, as in the proof for Theorem \ref{thm_main thm goemetric divcurl},
it suffices to prove for compact orientable manifold $M$; Second, without loss of generality,
we may take $\bar{\omega}=\bar{\tau}=0$.

Next we  perform a truncation to $\{\omega^\epsilon\}$ and $\{\tau^\epsilon\}$.
By assumption (i), sequence $\{|\omega^\epsilon|^2\}$ is weakly convergent in $L^1$.
Hence, applying Chacon's biting lemma ({\it cf.} \cite{BM}),
one can find subsets $K_\epsilon\subset M$
such that $|K_\epsilon|_{g}:=\int_M  \chi_{K_{\epsilon}}dV_g\rightarrow 0$
and $\{|\omega^\epsilon|^2 \chi_{M\setminus K_\epsilon}\}$ is equi-integrable.
We define the truncated differential
forms $\tilde{\omega}^\epsilon:=\omega^\epsilon\chi_{M\setminus K_\epsilon}$.
Similarly, we take $\tilde{K}_\epsilon\subset M$ such that $|\tilde{K}_\epsilon|_g \rightarrow 0$
and $\{|\tau^\epsilon|^2 \chi_{M\setminus \tilde{K}_\epsilon}\}$ is equi-integrable
to obtain the truncated forms $\tilde{\tau}^\epsilon:=\tau^\epsilon\chi_{M\setminus \tilde{K}_\epsilon}$.
	
Now, let us measure the $L^1$ difference of $\omega^\epsilon$ and $\tilde{\omega}^\epsilon$.
By the Cauchy-Schwarz inequality,
	\begin{equation*}
	\|\tilde{\omega}^\epsilon -\omega^\epsilon\|_{L^1(M;\ptens)}
= \|\omega^\epsilon\|_{L^1(K_\epsilon)} \leq \sqrt{|R_\epsilon|_g} \|\omega^\epsilon\|_{L^2(M;\ptens)} \rightarrow 0,
	\end{equation*}
since $\{\omega^\epsilon\}$ is uniformly bounded in $L^2$ owing to assumption (i).
Moreover, by assumption (ii), $\|d\tilde{\omega}^\epsilon\|_{W^{-1,1}(M;\bigwedge^{q+1}T^*M)} \rightarrow 0$.
Analogously, we find that $\|\tilde{\tau}^\epsilon -\tau^\epsilon\|_{L^1(M;\ptens)}\rightarrow 0$
and $\|\delta\tilde{\tau}^\epsilon\|_{W^{-1,1}(M;\bigwedge^{q-1}T^*M)}\rightarrow 0$. Moreover, we note that
\begin{equation}\label{dunford pettis}
\langle \omega^\epsilon, \tau^\epsilon\rangle - \langle\tilde{\omega}^\epsilon, \tilde{\tau}^\epsilon\rangle \rightarrow 0,
\end{equation}
in view of the equi-integrability assumption (iii) and $|K_\epsilon \cup \tilde{K}_\epsilon|_g \rightarrow 0$.	
By the Dunford-Pettis theorem and assumption (iii),
we know that $\{\langle\omega^\epsilon,\tau^\epsilon\rangle\}$ is weakly pre-compact in $L^1$.
Hence, Eq. \eqref{dunford pettis} implies that
$\{\langle\omega^\epsilon,\tau^\epsilon\rangle\}$
and $\{\langle\tilde{\omega}^\epsilon, \tilde{\tau}^\epsilon\rangle\}$ have the same distributional limits.

Thus, it remains to show that $\langle\tilde{\omega}^\epsilon, \tilde{\tau}^\epsilon\rangle \rightarrow 0$
in the distributional sense.
In fact, by the Lipschitz truncation argument in \cite{CDM},
$d \tilde{\omega}^\epsilon$ and $\delta\tilde{\tau}^\epsilon$ converge to $0$ in $H^{-1}_{\rm loc}$,
after passing to subsequences.
Therefore, we can conclude the proof from the intrinsic div-curl lemma, {\em i.e.},
Theorem \ref{thm_main thm goemetric divcurl}.
\end{proof}

\begin{remark}
In the proof of Theorem {\rm \ref{thm_critical case, div curl lemma}},
the technique of Lipschitz truncation has played an important role,
which reduces the div-curl lemma from the critical case to the usual case,
where the target spaces are reflexive {\rm (}{\it cf.} Theorem {\rm \ref{thm_main thm goemetric divcurl}}{\rm )}.
Such a technique depends explicitly on the geometry of the underlying manifolds.
\end{remark}

We remark that the endpoint case of the intrinsic div-curl lemma,
{\em i.e.}, Theorem \ref{thm_critical case, div curl lemma}, can also be generalized in a similar manner
as for Theorem \ref{thm_general goemetric divcurl}.

\begin{theorem}\label{thm_ generalised critical case, div curl lemma}
Let $(M,g)$ be an $n$-dimensional manifold.
Let $\{\omega^\epsilon\}\subset L^r_{\rm loc}(M;\ptens)$
and $ \{\tau^\epsilon\}\subset L^s_{\rm loc}(M;\ptens)$ be two families of differential $q$--forms,
for $0\leq q\leq n$, $1<r,s<\infty$, and $\frac{1}{r}+\frac{1}{s}=1$. Suppose that
\begin{enumerate}
\item[\rm (i)]
$\omega^\epsilon \rightharpoonup \overline{\omega}$ weakly in $L^r$, and $\tau^\epsilon \rightharpoonup \overline{\tau}$ weakly in $L^s$
as $\e\to 0${\rm ;}
\item[\rm (ii)] There are compact subsets of the corresponding Sobolev spaces, $K_d$ and $K_\delta$, such that
\begin{equation*}
\begin{cases}
\{d\omega^\epsilon\}\subset K_d \Subset W^{-1,1}_{\rm loc}(M;\bigwedge^{q+1}T^*M),\\
\{\delta\tau^\epsilon\}\subset K_\delta \Subset W^{-1,1}_{\rm loc}(M;\bigwedge^{q-1}T^*M);
\end{cases}
\end{equation*}
\item[\rm (iii)] $\{\langle\omega^\epsilon,\tau^\epsilon\rangle\}$ is  equi-integrable.
\end{enumerate}
Then $\langle\omega^\epsilon, \tau^\epsilon\rangle$ converges to $\langle\overline{\omega}, \overline{\tau}\rangle$
in the sense of distributions{\rm :}
\begin{equation*}
\int_M \langle\omega^\epsilon, \tau^\epsilon\rangle\psi\, {\rm d}V_g
\longrightarrow \int_M \langle\overline{\omega}, \overline{\tau}\rangle \psi\, {\rm d}V_g
\qquad \mbox{for any $\psi \in C^\infty_{c}(M)$}.
\end{equation*}
\end{theorem}

The proof of Theorem \ref{thm_ generalised critical case, div curl lemma} follows
directly by
combining the argument for Theorem \ref{thm_general goemetric divcurl} in the appendix
with the proof for Theorem \ref{thm_critical case, div curl lemma}.

\smallskip
After introducing the intrinsic div-curl lemmas in the critical case,
we can now establish the global weak rigidity of isometric immersions
for a surface into $\real^3$:

\begin{theorem}
Let $M$ be a $2$-dimensional, simply-connected surface,
and let $g$ be a metric in $H^{1}_{\rm loc}$. If $\{f^\epsilon\}$ is a family of $H^{2}_{\rm loc}$ isometric
immersions of $M$ into $\real^3$ such that the corresponding second fundamental
forms $\{B^\epsilon\}$ are uniformly bounded in $L^2$. Then, after passing to subsequences, $\{f^\epsilon\}$ converges
to $\bar{f}$ weakly in $H^2_{\rm loc}$ which is still an isometric immersion $\bar{f}:(M,g) \hookrightarrow \real^3$.
Moreover, the corresponding second fundamental form $\bar{B}$ is a weak limit of $\{\bep\}$ in $L^2_{\rm loc}$.
\end{theorem}

\begin{proof} We divide the proof into three steps.

\smallskip
{\bf 1.} Since the codimension of the immersions $\{f^\epsilon\}$ is $1$, the normal affine connections are trivial. This is because, in the Ricci equation:
\begin{equation*}
\langle[S_\eta, S_\xi]X,Y\rangle = R^\perp (X,Y,\eta,\xi),
\end{equation*}
we can only take the normal vector fields $\eta$ and $\xi$
to be linearly dependent,
as the fibre for the normal bundle is simply $\real$.
Then both sides are equal to zero, by the definition of $R^\perp$ and the Lie bracket.
Now we look at the Codazzi equation.
In the second formulation for  Theorem \ref{thm_main theorem weak continuity},
we have obtained Eq. \eqref{eqn_Codazzi Eqn in second computation}, which is
\begin{equation*}
0=d(\ob_{Z,\eta})(X,Y) + \sum_\beta \langle \vnab_{\eta,\beta}(X,Y),\ob_{Z,\beta}\rangle + E(B).
\end{equation*}
However, since $\na^\perp$ is trivial, the equation is reduced to
\begin{equation}
d(\ob_{Z,\eta})(X,Y) = -E(B),
\end{equation}
where $E(B)$ is linear in $B$.
We notice that the Codazzi equation in the critical case becomes linear, as the quadratic nonlinearities are absent, so that
we can pass to the weak limits.

\smallskip
{\bf 2.} It remains to prove the weak rigidity of the Gauss equation:
\begin{equation*}
\langle B^\epsilon(Y,W), B^\epsilon(X,Z) \rangle -\langle B^\epsilon(X,W), B^\epsilon(Y,Z)\rangle = R(X,Y,Z,W).
\end{equation*}

By assumption, $\{\bep\}$ is uniformly bounded in $L^2_{\rm loc}$,
so that $R \in L^1_{\rm loc}(M; \bigotimes^4 T^*M)$ is a fixed function,
by the Cauchy-Schwarz inequality.
It follows that
\begin{align*}
&\int_M \Big|\langle B^\epsilon(Y,W), B^\epsilon(X,Z) \rangle
- \langle B^\epsilon(X,W), B^\epsilon(Y,Z)\rangle\Big|\chi_A\, {\rm d}V_g \\
&\leq \int_M \chi_A|R(X,Y,Z,W)|\, {\rm d}V_g \rightarrow 0 \qquad \text{ as } |A|_g \rightarrow 0,
\end{align*}
{\it i.e.},
the set of quadratic nonlinearities
$\{\langle B^\epsilon(Y,W), B^\epsilon(X,Z) \rangle
- \langle B^\epsilon(X,W), B^\epsilon(Y,Z)\rangle\}$
is equi-integrable. Here it is crucial that $R$ is a locally integrable function, not just a Radon measure.

\smallskip
{\bf 3.} Now, we invoke again the arguments in \S \ref{section_GCR} to analyze the div-curl structure of the Gauss equation.
First, recall our definition for the tensor field $\vb_{Z,\eta}: \vf \times \vf \mapsto \vf$
and the $1$--form $\ob_{Z,\eta} \in \Omega^1(M)=\Gamma(T^*M)$. We set
\begin{equation*}
\begin{cases}
\vb_{Z,\eta}(X,Y):= B(X,Z,\eta)Y-B(Y,Z,\eta)X, \\[2mm]
\ob_{Z,\eta}:= -B(\bullet, Z, \eta).
\end{cases}
\end{equation*}
Then the first formulation in \S \ref{section_GCR} leads to
\begin{equation}\label{xx}
\begin{cases}
{\rm div}\big(\vb_{Z,\eta}(X,Y)\big)= YB(X,Z,\eta) - X B(Y,Z,\eta),\\[1mm]
d \big(\ob_{Z,\eta}\big)(X,Y) = YB(X,Z,\eta) - XB(Y,Z,\eta) + B([X,Y],Z,\eta),
\end{cases}
\end{equation}
and the second formulation  in \S 4 enables us to recast the Gauss equation into the following form:
\begin{equation*}
\sum_\eta \langle\vb_{Z,\eta}(X,Y), \ob_{W,\eta} \rangle =-R(X,Y,Z,W).
\end{equation*}

Since $\{\bep\}$ is uniformly bounded in $L^2_{\rm loc}$,
we can deduce from the equations in \eqref{xx} that $\{\delta V^{(B^\epsilon)}\}$ and $\{d\Omega^{(B^\epsilon)}\}$
are relatively compact in $W^{-1,1}_{\rm loc}$.
Now we can employ the critical case of our intrinsic div-curl lemma, Theorem \ref{thm_ generalised critical case, div curl lemma},
to conclude that
\begin{align}
\langle B^\epsilon(X,W), B^\epsilon(Y,Z)- \langle B^\epsilon(Y,W), B^\epsilon(X,Z) \rangle
\longrightarrow \langle \overline{B}(X,W), \overline{B}(Y,Z) - \langle \overline{B}(Y,W), \overline{B}(X,Z) \rangle
\end{align}
in the distributional sense,
where $\overline{B}$ is the $L^2_{\rm loc}$ weak limit of $\{B^\epsilon\}$ up to subsequences.
From here, we conclude the weak rigidity of the Gauss equation.
Therefore, the proof is complete.
\end{proof}

\section{Further Results and Remarks}\label{section_remark}
	
In this section, we first provide a proof of the
weak rigidity of the Cartan formalism.
Then we extend the weak rigidity theory to allow
manifolds $(M, g^\epsilon)$ with metrics $\{g^\epsilon\}$ converging to $g$ in $W^{1,p}_{\rm loc}, p>n$,
instead of the fixed manifold $(M, g)$ with fixed metric $g$.

\subsection{Weak Rigidity of the GCR Equations and Isometric Immersions  Revisited}

We now provide a proof of the weak rigidity of the Cartan formalism,
which leads to
the weak rigidity of the GCR equations and isometric immersions,
in view of the equivalence between the Cartan formalism, the GCR equations,
and isometric immersions
established in \S \ref{section_equivalence}.
Therefore, we also provide an alternative proof of
Theorem \ref{thm_main theorem weak continuity}
and Corollary \ref{cor_Weak continuity of isometric immersions}.	
The arguments are summarized as follows:
	
\begin{proof}[Alternative Proof of Theorem {\rm \ref{thm_main theorem weak continuity}}
and Corollary {\rm \ref{cor_Weak continuity of isometric immersions}}]

Observe that all the nonlinear terms in the Cartan formalism are contained
in the second structural equation, {\it i.e.},
\begin{equation}\label{eee}
dW=-W\wedge W.
\end{equation}
Thus it suffices to establish the weak rigidity of \eqref{eee}.

Indeed, consider the family of connection forms $\{W^\epsilon\}$ associated
with $\{(B^\epsilon,\nabla^{\perp,\epsilon})\}$ in the Cartan formalism for isometric
immersions; see \S 5.
As $\{(B^\epsilon,\nabla^{\perp,\epsilon})\}$ is uniformly bounded in $L^p_{\text{loc}}$,
the same holds for $\{W^\epsilon\}$.
Also, recall that $W^\epsilon$ is a Lie algebra-valued $1$--form, {\it i.e.}, it is an element
of $\Omega^1(\mathfrak{so}(n+k)) \cong \mathfrak{so}(n+k) \bigotimes \Omega^1(M)$.
Throughout this proof, the Hodge star $\ast$ is always understood as taken with respect to
the $\Omega^q(M)$ factor of $\mathfrak{so}(n+k) \bigotimes \Omega^q(M)$.
We see from the definition of Sobolev spaces of tensor fields in \S 2
that $\ast$ is an isometry between $W^{k,p}(M; \bigwedge^q T^*M)$
and $W^{k,p}(M; \bigwedge^{n-q} T^*M)$ for any $k, p$, and $q$.

Next, we take an arbitrary test differential form $\eta \in C^\infty_c(M; \bigwedge^{n-2} T^*M)$,
which is independent of $\epsilon$ and has a trivial $\mathfrak{so}(n+k)$--component.
Eq. \eqref{eee} can then be recast as
\begin{equation}\label{equation: eta in the alternate proof}
dW^\epsilon \wedge \eta = - W^\epsilon \wedge W^\epsilon \wedge \eta.
\end{equation}

Define $V^\epsilon := \ast(W^\epsilon \wedge \eta) \in \Omega^{1}(\mathfrak{so}(n+k))$.
Using
$\delta = \pm \ast d \ast$, $\ast \ast = \pm {\rm id}$, and
$\langle \alpha,\beta \rangle = \alpha \wedge \ast \beta$ for differential
forms $\alpha$ and $\beta$
of the same order ({\it cf.} \S 2), we obtain the next two equalities:
\begin{eqnarray}
&&\langle dW^\epsilon, \ast \eta \rangle = \pm \langle W^\epsilon,  V^\epsilon\rangle,
\label{equation for W, alternative proof}\\
&&\ast \delta V^\epsilon = \pm (\ast W^\epsilon) \wedge V^\epsilon \pm W^\epsilon \wedge d\eta.
\label{equation for V, alternative proof}
\end{eqnarray}
For our purpose, we bound only for relevant geometric quantities in some $W^{k,p}_{\rm loc}$
so that
there is no need to keep track of the specific signs: Instead, we just write $\pm$ in suitable places.

With these,
we now analyze the weak rigidity of Eq. \eqref{eee}.
As $\eta \in C^\infty_c$, and the Hodge star $\ast$ is isometric, the Cauchy-Schwarz inequality
gives that $\{\langle W^\epsilon, V^\epsilon \rangle\}$ and $\{(\ast W^\epsilon) \wedge V^\epsilon\}$
are uniformly bounded in $L^{p/2}_{{\rm loc}}$ for $p>2$.
Also, $\{W^\epsilon \wedge d\eta\}$ is uniformly bounded in $L^p_{\rm loc}$ for $p>2$.
Thus, in view of the Sobolev embeddings and
Eqs. \eqref{equation for W, alternative proof}--\eqref{equation for V, alternative proof},
we find that $\{dW^\epsilon\}$ and $\{\delta V^\epsilon\}$ are pre-compact in $W^{-1, r}_{{\rm loc}}$, with $1<r<2$.
On the other hand, $\{dW^\epsilon\}$ and $\{\delta V^\epsilon\}$ are uniformly bounded in $W^{-1, p}_{\rm loc}$ for $p>2$.
Thus, by interpolation, the curl ({\it i.e.}, $d$) of $\{W^\epsilon\}$ and the divergence ({\it i.e.}, $\delta$ modulo the sign)
of $\{V^\epsilon\}$ are pre-compact in $H^{-1}_{\rm loc}$.

Then, employing  the geometric div-curl lemma ({\it i.e.}, Theorem \ref{thm_main thm goemetric divcurl}),
$\langle W^\epsilon, V^\epsilon\rangle = \ast(W^\epsilon \wedge \ast V^\epsilon)$ converges
to $\langle W, V\rangle$ in the distributional sense, where $W$ and $V$ are the weak $L^p_{\rm loc}$ limits
of $\{W^\epsilon\}$ and $\{V^\epsilon\}$ respectively.
Moreover, $V = \ast (W \wedge \eta)$.
Substituting them into Eq. \eqref{equation: eta in the alternate proof}, we have
\begin{equation}
dW \wedge \eta = - W\wedge W \wedge \eta.
\end{equation}
Then, by the arbitrariness of $\eta \in C^\infty_c(M; \bigwedge^{n-2} T^*M)$,
we conclude the weak rigidity of the second structural equation \eqref{eee}.

Therefore, in view of the equivalence of the Cartan formalism, the GCR equations,
and the isometric immersions of Riemannian manifolds ({\it cf.} Theorem \ref{thm_equivalence of 3 formulations},
Theorem \ref{thm_main theorem weak continuity},
and Corollary \ref{cor_Weak continuity of isometric immersions}), the proof is now  complete.
\end{proof}

We remark that the above proof lies in the same spirit as in Chen-Slemrod-Wang \cite{CSW10} for the GCR equations.
Related arguments are also present in \cite{Ball,CLMS,Eva90,Evans-Muller,Muller,Mur78,Murat2,Tar79,Tartar2}
and the references cited therein.

\subsection{Weak Rigidity of Isometric Immersions with Different Metrics}

We now develop an extension of the weak rigidity theory of the GCR equations and isometric immersions established above.
In the earlier sections, we have analyze the isometric immersions of a {\em fixed} Riemannian manifold.
In particular, despite the change of second fundamental forms and normal connections as we shift between
distinctive immersions,  metric $g\in W^{1,p}_{\text{loc}}$ is always fixed.
In what follows, we generalize the weak rigidity results ({\it cf.} Theorem \ref{thm_main theorem weak continuity}
 and Corollary \ref{cor_Weak continuity of isometric immersions}) to the Riemannian manifolds with unfixed
 metrics $\{g^\epsilon\}$.

More precisely, we establish the following theorem.
\begin{theorem}\label{theorem_weak continuity with changing metrics}
Let $(M, g^\epsilon)$ be a sequence of $n$-dimensional simply-connected Riemannian manifolds
with metrics $\{g^\epsilon\}\subset W^{1,p}_{\text{loc}}(M;\text{Sym}^2T^*M)$ converging
strongly in $W^{1,p}_{\text{loc}}$ for $p>n$.
Suppose that there exists a family of corresponding isometric immersions
of $(M,g^\epsilon)$ converging weakly in $W^{2,p}_{\text{loc}}(M;\R^{n+k})$,
denoted by $\{f^\epsilon\}$, with second fundamental forms $\{B^\epsilon\}$ and normal connections $\{\na^{\perp,\epsilon}\}$.
Then, after passing to subsequences, $\{f^\epsilon\}$ converges weakly to an isometric immersion $\bar{f}$ of $(M,g)$ in $W^{2,p}_{\rm loc}$,
where $g$ is the $W^{1,p}_{\text{loc}}$ limit of $\{g^\epsilon\}$.
Moreover, the second fundamental form $\overline{B}$ and normal connection $\overline{\na^\perp}$ of immersion $\bar{f}$
coincide with the corresponding subsequential weak  $L^p_{\text{loc}}$--limits of $\{B^\epsilon\}$ and $\{\na^{\perp,\epsilon}\}$.
\end{theorem}

\begin{proof}
In this proof, we write the inner product of $X$ and $Y$ induced by $g^\epsilon$ by $g^\epsilon(X,Y)$,
and the previous notation $\langle\cdot,\cdot\rangle$ is reserved for the paring of $(TM, T^*M)$.
Moreover, $g_0$ always denotes the Euclidean inner product in $\R^{n+k}$.
We also write $\na^\epsilon$ and $R^\epsilon$ for the Levi-Civita connection and Riemann curvature
tensor, respectively, associated with $g^\epsilon$.
We divide the proof into four steps.

\smallskip
{\bf 1.} First of all, we again rewrite the GCR equations as in the first and second formulation
in \S 4 ({\it cf.} Lemmas \ref{lemma:4.1}--\ref{lemma:4.2}).
Indeed, defining the vector fields $V$ and $1$--forms $\Omega$ as follows:
\begin{equation*}
\begin{array}{lll}
&V^{(B^\epsilon)}_{Z,\eta}(X,Y)= g_0(B^\epsilon(X,Z),\eta) Y - g_0(B^\epsilon(Y,Z),\eta)X,
     &\qquad\Omega^{(B^\epsilon)}_{Z,\eta} = -g_0(B^\epsilon(\bullet,Z),\eta), \\
&V^{(\na^{\perp, \epsilon})}_{\xi,\eta} (X,Y) = g_0(\na^{\perp, \epsilon}_Y\xi,\eta)X-g_0(\na^{\perp, \epsilon}_X\xi, \eta)Y,
  &\qquad\Omega^{(\na^{\perp,\epsilon})}_{\xi,\eta} = g_0 (\na^{\perp, \epsilon}_\bullet\xi,\eta),
\end{array}
\end{equation*}
we can establish the equivalence of the GCR equations with the following three expressions:
\begin{equation}\label{eqn g epsilon 1}
\sum_\eta \langle V^{(B^\epsilon)}_{W,\eta}(X,Y),\Omega^{(B^\epsilon)}_{Z,\eta} \rangle = -R^\epsilon(X,Y,W,Z),
\end{equation}
\begin{equation}\label{eqn g epsilon 2}
d\Omega^{(B^\epsilon)}_{Z,\eta}(X,Y) + \sum_\beta \langle V^{(\na^{\perp,\epsilon})}_{\eta,\beta}(X,Y),\Omega^{(B^\epsilon)}_{Z,\beta}\rangle + g_0\big(B^\epsilon(Y,\na^\epsilon_X Z),\eta
\big) - g_0\big(B^\epsilon(X,\na^\epsilon_YZ),\eta\big) =0,
\end{equation}
\begin{equation}\label{eqn g epsilon 3}
d\Omega^{(\na^{\perp,\epsilon})}_{\xi,\eta}(X,Y)
+ \sum_\beta \langle V^{(\na^{\perp,\epsilon})}_{\eta,\beta}(X,Y), \Omega^{(\na^{\perp,\epsilon})}_{\xi,\beta} \rangle
= \sum_Z\langle V^{(B^\epsilon)}_{Z,\xi}(X,Y), \Omega^{(B^\epsilon)}_{Z,\eta} \rangle.
\end{equation}
The derivations of these identities are the same as in Lemma \ref{lemma:4.2}.

\smallskip
{\bf 2.}
We now analyze the mode of  convergence for each term.
Since $\{f^\epsilon\}$ is weakly convergent in the $W^{2,p}_{\text{loc}}$ norm
(hence $\{B^\epsilon\}$ and $\{\na^{\perp,\epsilon}\}$ are uniformly bounded in $L^p_{\text{loc}}$),
we know that $\big\{V^{(B^\epsilon)}_{Z,\eta}(X,Y), V^{(\na^{\perp,\epsilon})}_{\xi,\eta}(X,Y)\big\} \subset L^p_{\text{loc}}(TM)$
and $\big\{\Omega^{(B^\epsilon)}_{Z,\eta}, \Omega^{(\na^{\perp,\epsilon})}_{\xi,\eta}\big\} \subset L^p_{\text{loc}}(T^*M)$
are also uniformly bounded.
Moreover, from the assumption that $\{g^\epsilon\}$ is strongly convergent in $W^{1,p}_{\text{loc}}$,
the Levi-Civita connections $\na^\epsilon: \Gamma(TM) \times \Gamma(TM) \rightarrow \Gamma(TM)$
are strongly convergent in $L^p_{\text{loc}}$. This can be seen either by the local expression
\begin{equation*}
(\Gamma^\epsilon)^i_{jk}:={g^\epsilon}(\na_{\partial_i}\partial_j,\partial_k)
= \frac{1}{2}(g^\epsilon)^{il}\big(\partial_j g^\epsilon_{kl} +\partial_k g^\epsilon_{jl}-\partial_l g^\epsilon_{jk} \big),
\end{equation*}
or from the {\em Koszul formula}:
\begin{align}
2g^\epsilon(\na^\epsilon_XY,Z) =&\; Xg^\epsilon(Y,Z) + Yg^\epsilon(X,Z) - Zg^\epsilon(X,Y)\nonumber\\
&\; + g^\epsilon([X,Y],Z) - g^\epsilon([X,Z],Y) - g^\epsilon([Y,Z],X),
\end{align}
together with the compact embedding $W^{1,p}_{\text{loc}}(\R^n) \hookrightarrow L^\infty_{\text{loc}}(\R^n)$ for $p>n$.
As a consequence, since
$g_0\big(B^\epsilon(Y,\na^\epsilon_X Z),\eta \big) - g_0\big(B^\epsilon(X,\na^\epsilon_YZ),\eta\big)$
is a product of strongly $L^p_{\text{loc}}$ convergent and weakly $L^p_{\text{loc}}$ convergent (equivalently, $L^p_{\text{loc}}$ bounded)
tensors for $p>n$, we can pass to the limit,
{\it i.e.}, it converges to $g_0\big(\overline{B}(Y,\na_XZ),\eta\big)-g_0\big(\overline{B}(X,\na_YZ),\eta \big)$ modulo subsequences.

\smallskip
{\bf 3.}
Now let us invoke the analogous equations to the first formulation (Lemma \ref{lemma:4.1}):
\begin{eqnarray*}
&&{\rm div}\big(V^{(B^\epsilon)}_{Z,\eta}(X,Y)\big)
  = Yg_0(B^\epsilon(X,Z),\eta) - Xg_0(B^\epsilon(Y,Z),\eta), \\
&&{\rm div}\big(V^{(\na^{\perp,\epsilon})}_{\xi,\eta}(X,Y)\big)
 =-Yg_0 (\na^{\perp,\epsilon}_X\xi,\eta) +Xg_0 (\na^{\perp,\epsilon}_Y\xi,\eta).
\end{eqnarray*}
Again, by the uniform  bounds for $\{B^\epsilon\}$ and $\{\na^{\perp,\epsilon}\}$ in $L^p_{\text{loc}}$,
we know that $\big\{{\rm div}\big(V^{(B^\epsilon)}_{Z,\eta}(X,Y)\big)\big\}$
and $\big\{{\rm div}\big(V^{(\na^{\perp,\epsilon})}_{\xi,\eta}(X,Y)\big)\big\}$ are uniformly
bounded in $W^{-1,p}_{\text{loc}}$, where $p > n \geq 2$.
On the other hand, since
\begin{eqnarray*}
&& d\Omega^{(B^\epsilon)}_{Z,\eta}(X,Y)
 = Y g_0(B^\epsilon(X,Z),\eta)- Xg_0(B^\epsilon(Y,Z),\eta) + g_0( B^\epsilon([X,Y],Z),\eta),\\
&& d\Omega^{(\na^{\perp,\epsilon})}_{\xi,\eta}(X,Y) =  -Yg_0 (\na^{\perp,\epsilon}_X\xi,\eta)
+Xg_0 (\na^{\perp,\epsilon}_Y\xi,\eta) - g_0(\na^{\perp,\epsilon}_{[X,Y]}\xi,\eta),
\end{eqnarray*}
by substituting the left-hand sides into
\eqref{eqn g epsilon 1}--\eqref{eqn g epsilon 3}
and applying the H\"{o}lder inequality to the suitable terms,
we find that $\big\{{\rm div}\big(V^{(B^\epsilon)}_{Z,\eta}(X,Y)\big)\big\}$,
$\big\{{\rm div}\big(V^{(\na^{\perp,\epsilon})}_{\xi,\eta}(X,Y)\big)\big\}$,
$\big\{d\Omega^{(B^\epsilon)}_{Z,\eta}\big\}$, and $\big\{d\Omega^{(\na^{\perp,\epsilon})}_{\xi,\eta}\big\}$
are uniformly bounded in $L^{p/2}_{\text{loc}}$ so that, by interpolation,
they are pre-compact in $H^{-1}_{\text{loc}}$.

\smallskip
{\bf 4.} The arguments in Step 3 enable us to apply the geometric div-curl lemma (as in Theorem \ref{thm_main thm goemetric divcurl}),
which leads to the following convergence in the distributional sense:
\begin{eqnarray*}
&&\langle V^{(\na^{\perp,\epsilon})}_{\eta,\beta}(X,Y),\Omega^{(B^\epsilon)}_{Z,\beta}\rangle \longrightarrow \langle  V^{(\overline{\na^\perp})}_{\eta,\beta}(X,Y), \Omega^{(\overline{B})}_{Z,\beta}\rangle,\\
&&\langle V^{(B^\epsilon)}_{W,\eta}(X,Y),\Omega^{(B^\epsilon)}_{Z,\eta} \rangle  \longrightarrow \langle V^{(\overline{B})}_{W,\eta}(X,Y), \Omega^{(\overline{B})}_{Z,\eta} \rangle,\\
&&\langle V^{(\na^{\perp,\epsilon})}_{\eta,\beta}(X,Y), \Omega^{(\na^{\perp,\epsilon})}_{\xi,\beta} \rangle\longrightarrow  \langle V^{(\overline{\na^\perp})}_{\eta,\beta}(X,Y), \Omega^{(\overline{\na^\perp})}_{\xi,\beta}\rangle,\\
&&\langle V^{(B^\epsilon)}_{Z,\xi}(X,Y), \Omega^{(B^\epsilon)}_{Z,\eta} \rangle\longrightarrow \langle V^{(\overline{B})}_{Z,\xi}(X,Y), \Omega^{(\overline{B})}_{Z,\eta}\rangle,
\end{eqnarray*}
where $(\overline{B}, \overline{\na^\perp})$ are the subsequential weak limits of $\{(B^\epsilon, \na^{\perp,\epsilon})\}$.
Therefore, we conclude that $(\overline{B}, \overline{\na^\perp})$ satisfy the GCR equations in the distributional sense
with respect to metric $g$, {\it i.e.}, the subsequential strong $W^{1,p}_{\text{loc}}$ limit of $\{g^\epsilon\}$.
Furthermore, as $f^\epsilon$ converges to $\bar{f}$ weakly in $W^{2,p}_{\text{loc}}(M, \rnk)$,
Corollary \ref{cor_Weak continuity of isometric immersions} implies that $\bar{f}$ is an isometric immersion with respect
to the limiting metric $g$, and its second fundamental form and normal connection coincide
with $\overline{B}$ and $\overline{\na^\perp}$, respectively. Therefore, the proof is complete.
\end{proof}

We also refer the reader to some recent results on
the compactness of $W^{2,p}$ immersions
of $n$-dimensional manifolds for $p>n$ with appropriate {\it gauges}
for the case $n=2$ in Langer \cite{Langer}
and the higher dimensional case in Breuning \cite{Bre15}.

\appendix
\section{Proof of Theorem \ref{thm_general goemetric divcurl}}

In this appendix, we give a proof for the intrinsic
div-curl lemma, {\it i.e.}, Theorem \ref{thm_general goemetric divcurl},
on Riemannian manifolds.
Some arguments for the proof are motivated from the work by Robbin-Rogers-Temple \cite{RRT87},
in which a similar result has been established for the {\em local differential forms
on the flat Euclidean spaces $\real^n$}; also see \S 5 in Kozono-Yanagisawa \cite{KY13}.

\begin{proof}[Proof of Theorem {\rm \ref{thm_general goemetric divcurl}}]
We divide the proof into five steps.

\smallskip
{\bf 1.} It suffices to
prove Theorem \ref{thm_general goemetric divcurl} for closed, orientable manifolds.
Indeed, by a partition of unity argument and assigning an orientation to each chart,
it suffices to prove only for orientable manifolds.
Now, consider an orientable manifold $M$ which is not necessarily compact.
Fix any test function $\psi\in C^\infty_c(M)$, and then
choose $\phi\in C^\infty_c(M)$ such that $\phi|_{{\rm supp}(\psi)}\equiv 1$.
For any differential form $\alpha$ on a compact subset of $M$,
we can choose a compact oriented submanifold $\widetilde{M}\subset M$
with $supp(\phi)$ as its interior,
and denote by $\widetilde{\alpha}=\alpha\phi$ as the extension-by-zero of $\alpha$
outside $\widetilde{M}$.
Then it suffices to prove Theorem
\ref{thm_general goemetric divcurl}
for any compact orientable manifold $\widetilde{M}$,
since, if the assertion holds for $\langle {\omega}^\epsilon, {\tau}^\epsilon\rangle$
on $\widetilde{M}$,
then the theorem also holds for $\langle \omega^\epsilon, \tau^\epsilon \rangle$
on $M$ in the distributional sense
for any test function $\psi$ with $supp(\psi)\subset \widetilde{M}$.

\smallskip
{\bf 2.}
We can now assume that $M$ is a closed orientable manifold.
Thus, Theorem \ref{thm_hodge} holds
for the Laplace-Beltrami operator $\Delta=d\circ \delta + \delta \circ d$ on $M$.
Then
\begin{equation}\label{eqn_decomp one}
\omega^\epsilon = \pi_H \omega^\epsilon + d\delta \alpha^\epsilon + \delta d\alpha^\epsilon,
\end{equation}	
\begin{equation}\label{eqn_decomp two}
\tau^\epsilon = \pi_H \tau^\epsilon + d\delta \beta^\epsilon + \delta d\beta^\epsilon,
\end{equation}
where $\alpha^\epsilon:=G(\omega^\epsilon)$ and $\beta^\epsilon:=G(\tau^\epsilon)$.
As before, $ \pi_H: \pform \mapsto \text{Har}^q (M)$ denotes the canonical projection
onto the harmonic $q$--forms, and $G$ denotes the Green operator for $\Delta$.

\smallskip
{\bf 3.}  We now analyze the mode of convergence for each term
in the Hodge decompositions of $\{\omega^\epsilon\}$ and $\{\tau^\epsilon\}$,
{\it i.e.},  Eqs. \eqref{eqn_decomp one}--\eqref{eqn_decomp two}.
We will summarize the results in the following two claims:

\begin{claim*}
In {\rm Eq.} \eqref{eqn_decomp one},
$\{ \pi_H(\omega^\epsilon)\}$ and $\{\delta d\alpha^\epsilon\}$
are strongly convergent in $L^r(M;\bigwedge^qT^*M)$,
while the second term $\{d\delta \alpha^\epsilon\}$ is weakly convergent
in $L^r(M;\ptens)$.
Moreover, $\delta\alpha^\epsilon$ is strongly convergent
in $L^{r}(M;\bigwedge^{q+1}T^*M)$.
\end{claim*}

Indeed, $\{\omega^\epsilon\}$ is weakly convergent, hence bounded in $L^r$.
On the other hand, $ \pi_H \omega^\epsilon \in \harmp$,
which is finite-dimensional, by Theorem \ref{thm_hodge}. Then
$\{\pi_H \omega^\epsilon\}$ convergent strongly.

Next, we recall that $G$ commutes with any operator commuting with $\Delta$; hence, $Gd=dG$ and $G\delta=\delta G$.
Then we have
\begin{equation}
\delta d \alpha^\epsilon = \delta \big(G (d\omega^\epsilon)\big).
\end{equation}
By assumption, $\{d\omega^\epsilon\}$ is confined in a compact set
$K_d \Subset W^{-1,r}_{\rm loc}(M;\bigwedge^{q+1}T^*M)$
so that $\{\delta d\alpha^\epsilon\}$ is strongly convergent,
by the continuity of $\delta$ and $G$.

The middle term in decomposition \eqref{eqn_decomp one} is a weakly
convergent sequence, since
$d\delta \alpha^\epsilon = \omega^\epsilon - \pi(\omega^\epsilon)-\delta d\alpha^\epsilon$.
Moreover, we notice that $\delta \alpha^\epsilon = G\delta \omega^\epsilon$,
which is bounded in $W^{1,r}(M;\bigwedge^{q+1}T^*M)$,
again by the rigidity of $\delta$ and $G$.
Therefore, by the Rellich lemma, $\{\delta\alpha^\epsilon\}$ is pre-compact in $L^r(M;\bigwedge^{q+1}T^*M)$, hence converges strongly.
This implies the claim.

Applying the similar arguments to Eq. \eqref{eqn_decomp two},
we can also verify
\begin{claim*}
In {\rm Eq.} \eqref{eqn_decomp two},
$\{\pi_H \tau^\epsilon\}$ and $\{d\delta\beta^\epsilon\}$ are strongly convergent
in $L^s(M;\bigwedge^qT^*M)$, while $\{\delta d \beta^\epsilon\}$ is weakly convergent
in $L^s(M;\bigwedge^qT^*M)$.
Moreover, $\{d\beta^\epsilon\}$ is strongly convergent
in $W^{1,s}(M;\bigwedge^{q-1}T^*M)$.
\end{claim*}

{\bf 4.} We are now at the stage for proving the convergence
of $\langle \omega^\epsilon,\tau^\epsilon\rangle$ in the
distributional sense.
In view of the Hodge decomposition in
Eqs. \eqref{eqn_decomp one}--\eqref{eqn_decomp two}, we have
\begin{align}\label{long eqn}
\int_M \langle\omega^\epsilon,\tau^\epsilon\rangle \psi\, {\rm d}V_g
=&\;\Big\{\int_M \langle \pi_H \omega^\epsilon, \pi_H\tau^\epsilon \rangle  \psi\,{\rm d}V_g
  + \int_M \langle   \pi_H\omega^\epsilon,  d\delta \beta^\epsilon \rangle \psi\,{\rm d}V_g \nonumber\\
 &\,\,\,\,\,+\int_M \langle  \pi_H\omega^\epsilon, \delta d\beta^\epsilon \rangle \psi\,{\rm d}V_g
  + \int_M \langle  d\delta \alpha^\epsilon, \pi_H\tau^\epsilon \rangle \psi\,{\rm d}V_g \nonumber\\
 &\,\,\,\,\,+ \int_M \langle d\delta \alpha^\epsilon, d\delta \beta^\epsilon  \rangle \psi\,{\rm d}V_g
  + \int_M \langle \delta d\alpha^\epsilon, \pi_H\tau^\epsilon \rangle\psi\,{\rm d}V_g \nonumber\\
 &\,\,\,\,\, + \int_M \langle \delta d\alpha^\epsilon,d\delta \beta^\epsilon  \rangle \psi\,{\rm d}V_g
   + \int_M \langle \delta d\alpha^\epsilon, \delta d\beta^\epsilon \rangle \psi \,{\rm d}V_g\Big\}\nonumber\\
 &\,\,\,\,\,+\int_M  \langle d\delta \alpha^\epsilon, \delta d \beta^\epsilon \rangle\psi\,{\rm d}V_g,
\end{align}
where each of the first eight terms in the bracket
on the right-hand side is the pairing of one strongly convergent
sequence and one weakly (or strongly) convergent sequence. This can be deduced immediately from the two claims in Step 3.
Thus, as $\epsilon \rightarrow 0$, these eight terms pass to the desired limits, {\em i.e.},
\begin{equation*}
\int_M \langle \pi_H\omega^\epsilon, \pi_H\tau^\epsilon \rangle  \psi\,{\rm d}V_g
\rightarrow \int_M \langle \pi_H\bar{\omega}, \pi_H\bar{\tau} \rangle  \psi\,{\rm d}V_g
\qquad\mbox{as $\e\to 0$},
\end{equation*}
and similarly for the other seven terms.

\smallskip
{\bf 5.}
To deal with the last term on the right-hand side,
which is a pairing of two weakly convergent sequences, we integrate by parts to find that
\begin{align}\label{weak weak}
\int_M \langle d\delta\alpha^\epsilon, \delta d \beta^\epsilon \rangle\psi\,{\rm d}V_g
=& \int_M d\big( \psi\delta\alpha^\epsilon\wedge\ast\delta d \beta^\epsilon\big)\,{\rm  d}V_g
+(-1)^q \int_M \big(\delta\alpha^\epsilon\wedge d(\ast\delta d\beta^\epsilon)\psi\big)\,{\rm d}V_g \nonumber\\
&+(-1)^n \int_M \big(\delta\alpha^\epsilon\wedge\ast\delta d \beta^\epsilon\wedge d\psi\big)\,{\rm d} V_g,
\end{align}
where we have used the super-commutative property of the wedge product:
\begin{equation*}
d(\omega\wedge\tau) = d\omega \wedge \tau+(-1)^{\deg(\omega)}\omega\wedge d\tau.
\end{equation*}

Now, in Eq. \eqref{weak weak}, the first term on the right-hand side vanishes by the Stokes' theorem,
and the second term vanishes because $\delta=\pm \ast d \ast$, $\ast\ast=\pm 1$, and $d\circ d=0$ ($\ast$
denotes the Hodge star). For the remaining term, notice that, in Step 3, we have proved the boundedness
of $\{\delta \alpha^\epsilon\}$ in $W^{1,r}(M;\bigwedge^{q-1}T^*M)$
so that, by the Rellich lemma, it strongly converges in $L^r(M;\bigwedge^{q-1}T^*M)$;
in addition, since $\{\delta d\beta^\epsilon\}$ is weakly convergent in $L^s(M;\ptens)$ and $d\psi\in C^\infty(M;T^*M)$,
this term is the integral of the wedge product of one strongly convergent term and one weakly convergent term.
Thus, as before, we can pass the limits to obtain that
\begin{equation*}
\int_M \langle d\delta\alpha^\epsilon, \delta d \beta^\epsilon \rangle\psi\,{\rm d}V_g
\rightarrow \int_M \langle d\delta\alpha, \delta d \beta \rangle\psi\,{\rm d}V_g
\qquad\mbox{for any $\psi \in C^\infty_c(M)$}.
\end{equation*}
Therefore, in view of the decomposition in \eqref{long eqn},
we conclude
\begin{equation*}
\int_M \langle\omega^\epsilon, \tau^\epsilon\rangle\psi\,{\rm d}V_g
\longrightarrow \int_M \langle\overline{\omega}, \overline{\tau}\rangle \psi\,{\rm d}V_g
\qquad\mbox{as $\e \to 0$}.
\end{equation*}
This completes the proof.
\end{proof}

\bigskip
\noindent
{\bf Acknowledgement}.
The authors would like to thank Professor Deane Yang for insightful discussions
on the Cartan formalism and isometric immersions, and to Professors
John Ball, Jeanne Clelland, Lawrence Craig Evans,
Marshall Slemrod, and Dehua Wang for their helpful comments and interest.
Gui-Qiang Chen's research was supported in part by
the UK
Engineering and Physical Sciences Research Council Award
EP/E035027/1 and
EP/L015811/1, and the Royal Society--Wolfson Research Merit Award (UK).
Siran Li's research was supported in part by the UK EPSRC Science and Innovation award
to the Oxford Centre for Nonlinear PDE (EP/E035027/1), and the Keble Association Study Awards (UK).


\bigskip

\begin{thebibliography}{99}

\bibitem{Ball}
J.~M. Ball, {Convexity conditions and existence theorems in nonlinear elasticity},
\textit{Arch. Rational Mech. Anal.} \textbf{63} (1977), 337--403.

\bibitem{BM}
J.~M. Ball and F. Murat,
{Remarks on Chacon's biting lemma},
\textit{Proc. Amer. Math. Soc.} \textbf{107} (1989),
655--663.


\bibitem{Bre15}
P. Breuning, {Immersions with bounded second fundamental form},
\textit{J. Geom. Anal.} \textbf{25} (2015), 1344--1386.

\bibitem{BGY}
R.~L. Bryant,  P.~A. Griffiths, and D. Yang,  {Characteristics
and existence of isometric embeddings},
{\it Duke Math. J.} \textbf{50} (1983),
893--994.

\bibitem{BS}
Y. D. Burago and S. Z. Shefel, {The geometry of surfaces in
Euclidean spaces}, \textit{Geometry III}, pp. 1--85, Encyclopaedia Math. Sci.,
48, Burago and Zalggaller (Eds.), Springer-Verlag: Berlin, 1992.

\bibitem{Cartan}
E. Cartan, {Sur la possibilit\'e de plonger un espace Riemannian
donn\'e dans un espace Euclidien},
\textit{Ann. Soc. Pol. Math.} \textbf{6} (1927),
1--7.


\bibitem{CSW10-CMP}
G.-Q. Chen, M. Slemrod, and D. Wang,
{Isometric immersions and compensated compactness},
\textit{Commun. Math. Phys.} \textbf{294} (2010), 411--437.


\bibitem{CSW10}
G.-Q. Chen, M. Slemrod, and D. Wang,
{Weak continuity of the Gauss-Codazzi-Ricci system for isometric embedding},
\textit{Proc. Amer. Math. Soc.} \textbf{138} (2010), 1843--1852.



\bibitem{Cia08}
P.~G. Ciarlet, L. Gratie, and C. Mardare,
{A new approach to the fundamental theorem of surface theory},
\textit{Arch. Rational Mech. Anal.} \textbf{188} (2008), 457--473.


\bibitem{Clelland17}
J.~N. Clelland,
\textit{From Frenet to Cartan: The Method of Moving Frames},
AMS: Providence, 2017.


\bibitem{CLMS}
R. Coifman, P.-L. Lions, Y. Meyer, and S. Semmes,
{Compensated compactness and Hardy spaces},
\textit{J. Math. Pures Appl.} \textbf{72} (1993), 247--286.


\bibitem{CDM}
S. Conti, G. Dolzmann, and S. M\"{u}ller,
{The div-curl lemma for sequences whose divergence and curl are
compact in $W^{-1,1}$},
\textit{C.~R. Math. Acad. Sci. Paris}, \textbf{349} (2011), 175--178.


\bibitem{DoC92}
M.~P. Do Carmo, \textit{Riemannian Geometry},
Birkh\"{a}user: Boston, 1992.


\bibitem{Dacorogna}
B. Dacorogna, {\em Weak Continuity and Weak Lower Semicontinuity of
Nonlinear Functionals}, Springer-Verlag: Berlin, 1982.

\bibitem{Dafermos-book}
C. M. Dafermos, {\em Hyperbolic Conservation Laws in Continuum
Physics},  4th Ed., Springer-Verlag: Berlin, 2016.

\bibitem{DM2}
R.~J. DiPerna and A.~J. Majda,  {Concentrations in regularizations for 2D incompressible flow},
\textit{Comm. Pure Appl. Math.} \textbf{40} (1987), 301--345.


\bibitem{Eisenhart}
L. P. Eisenhart, {\em Riemannian Geometry}, Eighth Printing,
Princeton University Press: Princeton, NJ, 1997.

\bibitem{Eva90}
L.~C. Evans,
\textit{Weak Convergence Methods for Nonlinear Partial Differential Equations},
CBMS-RCSM, \textbf{74}, AMS: Providence, 1990.

\bibitem{Evans-Muller}
L.~C. Evans and S. M\"{u}ller,
Hardy spaces and the two-dimensional Euler equations with nonnegative vorticity,
\textit{J. Amer. Math. Soc.} \textbf{7} (1994), 199--219.


\bibitem{f}
M. Fabian, P. Habala, P. H\'{a}jek, V.~M. Santaluc\'{i}a, J. Pelant, and V. Zizler,
\textit{Functional Analysis and Infinite-Dimensional Geometry},
Springer: New York, 2001.


\bibitem{Goenner}
H. F. Goenner, {On the interdependency of the
Gauss-Codazzi-Ricci equations of local isometric embedding},
\textit{General
Relativity and Gravitation}, \textbf{8} (1977), 139--145.

\bibitem{Greene}
R. Greene, {\em  Isometric Embeddings of Riemannian and
Pseudo-Riemannian Manifolds},
Mem. Amer. Math. Soc. \textbf{97}, AMS:
Providence, RI, 1970.


\bibitem{Gunther89}
M. G\"unther,
{Zum Einbettungssatz von J. Nash},
\textit{Math. Nach.} \textbf{144} (1989), 165--187.

\bibitem{HanHon06}
Q. Han and J.-X. Hong,
\textit{Isometric Embedding of Riemannian Manifolds in Euclidean Spaces},
AMS: Providence, 2006.



\bibitem{Janet}
M. Janet, {Sur la possibilit\'e de plonger un espace Riemannian
donn\'e dans un espace Euclidien}, \textit{Ann. Soc. Pol. Math.} \textbf{5} (1926),
38--43.


\bibitem{Jos08}
J. Jost,
\textit{Riemannian Geometry and Geometric Analysis},
Springer: Berlin, 2008.

\bibitem{KY09}
H. Kozono and T. Yanagisawa,
{Global div-curl lemma on bounded domains in $\mathbb{R}^3$},
\textit{J. Funct. Anal.} \textbf{256} (2009), 3847--3859.

\bibitem{KY13}
H. Kozono and T. Yanagisawa,
{Global compensated compactness theorem for general
differential operators of first order},
\textit{Arch. Rational Mech. Anal.}
\textbf{207} (2013), 879--905.


\bibitem{Langer}
J. Langer,
{A compactness theorem for surfaces with $L_p$-bounded second fundamental form},
\textit{Math. Ann.} \textbf{270} (1985), 223--234.


\bibitem{Mar03}
S. Mardare,
{The fundamental theorem of surface theory for surfaces with little
regularity},
\textit{J. Elasticity}, \textbf{73} (2003), 251--290.

\bibitem{Mar05}
S. Mardare,
{On Pfaff systems with $L^p$ coefficients and their applications
in differential geometry},
\textit{J. Math. Pures Appl.} \textbf{84} (2005), 1659--1692.


\bibitem{Mar07}
S. Mardare,
{On systems of first order linear partial differential
equations with $L^p$ coefficients},
\textit{Adv. Diff. Eqs.} \textbf{12} (2007),
301--360.




\bibitem{Muller}
S. M\"{u}ller,
 Higher integrability of determinants and weak convergence in
 $L^1$,
 \textit{J. Reine Angew. Math.} \textbf{412} (1990), 20--34.





\bibitem{Mur78}
F. Murat,
{Compacit\'{e} par compensation},
\textit{Ann. Scuola Norm. Sup. Pisa Cl. Sci.}
\textbf{5} (1978), {489--507}.

\bibitem{Murat2}
F. Murat, {Compacit\'{e} par compensation. II}, In: {\it Proceedings
of the International Meeting on Recent Methods in Nonlinear Analysis
{\rm (}Rome, 1978{\rm )}},  pp. 245--256, Pitagora, Bologna, 1979.



\bibitem{NM}
G. Nakamura and Y. Maeda,  {Local isometric embedding problem of
Riemannian $3$-manifold into $\R\sp 6$}, \textit{Proc. Japan Acad. Ser. A:
Math. Sci.} \textbf{62} (1986), 257--259.

\bibitem{Nas54}
J. Nash,
{$\mathcal{C}^1$ isometric imbeddings},
\textit{Ann. of Math.}  \textbf{60} (1954), 383--396.

\bibitem{Nas56}
J. Nash,
{The imbedding problem for Riemannian manifolds},
\textit{Ann. of Math.} \textbf{63} (1956), 20--63.


\bibitem{Nir53}
L. Nirenberg,
{The Weyl and Minkowski problems in differential geometry in the large},
\textit{Comm. Pure Appl. Math.} \textbf{6} (1953), 337--394.


\bibitem{RRT87}
J. Robbin, R. Rogers, and B. Temple,
{On weak continuity and the Hodge decomposition},
\textit{Trans. Amer. Math. Soc.} \textbf{303} (1987), 609--618.

\bibitem{SY}
R. Schoen and S.-T. Yau,
\textit{Lectures on Differential Geometry},
International Press: Cambridge, MA, 1994.


\bibitem{schwartz}
Schwartz, L.:  \textit{Th\'{e}orie des Distributions},  2nd Edition, Hermann Press,  1966.

\bibitem{Spi79}
M. Spivak,  {\em A Comprehensive Introduction to Differential
Geometry},  Publish or Perish, Inc.: Boston, Mass., Vol. I-II, 1970;
Vol. III-V, 1975.


\bibitem{Szo08}
M. Szopos,
{An existence and uniqueness result for isometric immersions with little regularity},
\textit{Rev. Roumaine. Math. Pures Appl.} \textbf{53} (2008), 555--565.

\bibitem{Tar79}
L. Tartar,
{Compensated compactness and applications to partial differential equations},
In: \textit{Nonlinear Analysis and Mechanics: Heriot-Watt Symposium}, Vol. \textbf{4},
pp. 136--212, \textit{Res. Notes in Math.} \textbf{39},
Pitman: Boston, Mass.-London, 1979.

\bibitem{Tartar2}
L. Tartar, {The compensated compactness method applied to
systems of conservation laws}.  In: {\it Systems of Nonlinear Partial
Differential Equations {\rm (}Oxford, 1982{\rm )}},  pp. 263--285, NATO Adv. Sci.
Inst. Ser. C Math. Phys. Sci., 111, Reidel: Dordrecht, 1983.


\bibitem{Ten71}
K. Tenenblat,
{On isometric immersions of Riemannian manifolds},
\textit{Bull. Brazilian Math. Soc.} \textbf{2} (1971), 23--36.

\bibitem{War71}
F.~W. Warner,
\textit{Foundations of Differentiable Manifolds and Lie Groups},
Scott, Foresman and Co.: Glenview, Ill.-London, 1971.


\bibitem{Whitney57}
H. Whitney,
\textit{Geometric Integration Theory},
Princeton University Press: Princeton, 1957.

\bibitem{Yau}
S.-T. Yau: Review of geometry and analysis. In:
\textit{Mathematics: Frontiers and Perspectives},
pp. 353--401.
International
B. Mazur, AMS: Providence, RI, 2000.
Mathematics Union, Eds. V. Arnold, M. Atiyah, P. Lax,
\end{thebibliography}
\end{document}